\documentclass [10.5pt,oneside,a4paper,mathscr]{amsart}
\usepackage{geometry}
\newgeometry{margin=1.1in}

\usepackage{amsfonts, amsmath, amssymb, amsthm,mathtools, stmaryrd}
\usepackage[numbers]{natbib}  
\usepackage{verbatim}
\usepackage[normalem]{ulem}
\usepackage{tikz-cd}
\usepackage{enumitem}
\usepackage{array}

\usepackage{hyperref}

\theoremstyle{plain}
\newtheorem{thm}{Theorem}[section]
\newtheorem{cor}[thm]{Corollary}
\newtheorem{theorem}[thm]{Theorem}

\newtheorem{prop}[thm]{Proposition}

\numberwithin{equation}{section}

\newtheorem{conjecture}{Conjecture}  % Conjecture labeled by letters
\newtheorem{conj}[conjecture]{Conjecture}  % Conjecture labeled by letters

\theoremstyle{definition}
\newtheorem{defn}[thm]{Definition}
\newtheorem{example}[thm]{Example}
\newtheorem{lemma}[thm]{Lemma}
\newtheorem{rmk}[thm]{Remark}

\theoremstyle{remark}

\newcommand{\BA}{{\mathbb{A}}}

\newcommand{\BC}{{\mathbb{C}}}

\newcommand{\BE}{{\mathbb{E}}}

\newcommand{\BH}{{\mathbb{H}}}

\newcommand{\BL}{{\mathbb{L}}}

\newcommand{\BQ}{{\mathbb{Q}}}
\newcommand{\BR}{{\mathbb{R}}}

\newcommand{\BZ}{{\mathbb{Z}}}

\newcommand{\CA}{{\mathcal A}}

\newcommand{\CC}{{\mathcal C}}
\newcommand{\CD}{{\mathcal D}}

\newcommand{\CH}{{\mathcal H}}
\newcommand{\Chow}{{\mathrm{C}\mathrm{H}}}

\newcommand{\CL}{{\mathcal L}}

\newcommand{\CO}{{\mathcal O}}
\newcommand{\CP}{{\mathcal P}}

\newcommand{\CX}{{\mathcal X}}
\newcommand{\CY}{{\mathcal Y}}

\newcommand{\dual}{{^\vee}}

\newcommand{\Z}{\mathbb Z}

\newcommand{\e}{\mathbf e}

\newcommand{\pt}{{\mathsf{p}}}

\DeclareFontFamily{OT1}{rsfs}{}
\DeclareFontShape{OT1}{rsfs}{n}{it}{<-> rsfs10}{}
\DeclareMathAlphabet{\curly}{OT1}{rsfs}{n}{it}

\newcommand{\p}{\mathbb{P}}

\newcommand\Spec{\operatorname{Spec}}
\newcommand{\Mbar}{{\overline M}}
\newcommand{\vir}{{\text{vir}}}

\newcommand\ev{\operatorname{ev}}

\newcommand{\QMod}{\mathsf{QMod}}
\newcommand{\Mod}{\mathsf{Mod}}

\newcommand{\id}{\mathrm{id}}

\newcommand{\CAlev}{\CA^{\mathrm{lev}}}

\newcommand{\wt}{\mathsf{wt}}
\newcommand{\mult}{\mathsf{mult}}

\newcommand{\SL}{\mathrm{SL}}
\newcommand{\GL}{\mathrm{GL}}

\newcommand{\Mp}{\mathrm{Mp}}

\newcommand{\Sp}{\mathrm{Sp}}

\newcommand{\pr}{\mathrm{pr}}
\DeclareMathOperator{\Wt}{\mathsf{WT}}

\newcommand{\taut}{{\mathrm{taut}}}

\newcommand{\rk}{{\mathrm{rk}}}

\newcommand{\Lift}{{\mathrm{Lift}}}

\newcommand{\NL}{{\mathsf{NL}}}
\newcommand{\Hum}{{\mathsf{Hum}}}

\newcommand{\NS}{{\mathrm{NS}}}

\newcommand{\Rvert}{{R_{\mathsf{vert}}}}

\renewcommand{\div}{{\mathrm{div}}}

\begin{document}
	\baselineskip=14pt
	\title[Gromov-Witten theory of abelian varieties in families]{Gromov-Witten theory of abelian varieties in families\\ and modular forms}

	\author{Georg Oberdieck}
	\address{KTH Royal Institute of Technology, Department of Mathematics}
	\email{georgo@kth.se}
	\date{\today}

\begin{abstract}
This is the first paper in a series on the Gromov-Witten theory of the universal abelian variety over the moduli space of principally polarized abelian varieties of dimension $h$. We conjecture that the generating series of Gromov-Witten classes, when summed over the degree against the principal polarization, is a cycle-valued quasimodular form for $\SL_2(\mathbb{Z})$ and satisfy a holomorphic anomaly equation. These conjectures generalize the quasimodularity of the Gromov-Witten theory of elliptic curves to higher dimension and raise interesting questions regarding enumerative mirror symmetry for abelian varieties. In genus $1$ it specializes to a conjecture of Greer and Lian which was proven by Iribar Lopez after tautological projection. We also discuss a special family of abelian varieties with a conjectural relation to Siegel quasimodular forms of higher genus.

The main result of the paper is a proof of the conjectures in genus $2$ after tautological projection.
For that we introduce quotient Gromov-Witten invariants which are indexed by the characteristic polynomial of the curve class and are shown to determine all descendent Gromov-Witten invariants satisfying a degree conditions. We then give an explicit formula for all genus $2$ quotient invariants after tautological projection as the Shimura lift of the product of two Eisenstein series. The formula is based on a curious modular identity derived in a joint appendix with Brandon Williams.

\end{abstract}

	\maketitle

\setcounter{tocdepth}{1} 
\tableofcontents

\section{Introduction}
\subsection{Overview}
Enumerative mirror symmetry predicts that the Gromov-Witten invariants of a Calabi-Yau manifold
are strongly related to nearly-holomorphic sections of line bundles over the moduli space of complex structures on the mirror manifold \cite{BCOV}. This prediction was first confirmed by Dijkgraaf \cite{Dijkgraaf} for elliptic curves, who showed (with further generalizations by Okounkov-Pandharipande \cite{OP_VirCurves} and others \cite{HAE}) that the generating series of Gromov-Witten invariants of an elliptic curve $E$
\[
F_g^E(\tau_{k_1}(\gamma_1) \cdots \tau_{k_n}(\gamma_n)) = \sum_{d = 0}^{\infty} \langle \tau_{k_1}(\gamma_1) \cdots \tau_{k_n}(\gamma_n) \rangle^E_{g,d} q^d \] 
are quasimodular forms for $\SL_2(\BZ)$ and satisfy a holomorphic anomaly equation which is a recursive equation that determines the quasimodular form up to a purely modular form.
The Gromov-Witten invariants here are defined by integration over the virtual class of the moduli space of $n$-marked genus $g$ stable maps to an elliptic curve in degree $d$,
\[ \langle \tau_{k_1}(\gamma_1) \cdots \tau_{k_n}(\gamma_n) \rangle^E_{g,d}
=
\int_{[\Mbar_{g,n}(E,d)]^{\vir}} \prod_i \psi_i^{k_i} \ev_i^{\ast}(\gamma_i), \]
where $\psi_i$ are the cotangent line classes.

For an abelian variety $X$ of higher dimension the Gromov-Witten invariants vanish in all non-trivial curve classes due to the existence of non-trivial holomorphic $2$-forms, see \cite{BOPY} for a discussion. Heuristically, the abelian variety $X$ can be deformed to a complex torus without algebraic curves and so by deformation invariance the Gromov-Witten invariants vanish. Physically, one says that $X$ has extra supersymmetry.
A reduced Gromov-Witten theory of $X$ has been defined and studied intensively over the last $25$ years for both abelian surfaces and threefolds \cite{BL2, BOPY, CarBlomme2, OP_MCF, O_K3xP1, Blomme1, Blomme4, OP2} (the theory vanishes for abelian varieties of higher dimension \cite{BOPY}).
The key property of these invariants is a multiple cover formula that conjecturally expresses invariants in imprimitive curve classes through invariants for primitive class, see \cite[Appendix A]{O_NLGW} and \cite[Sec.7.6]{BOPY}.
Moreover, the generating series of primitive classes are known to be quasimodular \cite{BOPY}.
While the multiple cover formula has been proven recently for abelian surfaces in most cases \cite{CarBlomme2,OP_MCF}, the interpretation in terms of mirror symmetry remains mysterious and a uniform formulation of modularity is missing.

In this paper we propose a new viewpoint on the Gromov-Witten invariants of abelian varieties.
Instead of studying invariants of a fixed abelian variety itself, we study {\em family invariants}.
We consider the universal family of abelian varieties over the moduli space $\CA_h$ of principally polarized abelian varieties of dimension $h$ and define the invariants by pushforward of the standard (non-reduced!) virtual class to the base. The invariants hence become cohomology classes on $\CA_h$.
This is parallel to equivariant Gromov-Witten invariants of a variety with a torus action which take values in the cohomology of $BT$, where $T$ is the torus acting.
The Gromov-Witten invariants often vanish for a fixed abelian variety, so the family invariants restrict to zero on each point, but are non-trivial in general.
The cohomology of $\CA_1$ is generated by the fundamental class, hence in dimension $1$ one recovers the case of a fixed elliptic curve.
In higher dimension the Gromov-Witten invariants are naturally supported on the Noether-Lefschetz cycles in $\CA_h$ and yield connections to the theory of special cycles.

The main conjecture of this paper is that the family invariants of abelian varieties are in a uniform way cycle-valued quasimodular forms and satisfy a holomorphic anomaly equation, directly generalizing the case of elliptic curves.
This strongly suggests an interpretation for a family version of enumerative mirror symmetry relating the invariants of the family over $\CA_h$ to the mirror family in the sense of \cite{GLO} (which is a $h$-fold self-product of an elliptic curve).

The genus $1$ case of our conjecture recovers earlier independent predictions of Greer and Lian 
which is proven by Iribar Lopez after tautological projection.
The main result in this paper is the proof of quasimodularity and holomorphic anomaly equation in genus $2$ after tautological projection by an explicit computation. This yields strong evidence for the general framework.
In two sequel papers the case of family invariants of abelian surfaces and threefolds will be considered.

\subsection{Gromov-Witten classes}
Let $\CA_h$ be the moduli stack of principally polarized abelian varieties of dimension $h$, let
$\pi : \CX \to \CA_h$
be the universal abelian variety and let\footnote{All our Chow groups are taken with rational coefficients throughout. The rational class $\theta$ becomes an integral Chow class only on a cover of $\CA_h$, or when restricting to fibers.}
\[ \theta \in \Chow^1(\CX) \]
be a symmetrized principal polarization which is rigidified along the zero section.

The $\pi$-relative moduli space of stable maps
%Consider the $\pi$-relative moduli space
\[ \rho : \Mbar_{g,n}(\pi,d) \to \CA_h \]
parametrizes $n$-marked genus $g$ stable maps to fibers of $\pi$ with degree $d$ against the principal polarization.\footnote{In particular, the fiber of $\Mbar_{g,n}(\pi,d)$ over a point $[(X,\theta)] \in \CA_h$ is
$\bigsqcup_{\substack{\beta \in H_2(X,\BZ), \beta \cdot \theta = d}} \Mbar_{g,n}(X,\beta)$.
}
The moduli space has a virtual fundamental class
\[
[\Mbar_{g,n}(\pi,d) ]^{\vir} \in \Chow_{\mathrm{vd}}(\Mbar_{g,n}(\pi,d))
\]
which is
of dimension
\[ \mathrm{vd} = \dim(\CA_h) + (h-3)(1-g) + n. \]

Evaluating stable maps at the markings defines the evaluation morphism
\[ \ev : \Mbar_{g,n}(\pi,d) \to \CX^{\times n} := \underbrace{\CX \times_{\CA_h} \times \ldots \times_{\CA_h} \CX}_{n \text{ times }}.
\]

For $2g-2+n>0$ there is also a forgetful morphism to the moduli space of curves
\[ \tau : \Mbar_{g,n}(\pi,d) \to \Mbar_{g,n}. \]

For any $\Gamma \in \Chow^{\ast}(\CX^n)$ and $2g-2+n>0$ we define the Gromov-Witten classes:
\[ \CC_{g,d}(\Gamma) := (\tau \times \rho)_{\ast}( \ev^{\ast}(\Gamma) \cap [\Mbar_{g,n}(\pi,d) ]^{\vir} ) \in \Chow^{\ast}(\Mbar_{g,n} \times \CA_h). \]

\subsection{Main conjectures}
The $\BC$-algebra of quasimodular forms is the free polynomial ring 
\[ \QMod=\BC[G_2,G_4,G_6] \]
where for even $k \geq 2$ the $G_k(q)$ are the classical weight $k$ Eisenstein series defined by
\[ G_k(q) = - \frac{B_k}{2 k} + \sum_{n \geq 1} \sum_{d|n} d^{k-1} q^n. \]
While $G_2(q)$ is quasimodular, $G_k$ is modular for $k \geq 4$. The ring $\QMod$ is graded by weight.

The first main conjecture of this paper is the quasimodularity for the generating series of the Gromov-Witten classes. Define the series
\[
\CC_{g}(\Gamma) = \sum_{d \geq 0} \CC_{g,d}(\Gamma) q^d.
\]

\begin{conj} \label{conj:quasimodularity}
Every $\CC_{g}(\Gamma)$ is a cycle-valued quasimodular form:
\[
\CC_{g}(\Gamma) \in \Chow^{\ast}(\Mbar_{g,n} \times \CA_h) \otimes \QMod.
\]
\end{conj}
\vspace{7pt}

The second main conjecture will determine the dependence of the quasimodular form $\CC_g(\Gamma)$ on the non-modular generator $G_2$ through a holomorphic anomaly equation.
The key ingredient in the formula is the Lefschetz dual correspondence.
Consider the relative Lefschetz operator
\[ e_\theta : \Chow^{\ast}(\CX) \to \Chow^{\ast}(\CX), \quad \alpha \mapsto \alpha \cdot \theta \]
Consider the canonical Lefschetz dual correspondence
\[ f_\theta \in \Chow^{h-1}(\CX \times_{\CA_h} \CX) \]
as constructed by K\"unneman \cite{Kuennemann}, see Section \ref{subsec:motivic decompositions}.
The correspondence defines an action
\[ f_\theta : \Chow^{\ast}(\CX) \to \Chow^{\ast}(\CX). \]
We view here $f_{\theta}$ both as an operator on Chow as well as a cycle on $\CX^{\times 2}$.
More generally, we denote by $f_\theta^{(i)}$ the action of $f_\theta$ on $\Chow^{\ast}(\CX^n)$ in the $i$-th factor.

Consider the gluing map along the last two marked points,
\[ \iota : \Mbar_{g-1, n+2} \to \Mbar_{g,n}. \]
For any decompositions $g = g_1 + g_2$ and $\{ 1 , \ldots, n \} = S_1 \sqcup S_2$ let
\[
j : \Mbar_{g_1, S_1 \sqcup \{ n+1 \}}
\times \Mbar_{g_2, S_2 \sqcup \{ n+2 \}} \to \Mbar_{g, n}
\]
be the map which glues the points marked by $n+1$ and $n+2$,
where $\Mbar_{g_i, S_i}$ is the moduli space of stable
curves with markings in the set $S_i$.

\begin{conj}[Holomorphic anomaly equation] \label{conj:HAE}
Assuming that Conjecture~\ref{conj:quasimodularity} holds,
consider $\CC_g( \Gamma )$
as a polynomial in $G_2,G_4,G_6$ with coefficients cycles on $\Mbar_{g,n} \times \CA_h$.
Then we have:
\begin{align*}
		\frac{d}{dG_2} \CC_g( \Gamma )
		\ =\  & 
		(\iota \times \id_{\CA_h})_{\ast} \CC_{g-1}( \Gamma \boxtimes f_\theta )
		\\
		& + \sum_{\substack{ g= g_1 + g_2 \\ \{1,\ldots, n\} = S_1 \sqcup S_2}}
	(j \times \id_{\CA_h})_{\ast} \left( 
	\CC_{g_1,S_1 \sqcup \{ n+1 \} } \times \CC_{g_2, S_2 \sqcup \{ n+ 2 \}}(\Gamma \boxtimes f_{\theta})  \right) \\
		& - 2 \sum_{i=1}^{n} \psi_i \cdot \CC_g( f_{\theta}^{(i)}\Gamma ).
	\end{align*}
\end{conj}

We refer here to Section~\ref{subsec:defn of product GW classes} for the definition
of the product cycle $\CC_{g_1,I_1} \times \CC_{g_2,I_2}$.
The middle sum runs only over terms where all classes are defined, i.e. such that $2 g_i -2 + |S_i| + 1 > 0$ for $i=1,2$.

The holomorphic anomaly equation is very useful because it restricts strongly the quasimodular forms which can show up. It also predicts their weight
 as follows.
For every $m \in \BZ$ let 
\[ [m] := [m]^{\times n} : \CX^{\times n} \to \CX^{\times n} \]
be the multiplication by $m$ on the fibers of the abelian scheme $\CX^{\times n} \to \CA_h$.
By work of Deninger-Murre \cite{DM} the induced action of $\BZ$ on $\Chow^{\ast}(\CX^n)$ is simultaneously diagonalizable:
\begin{equation} \Chow^{\ast}(\CX^n) = \bigoplus_{i=0}^{2hn} \Chow^{\ast}(\CX^n)_i \label{motivic decomposition} \end{equation}
where\footnote{This is non-standard notation, compare \cite[Section 2.17]{DM}.}
\[ \Chow^{\ast}(\CX^n)_i = \{ \Gamma \in \Chow^{\ast}(\CX^n) | [m]^{\ast}(\Gamma) = m^i \Gamma \}. \]
For $\Gamma \in \Chow^{\ast}(\CX^n)_i$ we write $i = \mult(\Gamma)$ and call $\Gamma$ an eigenvector of multiplicity $i$.

\begin{cor} \label{cor:weight statement as a consequence of HAE}
	Assume Conjecture~\ref{conj:HAE} holds.
If $\Gamma \in \Chow^{\ast}(\CX^n)$ is an eigenvector, then $\CC_g(\Gamma)$ is of weight $h(2g-2)+\mult(\Gamma)$:
\[ \CC_g(\Gamma) \in \Chow^{\ast}(\Mbar_{g,n} \times \CA_h) \otimes  \QMod_{h(2g-2)+\mult(\Gamma)} \]
\end{cor}

\subsection{Earlier work}
The presented conjectures were motivated by two main sources.
The first is the case of elliptic curves.
Indeed, in dimension $h=1$ the moduli stack $\CA_{1}$ is (a finite quotient of) $\BA_1$. Thus the conjectures reduce to the quasimodularity and holomorphic anomaly equation of the Gromov-Witten classes of a fixed elliptic curve. Both of these statements were then proven in cohomology in \cite{HAE}. We expect that these results also hold in Chow since the cycle $\Gamma$ is defined over the whole moduli space and so the degeneration methods of \cite{HAE} should apply.

The second motivation comes from work of Iribar Lopez on Noether-Lefschetz cycles. Let
\[ \sigma = \frac{\theta^h}{h!}. \]
be the class of the zero section.
It is immediate that for $h>0$ one has
%The class $\CC_{1}(\sigma) \in  \Chow^{h-1}(\CA_h)$ for $h>0$ is easily computed to be
\[ \CC_{1}(\sigma) = \frac{1}{24} (-1)^h \lambda_{h-1} + \sum_{d \geq 1} \sum_{\delta|d} \sigma(d/\delta) \NL_{h,\delta} q^d. \]
where $\NL_{h,\delta}$ are the class of the loci of abelian varieties $(X,\theta)$ containing an elliptic curve $E \subset X$ with $\theta|_{E}$ of degree $\delta$, see Section~\ref{subsec:NL loci}.
Conjecture~\ref{conj:quasimodularity} and~\ref{conj:HAE} imply
\begin{equation} \CC_{1}(\sigma) \in \Mod_{2h} \otimes \Chow^{h-1}(\CA_h) \label{Greer-Lian} \end{equation}
which matches precisely 
a conjecture of Greer-Lian \cite[Conjecture 1]{GreerLian}, see also the refinement \cite[Conjecture 1]{AitorNL}.

The main evidence for this conjecture is the following.
Let $\BE_h \to \BA_h$ the Hodge bundle which is of rank $h$ and consider its Chern classes
\[ \lambda_i = c_i(\BE_h). \]
The {\em tautological ring} of $\CA_h$ is the subring of $\Chow^{\ast}(\CA_h)$ generated by the $\lambda_i$.
By a result of van der Geer we have
\[ R^{\ast}(\CA_h) = \frac{\BQ[\lambda_1,\ldots,\lambda_h]}{(\lambda_h, c(\BE_h) c(\BE_h^{\vee}) - 1)}. \]
By work of Canning-Molcho-Oprea-Pandharipande \cite{CMOP} there is a tautological projection operator:
\[ \taut : \Chow^{\ast}(\CA_h) \to R^{\ast}(\CA_h).
\]
The map $\taut$ is obtained from taking degrees against the $\lambda$-classes on a compatification, and is similar to taking the degree of a class in projective space.

\begin{thm}[{Iribar Lopez \cite{AitorNL}}]
	\[ \taut\ \CC_1(\sigma) = 
	\frac{(-1)^{h-1}h }{6 B_{2h}} G_{2h}(q) \lambda_{h-1}. 
\]
\end{thm}

\subsection{Tautological insertions}
To make computations manageable we will specialize to tautological insertions and integrate against tautological classes on the moduli space of curves.
Let $\mathsf{m} : \CX \times \CX \to \CX$ be the addition map and define the Mumford class
\[ \eta := \frac{1}{2} (\mathsf{m}^{\ast}(\theta) - \pr_1^{\ast} \theta - \pr_2^{\ast}\theta). \]
Let $\eta_{ij}$ be the pullback along the $(i,j)$-th projection of $\CX^{\times n}$ and let $\eta_{ii} = \pr_i^{\ast}(\theta)$. 
Following \cite{10authors}
the vertical tautological ring $R_{\mathsf{vert}}(\CX^{\times n})$ is the subring of $\Chow^{\ast}(\CX^n)$ generated by the $\eta_{ij}$,
\[ R_{\mathsf{vert}}(\CX^{\times n}) = \BQ[\eta_{ij} | 1 \leq i,j \leq n ] \subset \Chow^{\ast}(\CX^{\times n}). \]

For $k_1,\ldots,k_n \geq 0$ and $\Gamma \in R_{\mathrm{vert}}^{\ast}(\CX^{\times n})$ we define the {\em descendent Gromov-Witten invariants} by
\begin{equation} \label{descendent GW invariants}
\CC_{g,d}(\tau_{k_1\ldots k_n}(\Gamma)) := 
\rho_{\ast}\left( \ev^{\ast}(\Gamma) \prod_{i=1}^{n} \psi_i^{k_i} \cap  [\Mbar_{g,n}(\pi,d) ]^{\vir} \right).
\end{equation}
By definition the Gromov-Witten invariants take values in the Chow ring of $\CA_h$:
\[
\CC_{g,d}(\tau_{k_1 \ldots k_n}(\Gamma))  \in \Chow^{\ast}(\CA_h).
\]

By a simple dimension count all Gromov-Witten invariants vanish in cohomology if $g \geq h \geq 4$ and in Chow for $g \geq h \geq 7$. However one can prove a stronger vanishing after tautological projection:

\begin{prop} \label{thm:vanishing}
If $g \geq 4$ and $h \geq 4$, then 
$\taut\ \CC_{g,d}(\tau_{k_1 \ldots k_n}(\Gamma))$ vanishes.
\end{prop}

The proof uses that a map $f : C \to X$ from a genus $g$ curve to an abelian variety of dimension $h$ must be supported on a translate of an abelian subvariety of dimension $\leq g$. This will lead to a key decomposition of the Gromov-Witten classes as pushforward of explicit classes from Noether-Lefschetz loci, see Section~\ref{sec:NL and GW theory}.

We also have a strong constraint on the degrees of the insertions:
\begin{prop} \label{prop:dim constraint}
Let $h \geq 3$ and let $\Gamma \in \Rvert(\CX^{\times n})$ be homogeneous.
\begin{itemize}
	\item[(i)] If $g=1$, then $\taut\ \CC_g(\tau_{k_1 \ldots k_n}(\Gamma))$ is non-zero only if $\sum_i k_i + \deg(\Gamma) = h + n - 1$.
	\item[(ii)] If $g \geq 2$, then $\taut\ \CC_g(\tau_{k_1 \ldots k_n}(\Gamma))$ is non-zero only if $\sum_i k_i + \deg(\Gamma) = h + n$. 
\end{itemize}
\end{prop}

We expect Propositions~\ref{thm:vanishing} and~\ref{prop:dim constraint} to hold also without applying tautological projection.

\subsection{The genus  $2$ formula}
The main result of this paper is the analogue of the Iribar Lopez formula in genus $2$.
Form the generating series
\[
\CC_{g}\left( \tau_{k_1 \ldots k_n}(\Gamma)  \right)
=
\sum_{d=0}^{\infty} 
\CC_{g,d}(\tau_{k_1 \ldots k_n}(\Gamma))
q^d
%\left\langle \tau_{k_1 \ldots k_n}(\Gamma) \right\rangle^{\pi}_{g,d} q^d.
\]
We will consider the classes
\[
\CC_2(\tau_1(\sigma)) \in \Chow^{2h-3}(\CA_h)
\]
which is by our conjectures expected to be a cycle-valued quasimodular form of weight $4h$, and modular for $h \geq 3$. We will check this prediction here after tautological projection.
%The $\CC_{2,d}(\tau_1(\sigma))$ is a certain linear combination of genus $2$ Noether-Lefschetz classes 
%and thus 

Let $E_{5/2}$ be the Eisenstein series of weight $5/2$ for the Weil representation on the lattice $A_1$, normalized so that the $e_{0}$-coefficient has constant term $1$. 
The first Fourier coefficients read
\begin{equation} \label{E52_1}
	\begin{aligned}
		E_{5/2}^{(1)}(\tau)
		={}&
		\left(
		1 - 70q - 120q^2 - 240q^3 - 550q^4 - 528q^5
		- 720q^6 + \ldots 
		\right) e_{0}
		\\
		&+
		\left(
		-10q^{1/4}
		-48q^{5/4}
		-250q^{9/4}
		-240q^{13/4}
		-480q^{17/4}
		+ \ldots
		\right) e_{1/2}
	\end{aligned}
\end{equation}
and an explicit expression is recalled in Example~\ref{example eisenstein series}.
By a famous result of van der Geer \cite[Thm 8.1]{vdG_SiegelModularThreefold} these Fourier coefficients are exactly minus the degrees of the Humbert divisors (or Noether-Lefschetz divisors) on the moduli space of principally polarized abelian surfaces. 

For a power series
\[ F = \sum_{\gamma \in \frac{1}{2}\BZ/\BZ} \sum_{\substack{n \in \BQ \\ n - \gamma^2 \in \BZ}} c(n,\gamma) q^n e_{\gamma} \]
define the (formal) {\em Shimura lift} of $F$ in weight $k$ by
\begin{align*}
	\Lift_k(F) 
	& = -\frac{c(0,0) B_k}{2k} + \sum_{d \geq 1} \sum_{\ell | d} \ell^{k-1} c\left(\frac{d^2}{4\ell^2} , \frac{d}{2 \ell} \right)  q^{d}.
	%	& = -\frac{c(0,0) B_k}{2k} + \sum_{n \geq 1} \sum_{\ell|n} \ell^{k-1} c\left( \frac{n^2}{4 \delta \ell^2}, \frac{n}{2 \delta \ell} \right) q^n.
\end{align*}
The map $\Lift_k$ is known to take (quasi)modular forms for the Weil representation on $A_1$ of weight $k+\frac{1}{2}$ to (quasi)modular forms for $\SL_2(\BZ)$ of weight $2k$.

\begin{thm} \label{thm:genus 2 formula}
	For every $h \geq 2$ we have
	\[ \taut\ \CC_2(\tau_1(\sigma)) 
	=
	-\frac{1}{360} \frac{(h-1) h}{|B_{2h-2}| |B_{2h}|}
	\Lift_{2h}( G_{2h-2}(q) E_{5/2}^{(1)}) \lambda_{h-1} \lambda_{h-2}.
	\]
	In particular, $\taut\ \CC_2(\tau_1(\sigma))$ is a modular form of weight $4h$ for $h \geq 3$, and is quasimodular if $h=2$.
\end{thm}

For the proof we write the Gromov-Witten cycles as linear combinations of genus $1$ and $2$ Noether-Lefschetz cycles in the moduli space of $\CA_h$.\footnote{In fact, the linear combination may be seen as providing the correct "weights" of the genus $2$ Noether-Lefschetz cycles to make their generating series modular.} Their tautological projections are known by work of \cite{AitorNL}.
The result is a fairly elaborate generating series involving the degrees of the Humbert cycles on the moduli spaces of abelian surfaces of polarization type $(1,\delta)$, which are encoded in the Eisenstein series $E_{5/2}^{(\delta)}$ for the lattice $A_1(\delta) = (2 \delta)$.
By a curious modular identity, obtained in joint work with B. Williams in the appendix, this yields the above formula.

\subsection{Quotient invariants and genus $2$ invariants in general}
An alternative way to define Gromov-Witten invariants of abelian varieties is to consider stable maps up to translation. This will yield quasimodularity results in genus $2$ for a bigger class of insertions.
Since it is cumbersome to use insertions up to translation, we will instead capture the full theory by refining the notion of degree of a curve class.
For $(X,\theta)$ a principally polarized abelian variety and curve class $\beta \in H_2(X,\BZ)$
the {\em characteristic polynomial} of $\beta$ is defined to be
\[ \chi_{\beta}(t) = \int_{X} \frac{(t \theta - v_{\beta})^{h}}{h!} \]
where $v_{\beta} \in H^2(X,\BZ)$ is the image of $\beta$ under the chain of isomorphism
\[ H_2(X,\BZ) \cong H^2(\widehat{X},\BZ) \xrightarrow{\varphi_{\theta}^{\ast}} H^2(X,\BZ) \]
where $\widehat{X}$ is the dual abelian variety and $\varphi_{\theta} : X \to \widehat{X}$ is the isomorphism induced by the polarization.
The polynomial $\chi_{\beta}$ is monic, has integral coefficients, its $t^{h-1}$-coefficient is the negative of the degree of $\beta$, and if $\beta$ is effective then $\chi_{\beta}$ has non-negative real roots. 

Let $\chi(t)$ be a degree $h$ polynomial. Let
\[ \Mbar_{g,n}(\pi,\chi) \]
be the moduli space of stable maps $f : C \to X$ to fibers of $\pi$ such that the curve class $\beta=f_{\ast}[C]$ has characteristic polynomial $\chi_{\beta} = \chi$. This defines a decomposition into open and closed components
\[ \Mbar_{g,n}(\pi,d) = \bigsqcup_{\substack{\chi = t^h - d t^{h-1} + \ldots}} \Mbar_{g,n}(\pi,\chi). \]
The translation action of $\CX$ on itself induces an action on $\Mbar_{g,n}(\CX,\chi)$.
If $g \geq 2$ the action has finite stabilizer and the quotient
$\Mbar_{g,n}(\pi,\chi)/\CX$ is a proper DM stack.
In Section~\ref{subsec:vanishing reln 1} equation \eqref{virtual class quotient} we define a virtual class on the moduli space.
The {\em quotient Gromov-Witten invariants} of the family $\pi$ are then defined by
\[
N_{g,\chi}^{\CA_h} := \rho_{\ast} [ \Mbar_{g,n}(\pi,\chi)/\CX ]^{\vir} \ \in \Chow^{gh - (3g-3)}(\CA_h).
\]

It is not difficult to show that all quotient invariants are determined by the descendent Gromov-Witten invariants as defined in \eqref{descendent GW invariants}, see Section~\ref{sec:quotient GW invariants}. 
Conversely we have the following reconstruction result which shows that the descendent and quotient Gromov-Witten invariants essentially contain the same information\footnote{At least under the expected degree condition on the insertion of a descendent Gromov-Witten invariant, see Proposition \ref{prop:dim constraint}.}:

\begin{prop}
Let $g \geq 2$, let $k_1,\ldots,k_n \geq 0$ and let $\Gamma \in \Rvert(\CX^{\times n})$ be a tautological class satisfying the expected degree constraint $\sum_i k_i + \deg(\Gamma) = h + n$ (see Proposition \ref{prop:dim constraint}).
Then the Gromov-Witten invariant
$\CC_{g,d}(\tau_{k_1 \ldots k_n}(\Gamma))$
can be effectively reconstructed from the quotient invariants $N_{g,\chi}^{\CA_h}$
through the intersection theory of $\psi$-classes of the moduli space of curves
and the invariant theory of the symplectic group.
\end{prop} 

For explicit formulas in particular cases (e.g. the primary case) we refer to Section~\ref{sec:quotient GW invariants}.

The advantage of the quotient invariants is that they are easier to work with in computations.

Quasimodularity and holomorphic anomaly equation
for descendent Gromov-Witten invariants implies parallel constraints for the quotient invariants.
For every polynomial $P$ in formal variables $e_1,\ldots,e_h$
define the generating series of quotient invariants by
\[ F_g^{\CA_h}(P(e_1,\ldots,e_h)) = \sum_{\chi} P(e_1,\ldots,e_h) N_{g,\chi}^{\CA_h} q^{e_1} \]
where we sum over all characteristic polynomials $\chi = t^h - e_1 t^{h-1} + \ldots + e_h (-1)^h$.

\begin{thm} \label{thm:hae for quotient}
Let $g \geq 2$.
Assume that Conjectures~\ref{conj:quasimodularity} and Conjecture~\ref{conj:HAE} hold.
Then
\[ F_{g}^{\CA_h}(e_{r_1} \cdots e_{r_s}) \in \QMod_{2gh + 2 \sum_{i=1}^{h} r_i} \otimes \Chow^{\ast}(\CA_h). \]
If moreover $h \geq 3$, then:
	\[
	\frac{d}{d G_2}F_g^{\CA_h}( P )
	=
	2 \delta_{h=3} F_{g-1}^{\CA_h}( e_2 \cdot P ) 
	+ F_g^{\CA_h}( \CL_{g,h} P )
	\]
	where $\CL_{g,h}$ is the following differential operator on the ring $\BC[e_1,\ldots,e_h]$, with the convention that $e_0=1$ and $e_k=0$ for $k>h$,
	\begin{equation*}
		\begin{aligned}
			\CL_{g,h}
			={}&-2\sum_{i,j=1}^h
			\left(\sum_{a=0}^{\min(i,j)-1}(i+j-1-2a)e_a e_{i+j-1-a}\right)
			\frac{d}{d e_i} \frac{d}{d e_j} \\
			&\quad -4\sum_{i=1}^h (g-i+1)(h-i+1)e_{i-1} \frac{d}{de_i} .
		\end{aligned}
	\end{equation*}
\end{thm}

Our main result towards this conjecture is the check in genus $2$.
We give explicit formulas for all the genus $2$ quotient invariants
in Corollary~\ref{cor:genus 2 quotient invariants}
which yields the following:

\begin{thm} \label{thm:genus 2 quotient quasimodularity}
After tautological projection all genus $2$ quotient invariants $F_{2}^{\CA_h}(P)$ are quasimodular forms of weight $4h+2a+4b$ and satisfy the holomorphic anomaly equation of Theorem~\ref{thm:hae for quotient}. Concretely,
\[ \taut \ F_{2}^{\CA_h}(e_1^{a} e_2^b)
=
-\frac{1}{360} \frac{(h-1) h}{|B_{2h-2}| |B_{2h}|}
D_{\tau}^a \Lift_{2(h+b)}\left( D_{\tau}^b(G_{2h-2}) E_{5/2}^{(1)} \right) \lambda_{h-1} \lambda_{h-2}
\]
where $D_{\tau} = q \frac{d}{dq}$.
\end{thm}

As a corollary we obtain the quasimodularity of all genus $2$ Gromov-Witten invariants:
\begin{cor} \label{cor:genus 2}
For any $\Gamma \in \Rvert(\CX^{\times n})$ and $k_1,\ldots,k_n \geq 0$ we have
\[ \taut\ \CC_{2}^{\CA_h}\left( \tau_{k_1 \ldots k_n}(\Gamma)  \right) \in \QMod_{2h + 2 \deg(\Gamma)} \otimes \BC\lambda_{h-1} \lambda_{h-2} . \]
\end{cor}

\subsection{Upcoming work}
In genus $g=2$ we have obtained convincing results.\footnote{The genus $1$ invariants are essentially all determined by the single series $\CC_1(\tau_0(\sigma))$ and thus not as rich.}
In a follow-up paper we will treat the case of abelian surfaces in arbitrary genus. The main result is as follows:

\begin{thm}[{\cite{O_Ab2}}] \label{thm:abelian surfaces}
For $\Gamma$ a vertical tautological class the series $\CC_{g}^{\CA_2}(\tau_{k_1 \ldots k_n}(\Gamma))$ is the Shimura lift of the product of the generating series of reduced Gromov-Witten invariants of an abelian surface in primitive classes with $E_{5/2}^{(1)}(q)$.
In particular, $\CC_{g}^{\CA_2}(\tau_{k_1 \ldots k_n}(\Gamma))$ is a quasimodular form.
\end{thm}

The result is conditional on the multiple cover formula for the reduced Gromov-Witten invariants of an abelian surface
recently proven by Blomme-Carocci \cite{CarBlomme2} (in most cases) and proven a second time in full generality but relying on a conjectural GW/PT correspondence in \cite{OP_MCF}.
Theorem~\ref{thm:abelian surfaces} can be viewed as a modular interpretation for the multiple cover formula for abelian surface. The multiple cover formula is precisely what is needed to express the invariants of the (universal) abelian surface as the Shimura lift of the reduced series.

As a basic example, let $\mu_k = c_k(\BE_{\Mbar_{g,n}})$ be the Chern classes of the Hodge bundle on the moduli space of curves.\footnote{We denoted the Chern classes of the Hodge bundle on $\CA_h$ by $\lambda_i$, so the usual name for the $\lambda$-classes is occupied.}
Define the Gromov-Witten invariants twisted by the $\mu$-classes:
\[ \CC_{g}^{\CA_2}( \tau_{1}(\sigma) \mu_{g-2}) :=
\sum_{d \geq 0}
\rho_{\ast}\left( \ev_1^{\ast}(\sigma) \psi_1 \mu_{g-2} \cap [\Mbar_{g,1}(\pi,d)]^{\vir} \right) q^d. \]
Then one finds the explicit evaluation (see \cite{O_Ab2})
\[
\CC_{g}^{\CA_2}( \tau_{1}(\sigma) \mu_{g-2})
=
\frac{-2}{(2g-3)!} \Lift( G_{2g-2}(q) E_{5/2}^{(1)} ) \lambda_1. \]
It is interesting (but not unexpected) that up to scalar
this is the $q$-series appearing in Theorem~\ref{thm:genus 2 formula}.

The case of abelian threefolds is special since here the quotient Gromov-Witten invariants are defined and non-zero for arbitrary genus. The invariants of this Calabi-Yau threefold family 
will be studied in forthcoming joint work with N. Ludwig.

\subsection{Mirror symmetry}
The modularity prediction of Conjecture~\ref{conj:quasimodularity} can be made plausible through mirror symmetry.
Classically mirror symmetry says that generating series $F_g$ of genus $g$ Gromov-Witten invariants of a Calabi-Yau manifold $X$
are the asymptotic expansions at a cusp of "almost-holomorphic" sections of the $(2g-2)$-th power of the Hodge line bundle $\CL$ on the moduli space of the mirror Calabi-Yau $X^{\vee}$.
For an elliptic curve $E$ the mirror $E^{\vee}$ is again an elliptic curve which has moduli space $\CA_1 = [\BH_1/\SL_2(\BZ)]$ and the Hodge bundle $\BE_1$. The prediction gives that $F_g$ are quasimodular forms of weight $2g-2$, see \cite{Dijkgraaf}.
In dimension greater than one, the Gromov-Witten invariants of an abelian variety are almost always zero, so this prediction holds but for trivial reasons.

To gain insight we can instead consider Hodge-theoretic mirror symmetry for families of polarized abelian varieties as worked out in \cite{GLO}. The idea is that mirror symmetry interchanges the role of the Mumford-Tate group (which acts horizontally) with the action of the Neron-Severi part of the Loijenga-Lunts-Verbitsky Lie algebra (which acts vertically). This is similar in spirit to mirror symmetry for K3 surfaces.
It has been argued in \cite[Proposition 9.6.1]{GLO} that the family of polarized abelian varieties $\CX_h \to \CA_h$ is mirror to the family of $h$-fold powers of elliptic curves:
\[ \pi : \CX_h \to \CA_h \quad \overset{\text{mirror}}{\Longleftrightarrow} \quad \widehat{\pi} : \CX_1^{h} \to \CA_1. \]
To give some intuition for this, recall that mirror symmetry interchanges complex K\"ahler and complex deformation parameters of the families.
The K\"ahler parameter of a generic element of $\pi$ is
\[ \tau \theta, \quad \theta \in \BH_1 = \BR + i \BR_{>0}. \]
The complex deformation parameters of the family are given by $\mathrm{Sp}(2h,\BZ)$-equivalence classes of an element $Z \in \BH_h$ in the genus $h$ Siegel upper half space $\BH_h$ of symmetric $h \times h$-matrices with positive definite imaginary part.
On the mirror side, for $E$ a generic elliptic curve, $\mathrm{NS}(E^h)$ can be naturally identified with symmetric half-integral $h \times h$-matrices with integral diagonal entries by
\[
%M_{h \times h, \BZ}^{\mathrm{sym}} \cong \NS(E^h), \quad 
( m_{ij} )_{i,j} \mapsto \sum_{i,j} m_{ij} \eta_{ij}.
\] 
It follows that K\"ahler parameters are now indexed by elements in $\BH_h$, while deformation parameters are elements in $[\BH_1/\SL_2(\BZ)]$.
The Hodge bundle of the mirror family $\widehat{\pi}$ is $\CL = \det(\BE_1^{\oplus h}) = \BE_1^{\otimes h}$.
Moreover, as discussed in \cite{GLO} the fiberwise $\SL_2(\BZ)$ action on the Chow group of $\pi : \CX \to \CA_h$ %as constructed in \cite{Beauville} (based on earlier work of Mukai) 
should exactly correspond to the monodromy of the mirror family and hence give rise to the $\SL_2(\BZ)$-invariance (=modularity) of the section of $\CL$.
We see that (ignoring insertions) the proposed quasimodularity of $\CC_{g}^{\CA_h}$ of weight $h(2g-2)$ matches perfectly with the expectations that they should be related to asymptotic expansion of an almost-holomorphic section of $\CL^{2g-2}$.
It would be very interesting to understand better this relationship and in particular, to clarify the role that the cycle part of $\CC_{g}^{\CA_h}$ plays under mirror symmetry.

Fully aware of the author's ignorance in these matters, this viewpoint nevertheless invites natural generalizations of the quasimodularity conjecture for more special families.
We highlight a beautiful case. Consider the family of $k$-fold products of principally polarized abelian varieties of genus $h$:
\[ \CX_h^{\times k} \to \CA_h. \]
The mirror family is then
\[ \CX_k^h \to \CA_k. \]
The degree of a curve class $\beta \in H_2(X^k,\BZ)$ is the symmetric half-integral $k \times k$-matrix $d$ given by $d_{ij} = \eta_{ij} \cdot \beta$. 
For $\beta$ to be effective (or zero) we need $d$ to be positive-semidefinite.
We write
$\CC_{g,d}^{\CX^k/\CA_h}$ for the Gromov-Witten classes of degree $d$ and form the generating series
\[ \CC_{g}^{\CX^k/\CA_h}(\Gamma) = \sum_{d \geq 0} \CC_{g,d}^{\CX^k/\CA_h}(\Gamma) \exp(2 \pi i \mathrm{Tr}(d \cdot Z)) \]
where $Z \in \BH_k$
and
the sum is over positive-semidefinite symmetric half-integral $k \times k$ matrices $d$ with integral diagonals.

\begin{conj}
Each $\CC_{g}^{\CX^k/\CA_h}(\Gamma)$ is a cycle-valued Siegel-quasimodular form of genus $k$ and satisfies a holomorphic anomaly equation.
\end{conj}

In genus $1$ a similar Siegel modularity conjecture has been recently proposed by Greer-Tayou \cite{greer2026modularityspecialcyclesshimura}.

\subsection{Future directions}
Aside from computing the remaining cases for the invariants of the family $\CA_h$
and the questions on mirror symmetry, we mention three interesting directions:
\begin{itemize}
	\item[(i)] Let $\overline{\pi} : \overline{\CX} \to \overline{\CA_h}$ be a toroidal compactification of the family $\pi$. Gromov-Witten classes of the family $\overline{\pi}$ can be defined using log Gromov-Witten theory. It would be very interesting to explore how much of the quasimodularity conjecture extends to the boundary. Already for genus $1$ this is a difficult problem, see \cite{MR4329889}.
	\item[(ii)] Over the product $\CA_g \times \CA_h$ there exists a natural space of maps
	\[ \mathrm{Maps}_{(g,h),\chi} \to \CA_g \times \CA_{h} \]
	parametrizing triples $((A,\theta_A), (B,\theta_B), \varphi)$ where $(A,\theta_A) \in \CA_g$, $(B,\theta_B) \in \CA_h$ and $\varphi : B \to A$ is a homomorphism such that $\varphi^{\ast}(\theta_A)$ has characteristic polynomial $\chi$ with respect to $\theta_B$.
	By Hodge theory the expected codimension of $\mathrm{Maps}_{(g,h),\chi}$ is $gh$ and there are associated virtual cycles \cite{GreerLian}. By the theory of Kudla-Milson they are expected to have modular behaviour of weight $2gh$. However as shown in \cite{GreerLian} their images in $\Chow^{\ast}(\CA_g \times \CA_h)$ is mostly zero. 	In the case $g \leq 3$ the Torelli map $M_g \to \CA_g$ is a birational morphism.
	The virtual class of $\Mbar_{g}(\pi,\chi)/\CX$ can then be viewed as providing a natural extension of the virtual cycles of $\mathrm{Maps}_{(g,h),\chi}$ to the compactification $\Mbar_{g} \times \CA_h$.
	It is interesting to find natural compactifications of the virtual cycles of $ \mathrm{Maps}_{(g,h),\chi}$ for arbitrary $g,h \geq 4$ and understand their (quasimodular) properties. This would extend the Gromov-Witten classes considered here outside of the range prescribed by Proposition~\ref{thm:vanishing}.
	In fact, the vanishing of Proposition~\ref{thm:vanishing} may be seen as reflecting the fact that the image of the Torelli map has positive codimension for $g \geq 4$.
	\item[(iii)] One can consider the Gromov-Witten invariants of the universal K3 surface over a moduli space of lattice-polarized K3 surfaces, or more generally of the universal family over a moduli space of lattice-polarized hyperk\"ahler varieties.
	Multiple cover formulas in these cases were proposed in \cite{ObPand, O_NLGW} and show that the Gromov-Witten cycles in codimension $1$ should behave like theta-lifts of quasimodular forms.
	Under certain conditions these theta-lifts are quasimodular forms for the orthogonal group \cite{OWilliams}.
	However as discussed in \cite{OWilliams} these conditions are usually not met for K3 surfaces and so we expect more general automorphic objects such as higher Green functions or Eichler integrals of modular forms.
\end{itemize}

\subsection{Acknowledgements}
I would like to thank Rahul Pandharipande
for an inspiring walk along the Triest coastline in Spring 2025
and repeatedly promoting the family viewpoint. Many thanks also to Jan Bruinier, Aitor Iribar-Lopez, Niklas Ludwig, Aaron Pixton, Benjamin Sung and Brandon Williams for useful discussions about abelian varieties and modular forms.
The author was supported by the starting grant 'Correspondences in enumerative geometry: Hilbert schemes, K3 surfaces and modular forms', No 101041491 of the European Research Council,
and through the Collaborative Research Centre TRR 326 'Geometry and Arithmetic of Uniformized Structures', project number 444845124 by the Deutsche Forschungsgemeinschaft (DFG).

\section{Preliminaries on the Chow ring of abelian varieties}
\subsection{Definitions}
%Useful references are :
%[The action of SL2 on abelian varieties
%Arnaud BEAUVILLE, Theorem 3.3] \\
Let $\pi : \CX \to B$ be an abelian scheme of genus $h$, let $\sigma : B \to \CX$ be the zero section, let $\theta \in \Chow^1(\CX)$ be an ample divisor which is symmetric (i.e. $[-1]^{\ast} \theta = \theta$) and rigidified along the zero section, i.e. $\sigma^{\ast} \theta = 0$.
We assume that $\theta$ is a principal polarization and hence defines (after passing to a cover where it becomes integral) an isomorphism $\varphi_{\theta} : \CX \xrightarrow{\cong} \widehat{\CX}$, which we will use to identify $\CX$ with its dual.
The theta divisor satisfies the basic relations
\begin{equation} \frac{\theta^h}{h!} = \sigma, \quad \theta^{h+1} = 0. \label{sef3423} \end{equation}

The $n$-fold fiber product of the family $\CX \to B$ is denoted
\[ \CX^{\times n} =  \CX\times_B \ldots \times_B \CX. \]
Given a list of integers $I=(i_1,\ldots,i_k)$ with $1 \leq i_j \leq n$ for all $j$,
we have an associated morphism
\[ \pr_I : \CX^{\times n} \to \CX^{\times k}, \quad \pr_I(x_1,\ldots,x_n) = (x_{i_1},\ldots,x_{i_k}). \]
If $I$ is a list of distinct indices, then $\pr_I$ is just the projection.
However, we allow the indices to coincide. For example $\pr_{11} = \Delta \circ \pr_1$, where $\Delta : \CX \to \CX^2$ is the diagonal morphism.

For a class $\alpha \in \Chow^{\ast}(\CX^{\times k})$ we denote its pullback by $\pr_I$ by
\[ \alpha_I := \pr_I^{\ast}(\alpha). \]

\subsection{Motivic decompositions} \label{subsec:motivic decompositions}
The Lefschetz operator is the correspondence
\[ e_{\theta} = \Delta_{\ast}(\theta) \in \Chow^{h+1}(\CX \times_B \CX). \]
A central result of K\"unnemann is that there is a Lefschetz dual correspondence,
which by \cite[Lemma 3.1]{Kuennemann}
is given by the correspondence
\[ f_{\theta} = \pr_{13 \ast}\left( \Gamma_{\mathsf{m}} \cdot \pr_2^{\ast}\left(\frac{\theta^{h-1}}{(h-1)!} \right) \right) = \mathsf{d}^{\ast}\left( \frac{\theta^{h-1}}{(h-1)!} \right)
\in \Chow^{h-1}(\CX \times_B \CX) \]
where
\begin{itemize}
\item $\mathsf{m} : \CX \times_{B} \CX \to \CX$, $(x,y) \mapsto x+y$, is the addition map and $\Gamma_{\mathsf{m}}$ is its graph,
\item $\mathsf{d} : \CX \times_B \CX \to \CX$ is the difference map $(x,y) \mapsto x-y$.
\end{itemize}
In particular, the Lefschetz dual $f_{\theta}$ acts by 
\[ f_{\theta}(\beta) = \mathsf{m}_{\ast}\left(\frac{\theta^{h-1}}{(h-1)!} \times \beta \right). \]

\begin{example}
	For $h=1$ we have $f_{\theta} = \pi^{\ast} \pi_{\ast}$.
	Indeed, both sides are degree zero classes in $\CX \times_{B} \CX$ so the claim can be checked by restricting to a fiber over $B$, where it follows by a direct computation.
\end{example}

As an abelian scheme itself the $n$-fold product $\pi_n : \CX^n \to B$
is principally polarized by the symmetric and rigidified divisor
\[ \theta^{(n)} = \sum_{i=1}^{n} \theta_i. \]
The Lefschetz operators are easily checked to be
\[ e_{\theta(n)} = \sum_{i=1}^{n} e_{\theta_i} = \sum_{i=1}^{n} e_{\theta}^{(i)}. \]
and
\[ f_{\theta^{(n)}} = \sum_{i=1}^{n} f_{\theta}^{(i)}. \]
where $f_{\theta}^{(i)}$ is the Lefschetz dual $f_{\theta} \in \Chow(\CX^{\times 2})$ acting on the $i$-th factor.

\subsection{Multiplicity}
For an integer $m \in \BZ$ we let $[m] : \CX \to \CX$ be the morphism given by multiplication by $m$. 
If $\Gamma \in \CH^{\ast}(\CX)$ a common eigenvector under $[m]$ satisfying $[m]^{\ast}\Gamma = m^{i} \Gamma$, we write $i = \mult(\Gamma)$ and call it the multiplicity.

\begin{example} \label{example:multiplicity under homomorphism}
	For any homomorphism $\varphi : \CX \to \CY$ of abelian schemes, we have $\mult(\varphi^{\ast}(\Gamma)) = \mult(\Gamma)$ by  \cite[Prop 3.3]{DM}.
	In particular, if $\Delta : \CX \to \CX^{\times n}$ is the diagonal morphism, then $\mult(\Delta^{\ast}(\Gamma)) = \mult(\Gamma)$.
	%	(Since $\Delta$ is a homomorphisms this follows by \cite[Prop 3.3]{DM}).
\end{example}

\begin{example} \label{example:pushforward lemma}
	Let $\pi : \CX \to \CA_h$ be the projection. If $\Gamma$ is an eigenvector which is not of multiplicity $\mult(\Gamma)=2h$,
	then $\pi_{\ast} \Gamma = 0$. (By \cite[Lemma 2.18]{DM} we have $[m]_{\ast} \Gamma = m^{2h-\mult(\Gamma)} \Gamma$. Applying $\pi_{\ast}$ and using that $\pi_{\ast} \circ [m]_{\ast}  = \pi_{\ast}$, the claim follows.)
\end{example}

\begin{example}
	By the theorem of the cube $\mult(\theta) = 2$.
\end{example}

\begin{example}
	On $\CX \times \CX$ we have:
	\begin{gather*}
		\mult(\Delta) = \mult(\mathsf{d}^{\ast}(\theta^h/h!)) = \mult( \theta^h/h!) = 2h \\
		\mult(e_{\theta}) = \mult(\theta) + \mult(\Delta) = 2h+2 \\
		\mult(f_{\theta}) = \mult(\mathsf{d}^{\ast}(\theta^{h-1}/(h-1)!)) = 2h-2.
	\end{gather*}
	Moreover, $\mult(e_{\theta} \Gamma) = \mult(\Gamma) + 2$ and $\mult( f_{\theta} \Gamma) = \mult(\Gamma) - 2$.
	%We have $([m] \times [m])^{-1}(\Delta) = \sqcup_{a \in A[m]} [\Gamma_{\mathrm{tr}_a}]$, so need $[\Gamma_{\mathrm{tr}_a}] = [\Delta]$.
\end{example}

\subsection{Tautological rings} \label{subsec:tautological rings}
Define the Mumford class
\[ \eta = \frac{1}{2} (\mathsf{m}^{\ast}(\theta) - \theta_1 - \theta_2). \]
By the theorem of the cube one has\footnote{On a finite \'etale cover of $\CA_h$ where $\theta$ becomes integral, one has
$\eta = \frac{1}{2} c_1(\CP)$, where 
$\CP := \mathsf{m}^{\ast}(\CO(\theta)) \otimes \CO(\theta_1)^{-1} \otimes \CO(\theta_2)^{-1}$ is the Poincar\'e bundle. By the theorem of the cube one has $(1 \times [m])^{\ast} \CP = \CP^{\otimes m}$
which implies $(1 \times [m])^{\ast}\eta=m \eta$ on the cover, so also on $\CA_h$.}
%\[ (1 \times [m])^{\ast} \CP = \CP^{\otimes m} \]
\[ (1 \times [m])^{\ast}\eta=m \eta. \]
Also observe that
\[ \Delta^{\ast}(\eta) = \theta. \]
In particular,
\[ \eta_{ii} = \theta_i. \]

The tautological subring of the base $R^{\ast}(B)$ is the subring generated by the Chern classes $\lambda_i$ of the Hodge bundle $\BE = \pi_{\ast} \Omega_{\CX/B}$.

\begin{defn}[{\cite{10authors}}] The vertical tautological ring of $\CX^{\times n}$ is
	the subalgebra
\[ R^{\ast}_{\mathsf{vert}}(\CX^{\times n}) = \BQ[ \eta_{kl} | 1 \leq k,l \leq n ] \subset \Chow^{\ast}(\CX^{\times n}). \]
The tautological ring of $\CX^{\times n}$ is the subring $R^{\ast}(\CX^{\times n})$ generated by $R^{\ast}_{\mathsf{vert}}(\CX^{\times n})$ and $R^{\ast}(B)$.
\end{defn}

\begin{lemma}[{\cite{10authors}}] \label{lemma:10 author lemma}
(i) For the universal family $\pi : \CX \to \CA_h$ we have
\[ R^{\ast}(\CX^{\times n}) \cong R^{\ast}_{\mathsf{vert}}(\CX^{\times n}) \otimes R^{\ast}(B). \]
(ii) For any $b \in B$ with fiber $(X,\theta)$ over $b$, there is an isomorphism
\[ R^{\ast}_{\mathsf{vert}}(\CX^{\times n}) \to R^{\ast}_{\mathsf{vert}}(X^n) \to H^{\ast}(X^n)^{\mathrm{Sp}(2h,\BZ)}
=
\left( \bigwedge^{\ast} V^{\oplus n} \right)^{\mathrm{Sp}(V)}, \]
where $V = H^1(X,\BQ)$.
\end{lemma}

In particular, to understand relations among the classes $\eta_{ij}$ we may work with a fixed abelian variety of our choice.
For example, we may specialize to the product of elliptic curves.

For any symmetric $n \times n$ matrix $M = (m_{ij})$ there is an associated divisor
$\sum_{i,j} m_{ij} \eta_{ij}$.
By Grothendieck-Riemann-Roch one has the following evaluation \cite{10authors}:
\begin{equation} \int_{\CX^{n}} \frac{1}{(hn)!}\left( \sum_{1 \leq i,j \leq n} m_{ij} \eta_{ij} \right)^{hn} = \det(M)^h. \label{integral formula} \end{equation}

Consider the morphism
\[ \varphi_{a_1,\ldots,a_n} : \CX^{\times n} \to \CX, \quad (x_1,\ldots,x_n) \mapsto \sum_i a_i x_i. \]
One has \cite{10authors}
\begin{equation} \varphi^{\ast}(\theta) = \sum_{1 \leq i,j \leq n} a_i a_j \eta_{ij}. \label{pullback relation of theta} \end{equation}
This implies the vanishing:
\begin{equation} \left(\sum_{1 \leq i,j \leq n} a_i a_j \eta_{ij} \right)^{h+1} = 0. \label{3sdf03} \end{equation}
Treating $a_i$ here as formal variables,
for any string $I$ of integers from $\{ 1, \ldots, n \}$ of length $2h+2$ the coefficient of $a_I = \prod_{i \in I} a_i$ in \eqref{3sdf03} spans all the relations in $R_{\mathrm{vert}}(\CX^{\times n})$, see \cite{10authors}.
In particular, this says that for any $I$ we have
\begin{equation} \sum_{(J_1,\ldots,J_{h+1})} \eta_{J_1} \cdots \eta_{J_{h+1}} = 0 \label{explicit vanishing} \end{equation}
where the sum is over all possible partitions $(J_1,\ldots, J_{h+1})$ of $I$ into $h+1$ parts, each consisting of a string of length $2$.

\begin{example} \label{example:E}
Consider an elliptic curve $E$ with class of a point $\omega \in H^2(E)$ and diagonal class $\Delta \in \Chow^{\ast}(E \times E)$. Then
\[ \eta = \frac{1}{2} \left( - \Delta + \omega_1 + \omega_2 \right). \]
The relations for the strings $(1112)$, $(1122)$, $(1123)$, $(1234)$ give the relations
\begin{gather} \omega_1 \eta = \omega_2 \eta = 0, \quad \eta \cdot \eta = -\frac{1}{2} \omega_1 \omega_2  \label{basicidentity2cxx} \\
 \eta_{12} \eta_{13} = -\frac{1}{2} \theta_1 \eta_{23} \label{basic identity}  \\
 \eta_{12} \eta_{34} + \eta_{13} \eta_{24} + \eta_{14} \eta_{23} = 0 \notag
 \end{gather}
which of course can also be checked directly by hand. Iterating \eqref{basic identity} one finds
\begin{equation} \label{cycle integral}
\int_{E^r} \eta_{12} \eta_{23} \cdots \eta_{r-1} \eta_{r 1} = \left( -\frac{1}{2} \right)^{r-1}
\end{equation}
\end{example}

\begin{example} \label{example:some eta relation}
By applying \eqref{explicit vanishing} for the string $(1^{2h} 23)$ we have 
\[ \frac{\theta_1^{h-1}}{(h-1)!} \eta_{12} \eta_{13} = -\frac{1}{2} \frac{\theta_1^h}{h!} \eta_{23} \]
\end{example}

\begin{comment}
\begin{lemma}
Let $\Gamma \in \Rvert(\CX^{\times n})$ such that the degree in the first factor is exactly $2h$, that is $\Gamma|_{X^n} \in H^{2h}(X) \otimes H^{\ast}(X^{n-1})$ for $(X,\theta) \in \CA_h$ any point. Then
\[ \Gamma = \theta_1^{h} P \]
where $P$ is a polynomial in the classes $\eta_{ij}$ for $i,j \geq 2$.
\end{lemma}
\begin{proof}
We can assume $\Gamma =  \prod_{i \leq j} \eta_{ij}^{b_{ij}} $.
By assumption we have $2 b_{11} + \sum_{j \geq 2} b_{1j} = 2h$.
We argue by increasing induction on $k$ that
\[ \Gamma = \theta_1^k P_k \]
for some polynomials $P_{k}$ in $\{ \eta_{ij} \}_{i,j \geq 1}$.
In the case $k=0$ this is no restriction, so this case holds. The case $k=h$ is what we want to prove since then $P_h$ can not depend on $\eta_{1j}$.

So by the induction hypothesis for $k$ we may assume that 
\[ \Gamma = \theta_1^k \eta_{1 i_1} \cdots \eta_{1 i_{2(h-k)}}. \]
Consider the string 
\[ I = (1^{2k+m} i_1 \cdots i_m) \]
where $m=h-k+1$, so the string is of length $2k+2m=2h+2$.
In a partition $(J_1,\ldots,J_{h+1})$ if all the $i_1,\ldots, i_m$ pair with elements $1$, then we get the contributing factor $\Gamma$. Otherwise, we get a factor of the form $\theta_1^{k+1} \cdot P'$ for a polynomial $P'$. Hence the induction hypothesis holds for $k+1$.
\end{proof}
\end{comment}

\begin{lemma} \label{lemma:phi pullback eval}
For $a \in \BZ$ consider
$\varphi : X^{n+1} \to X^n, (x_0,\ldots,x_n) \mapsto (a x_0+ x_1,\ldots,ax_0 + x_n)$. Then for all $i,j \geq 1$ we have
\[ \varphi^{\ast}(\theta_i) = a^2 \theta_0 + 2 a \eta_{0i} + \theta_i \]
\[ \varphi^{\ast}(\eta_{ij}) = a^2 \theta_0 + a \eta_{0i}+ a \eta_{0j} + \eta_{ij} \]
\end{lemma}
\begin{proof}
This follows directly from \eqref{pullback relation of theta} and the definition of $\eta_{ij}$.
\end{proof}

The tautological ring contains the class of the diagonal and of the Lefschetz correspondences $e_{\theta}, f_{\theta}$:
\begin{align*}
\Delta & = d^{\ast}\left( \frac{\theta^h}{h!} \right) = \frac{(\theta_1+\theta_2 - 2\eta)^h}{h!} \\
f_{\theta} & = 
d^{\ast}\left( \frac{\theta^{h-1}}{(h-1)!} \right) = \frac{(\theta_1+\theta_2 - 2\eta)^{h-1}}{(h-1)!} \\
e_{\theta} & = \Delta \theta_1
\end{align*}

We will need some basic computations later on:
\begin{lemma} \label{lemma:F theta computation}	Let $i_1,\ldots, i_4 \geq 1$ be integers.
\begin{enumerate}
	\item[(a)] $f_{\theta}(\theta^k) = k (h+1-k) \theta^{k-1}$.
	\item[(b)] $f_{\theta}^{(0)}( \eta_{0 i_1} \eta_{0 i_2} ) = - \frac{1}{2} \eta_{i_1 i_2}$
	\item[(c)] $f_{\theta}^{(0)}(\theta_0 \eta_{0 i_1} \eta_{0 i_2}) = (h-2) \eta_{0 i_1} \eta_{0 i_2} - \frac{1}{2} \theta_0 \eta_{i_1 i_2}$
	\item[(d)] $f_{\theta}^{(0)}(\eta_{0 i_1} \eta_{0 i_2} \eta_{0 i_3} \eta_{0 i_4}) = -\frac{1}{2} \sum_{1 \leq a < b \leq 4} \eta_{0 i_a} \eta_{0 i_b} \eta_{i_c i_d}$, where $\{c , d \}$ is the complement to $\{ a , b \}$. Equivalently,
	\begin{align*} f_{\theta}^{(0)}(\eta_{0 i_1} \eta_{0 i_2} \eta_{0 i_3} \eta_{0 i_4}) = -\frac{1}{2}( \eta_{0 i_1} \eta_{0 i_2} \eta_{i_3 i_4} + \eta_{0 i_1} \eta_{0 i_3} \eta_{i_2 i_4} + \eta_{0 i_1}\eta_{0 i_4} \eta_{i_2 i_3} \\
	 + \eta_{0 i_2} \eta_{0 i_3} \eta_{i_1 i_4} + \eta_{0 i_2}\eta_{0 i_4} \eta_{i_1 i_3} + \eta_{0 i_3} \eta_{0 i_4} \eta_{i_1 i_2}). \end{align*}
\end{enumerate}
\end{lemma}
\begin{proof}
To compute this we can specialize to $X = E_1 \times \cdots \times E_h$. We write $\omega^{(i)}$ for the point class on $E_i$ and denote by this also the pullback to the product. We denote by $\eta^{(i)}$ the class $\eta$ on $E_i \times E_i$, and view it also as a class on $X \times X$ by pullback.
We take the principal polarization
\[ \theta = \sum_i \omega^{(i)} \]
which gives us
\[ \eta = \sum_i \eta^{(i)}. \]
The claims then follow by a straightforward computation using the relations of Example~\ref{example:E}.
\end{proof}

\begin{rmk}
The tautological ring and the classes $\eta$ can be defined identically also for
families $\pi : \CX \to B$ of polarized varieties (not only for principal polarizations). The definitions
are identical and (with $\mathrm{Sp}(2h,\BZ)$ replaced by the subgroup of $\GL(H^1(X,\BZ))$ respecting the symplectic form defined by the polarization) Lemma~\ref{lemma:10 author lemma}(ii) holds likewise.
\end{rmk}

\subsection{The characteristic polynomial}
The characteristic polynomial associated to any element $v \in H^2(X,\BZ)$ is
\[ \chi_{v}(t) = \frac{1}{h!} \int_{X} (t \theta - u)^h. \]
This is a monic polynomial with integral\footnote{Indeed, if $\alpha \in H^2(X,\BZ)$ then there exists a basis $e_i,f_i$ such that $\alpha = \sum_i \alpha_i e_i \wedge f_i$. This shows that $\alpha^k/k!$ is integral for all $k \geq 1$.} coefficients:
\[ \chi_{v}(t) = t^{h} - e_{1} t^{h-1} + \ldots + (-1)^h e_h. \]
If $v \in \NS(X)$, then $\chi_v(t)$ is also the characteristic polynomial of the associated endomorphism of $X$ (over $\BQ$) defined by
%\footnote{The square $\chi_v(t)^2$ is the characteristic polynomial of the induced action on $H^1(X,\BQ)$, so $\chi }
\[ X \xrightarrow{\varphi_{v}} \hat{X} \xrightarrow{\varphi_{\theta}^{-1}} X. \]
Since this endomorphism is self-adjoint, all roots of $\chi_v$ are real.
If $v$ is effective, then the roots are all non-negative, see \cite[Theorem 5.2.4]{BL} for the case if $v$ is ample and \cite[Theorem 2.1(c)]{DebarreCurves} how to reduce to the ample case.

\begin{lemma} \label{lemma:bound on e_k} 
 Let $v \in \NS(X)$ be effective and write
	$\chi_{v}(t) = t^{h} - e_{1} t^{h-1} + \ldots + (-1)^h e_h$. Then
	\[ e_k \leq \binom{h}{k} \left( \frac{e_1}{h} \right)^k. \]
\end{lemma}
\begin{proof}
	The coefficient 
	$e_k = e_k(b_1,\ldots,b_h)$
	is the elementary symmetric polynomial in the roots $b_1,\ldots,b_h$ of the characteristic polynomial.
	Thus the result follows from the Maclaurin-inequality for elementary symmetric polynomials:
	\[
	\frac{e_k(b_1,\ldots,b_h)}{\binom{h}{k}} \leq \left( \frac{e_1(b_1,\ldots,b_h)}{h} \right)^k.
	\]
\end{proof}

Basic invariance theory of the symplectic group gives the following:
\begin{lemma} \label{lemma:invariants pfaffian}
For any polynomial $P(x_1,\ldots,x_n)$ and tautological class $\Gamma \in R_{\mathsf{vert}}(X^n)$ there exists a polynomial $R$ such that for any $v \in H^2(X,\BZ)$ we have
\[
\int_{X^n} P(v_1,\ldots,v_n) \cdot \Gamma = R(e_1,\ldots,e_h)
\]
where $\chi_{v}(t) = t^{h} - e_{1} t^{h-1} + \ldots + (-1)^h e_h$ is the characteristic polynomial.
If $P$ has degree $m$, then $R$ has weighted degree $m$ where $e_i$ has weight $i$.
\end{lemma}
\begin{proof}
Consider $V=H^1(X)$ viewed as a symplectic vector space defined via $\theta$.
Then $H^2(X) = \wedge^2 V$ can be identified with skew-symmetric endomorphisms of $V$,
and $\chi_v(t)$ is the Pfaffian of an element $v \in \wedge^2 V$. By invariant theory of the symplectic group all the invariants of $\wedge^2 V$ are polynomials in the coefficients of the Pfaffian, see for example \cite[Section 12.4.3, Type AII]{GoodmanWallach}.
%Roe Goodman and Nolan R. Wallach, Symmetry, Representations, and Invariants, Graduate Texts in Mathematics 255, Springer, 2009, §12.4.2, Theorem 12.4.5, and §12.4.3, Type AII, pp. 592–594.
Since $\eta_{ij}$ are invariant under $\Sp(V)$ also the integral is, so it is an invariant.
For the second claim, assume that $P$ has degree $k$. If $v$ is replaced by $\lambda v$, then $\chi_{\lambda v} = t^h - e_1' t^{h-1} + \ldots + (-1)^h e'_h$ where $e'_k = \lambda^k e_k$. So we find
\[
\lambda^k R(e_1,\ldots,e_h) = 
\int_{X^n} P( (\lambda v)_1,\ldots, (\lambda v)_n)
=
R(e'_1,\ldots,e'_h) = R(\lambda^1 e_1, \ldots, \lambda^h e_h),
\]
so $R$ is of weighted degree $k$.
\end{proof}

\subsection{Characteristic polynomial of a curve class}
Let $\beta \in H_2(X,\BZ)$ be a class.
We will write $v_{\beta} \in H^2(X,\BZ)$ for the image of $\beta$ under the natural isomorphism
\[ H_2(X,\BZ) \cong H^2(X,\BZ)^{\ast} \cong H^2(\widehat{X},\BZ) \xrightarrow{\varphi_{\theta}^{\ast}} H^2(X,\BZ). \]
%For $\beta \in H_2(X,\BZ)$ we write $v_{\beta} \in H^2(X,\BZ)$ be the associated divisor class.
We define the characteristic polynomial of $\beta$ to be the one of $v_{\beta}$. We write $\chi_{\beta} := \chi_{v_{\beta}}$.

\begin{lemma}
	Let $\beta \in H_2(X,\BZ)$. Let $\chi_{\beta} = t^h - e_1 t^{h-1} + \ldots + (-1)^h e_h$. Then
	\[ e_1 = \int_{\beta} \theta. \]
\end{lemma}
\begin{proof}
Observe that the statement makes sense for any $\beta \in H_2(X,\BR)$. We will work in this greater generality.
We assume first that the characteristic polynomial of $\beta$ has no multiple roots. 
	By \cite[Lemma A.3]{hotchkiss2024periodindexconjectureabelianthreefolds}
	there exists a basis $\{ dx_i, dy_i \}_{i=1,\ldots, h}$ of $H^1(X,\BR)$ such that
	\[ \theta = \sum_i dx_i \wedge dy_i, \quad v_{\beta} = \sum_i b_i dx_i \wedge dy_i. \]
	
	Since the question is intersection-theoretic, we may hence assume that 
	\[ X=E_1 \times \cdots \times E_h, \quad \theta=\sum_i \omega^{(i)},
	\quad v_{\beta} = \sum_i b_i \omega^{(i)} \]
	where $E_i$ are elliptic curves and $\omega^{(i)} \in H^2(E_i,\BZ)$
	are the point classes on $E_i$ viewed here as classes on $X$ via pullback along the projection $X \to E_i$.
	Moreover, if we identify $\widehat{E_i} = E_i$ in the natural way, then $\varphi_{\theta} : \prod_i E_i \to \widehat{X} = \prod_i E_i$ is just the identity.
	The isomorphism $H_2(X,\BZ) \cong H^2(\widehat{X},\BZ)$ 
	restricts to $H_2(E_i,\BZ) \subset H_2(X,\BZ)$ as the map taking $[E_i] = \prod_{\ell \neq i} \omega^{(\ell)}$ to 
	$\omega^{(i)}$. Thus
	\[ \beta = \sum_i b_i \cdot \prod_{\ell \neq i} \omega^{(\ell)}
	\]
	Then
	\[ e_1 = \frac{1}{(h-1)!} \int_{E_1 \times \cdots \times E_h} \theta^{h-1} v_{\beta} = b_1 + \ldots + b_h = \int_{\beta} \theta. \]
	The claim follows in general since the locus with no multiple roots is dense.
\end{proof}

We have the following analogue of Lemma~\ref{lemma:invariants pfaffian}:
\begin{lemma} \label{lemma:reconstruction}
	For any tautological class $\Gamma \in R_{\mathsf{vert}}(X^n)$ there exists a polynomial $R(z_1,\ldots,z_h)$
	of weighted degree $n$ (with $z_i$ having weight $i$) such that for any $\beta \in H_2(X,\BZ)$ we have
	\[
	\int_{[\beta] \times \cdots \times [\beta]} \Gamma = R(e_1,\ldots,e_h)
	\]
	where $\chi_{\beta}(t) = t^{h} - e_{1} t^{h-1} + \ldots + (-1)^h e_h$ is the characteristic polynomial of $\beta$.
\end{lemma}
\begin{proof}
	By Poincar\'e duality we may view $\beta$ as a class in $H^{2h-2}(X,\BZ)$.
	The integral on the left hand side is hence
	\[ \int_{X^n} \beta_1 \cdots \beta_n \cdot \Gamma. \]
The Fourier Mukai transform $F$ is the correspondence $\exp(2 \eta) = \exp(c_1(\CP))$, where $\CP$ is the Poincar\'e bundle on $X \times X$.
We have $F(\beta) = -v$, where we write $v := v_{\beta}$.
%	The class $u:=v_{\beta}$ is $(-1)$ times the Fourier transform of $\beta$.
Moreover for $\alpha_1,\alpha_2 \in H^{\ast}(X)$
	\[ \int_{X} F(\alpha_1) F(\alpha_2) = (-1)^h \int_{X} \alpha_1 \cdot \alpha_2. \]
Moreover, write also $F$ for the action of the Fourier-Mukai transform on $X^n$ which acts factor-wise.
Then 
%	Since the Fourier-Mukai transform acts factorwise 
$F(\beta_1 \cdots \beta_n) = (-1)^n v_{1} \cdots v_n$. So we get
	\[ \int_{X^n} \beta_1 \cdots \beta_n \cdot \Gamma = (-1)^{hn + n} \int_{X^n} v_{1} \cdots v_n \cdot F(\Gamma). \]
	Moreover, $F$ sends the tautological ring to the tautological ring (since pushforward along projections $X^n \to X^m$ preserves the tautological ring) So the claim follows from Lemma~\ref{lemma:invariants pfaffian}.
\end{proof}

We give a concrete example of Lemma~\ref{lemma:reconstruction} that will be used repeatedly below.
Let $r \geq 1$. An $r$-cycle of $\eta$-classes is the class
\[ \Gamma_r = \eta_{12} \eta_{23} \cdots \eta_{r-1, r} \eta_{r1}. \]
In particular
\[ \Gamma_1 = \eta_{11} = \theta_1, \quad \Gamma_2 = \eta_{12}^2, \quad \Gamma_3 = \eta_{12} \eta_{23} \eta_{31}. \]

\begin{lemma} \label{lemma:intersection computation}
	Let $\beta \in H_2(X,\BZ)$ with characteristic polynomial 
	$\chi_{\beta}(t) = t^{h} - e_{1} t^{h-1} + \ldots + (-1)^h e_h$. Then
	\[
	\int_{\beta^{\times r}} \Gamma_r = \left( -\frac{1}{2} \right)^{r-1} p_r(e_1,\ldots,e_h),
	\]
	where $p_r$ is the $r$-th power sum symmetric function written in terms of the elementary symmetric functions $e_1,\ldots,e_h$.
\end{lemma}

Concretely, if we let
\[ E(z) = 1 + e_1 z + e_2 z^2 + \ldots + e_h z^h = (1+b_1 z) \ldots (1 + b_h z) \]
be the generating series of elementary symmetric polynomials in variables $(b_1,\ldots,b_h)$, then
$p_r = \sum_i b_i^r$, so \[
E(z) = \exp\left( \sum_{r \geq 1} \frac{(-1)^{r-1}}{r} z^r p_r \right).
\]
For example,
\begin{align*}
	p_1 & = e_1 \\
	p_2 & = e_1^2 - 2 e_2 \\
	p_3 & = e_1^3 - 3 e_1 e_2 + 3 e_3.
%	p_4 & = e_1^4 - 4 e_1^2 e_2 + 2 e_2^2 + 4 e_1 e_3 - 4 e_4 
\end{align*} 

\begin{proof}
	As before we may assume that
	\[ X=E_1 \times \cdots \times E_h, \quad \theta=\sum_i \omega^{(i)},
	\quad \beta = \sum_i b_i \cdot \prod_{\ell \neq i} \omega^{(\ell)},
	\quad \eta = \sum_i \eta^{(i)}. \]
	where $\eta^{(i)}$ are the $\eta$-classes for $E_i$ with respect to the unique principal polarization, given here explicitly as
	\[ \eta^{(i)} = \frac{1}{2} (- \Delta_{E_i} + \omega^{(i)}_1 + \omega^{(i)}_2 ) \in H^2(E_i \times E_i) \subset H^2(X \times X). \]
	Observe that
	\begin{equation*} \omega_1^{(i)} \eta^{(i)} = \omega_2^{(i)} \eta^{(i)} = 0. \end{equation*}
	
	We find
	\[
	\int_{\beta^{\times r}} 
	\eta_{12} \eta_{23} \cdots \eta_{r1}
	=
	\sum_{i_1,\ldots,i_r = 1}^{h} 
	\int_{\beta^{\times r}}
	\eta_{12}^{(i_1)} \eta_{23}^{(i_2)} \cdots \eta_{r1}^{(i_r)}. \]
	We expand the first factor of $\beta$ as $\sum_i b_i [E_i]$.
	For the contribution from the summand $b_i [E_i]$ to be non-zero we must have $i_1=i$ by \eqref{basicidentity2cxx}.
	But then in the second factor $\beta$ only $b_i [E_i]$ can contribute, which then forces $i_2=i$. Continuing in this way the above integral becomes
	\[
	\sum_{i=1}^{h} b_i^r \int_{E_i^r} \eta_{12}^{(i)} \eta_{23}^{(i)} \cdots \eta_{r1}^{(i)}
	=
	p_r \cdot \left( -\frac{1}{2} \right)^{r-1}
	\]
	where we used \eqref{cycle integral}.
	%
	%By \eqref{integral formula} we have the identity for integrals over an elliptic curve $E$:
	%\[ \frac{1}{r!} \int_{E^r} (2 m_{12} \eta_{12} + 2 m_{23} \eta_{23} + \ldots + 2 m_{r1} \eta_{r1})^r
	%=
	%\det \begin{pmatrix} 0 & m_{12} & & \cdots & m_{n1} \\
		%	m_{12} & 0 & m_{23} & \\
		%	\vdots & m_{23} & 0 & \ddots & \vdots  \\
		%	 & & \ddots & \ddots &  m_{n-1,n} \\
		%	m_{n,1} & \cdots & & m_{n-1,n} & 0 
		%	\end{pmatrix} 
	%\]
	%By the sum over partitions formula for the determinant, extracting the coefficient of $m_{12} m_{23} \cdots m_{r1}$ gives
	%\[
	%2^r \int_{E^r}\eta_{12} \eta_{23} \cdots \eta_{r1} = 2 (-1)^{r-1}.
	%\]
	This proves the formula for a Zariski open subset of $\beta$'s, so the formula holds in general.
\end{proof}

\subsection{Basic evaluation}
For $\BE_h\to \CA_h$ and $\BE_{\Mbar_{g,n}} \to \Mbar_{g,n}$ the Hodge bundles we set
\[ \lambda_i = c_i(\BE_h), \quad \mu_i = c_i(\BE_{\Mbar_{g,n}}). \]
%We let $\lambda_i = c_i(\BE_h)$ and $\mu_i = c_i(\BE_{\Mbar_{g,n}})$.
By definition $\lambda_0 = \mu_0 = 1$.
\begin{lemma} \label{lemma:e(Hodge Hodge) eval}
	Let $g \geq 1$ and $h \geq 2$.
	On $\Mbar_{g,n} \times \CA_h$ we have
	\[
	c_{gh}( \BE_{\Mbar_{g,n}} \otimes \BE_h )
	=
	\begin{cases}
		\mu_1 \lambda_{h-1} & \text{ if } g=1 \\
		\mu_{2} \mu_1 \lambda_{h-1} \lambda_{h-2} & \text{ if } g=2 \\
		\mu_{3} \mu_2 \lambda_{1} & \text{ if } g=3, h = 2 \\
		\mu_{3} \mu_2 \mu_1 \lambda_{h-1} \lambda_{h-2} \lambda_{h-3}  & \text{ if } g=3, h \geq 3 \\
		\mu_g \mu_{g-1} \lambda_1 & \text{ if } g \geq 4, h=2 \\
		\mu_g \mu_{g-1} \mu_{g-2} \lambda_2 \lambda_1 & \text{ if } g \geq 4, h=3 \\
		0 & \text{ if } g \geq 4 \text{ and } h \geq 4 \\
		%
		% \mu_{2} \mu_1 \lambda_{h-1} \lambda_{h-2} & \text{ if } g=2, h \geq 2 \\
		% \mu_1 \lambda_{h-1} & \text{ if } g=1, h \geq 2 \\
	\end{cases}
	\]
	Moreover, $c_{gh}(\BE_{\Mbar_{g,n}} \otimes \BE_h) = \mu_{g}$ if $h=1$.
\end{lemma}
\begin{proof}
	Let $h \geq 2$.
	By the splitting principle we may assume $\BE_{\Mbar_{g,n}} = L_1 \oplus \ldots \oplus L_g$ with $L_i$ line bundles with $l_i = c_1(L_i)$.
	Then
	\[ c_{gh}(\BE_{\Mbar_{g,n}} \otimes \BE_h)
	= (\lambda_h + \ell_1 \lambda_{h-1} + \ldots + \ell_1^h) \cdots (\lambda_h + \ell_g \lambda_{h-1} + \ldots + \ell_g^h) \]
	Now using that $\lambda_h = 0$ we can pull out a factor of $l_1 \ldots l_g = \mu_g$.
	Using $\mu_g^2=0$ we further pull out a factor of $\lambda_{h-1}$.
	Then using $\lambda_{h-1}^2 = 0$ (part of the $c(\BE_h) c(\BE_h^{\vee}) = 1$ relation) the above simplifies to
	\[ \mu_g \lambda_{h-1} \sum_{i=1}^{g} \left( \prod_{j \neq i} \ell_j \right) \left( \prod_{j \neq i} \lambda_{h-2} + \ell_j \lambda_{h-3} + \ldots + \ell_j^{h-2} \right) \]
	If $g=1$ we are done. For $g \geq 2$ the $\mu_i$ are pulled back from $\Mbar_{g}$, so the total degree in the $\mu_i$ is at most $3g-3$.
	Hence the inner product must be of total degree $g-2$. If $g=2$ the inner term contributes a $\lambda_{h-2}$ and we are done.
	Let hence $g \geq 3$. If $h=2$ then the inner product is $1$ and we get $\mu_g \mu_{g-1} \lambda_{h-1}$.
	So let $h \geq 3$ also. Since the inner product has $g-1$ terms but is of degree $g-2$ in the $\ell_i$ we must have one factor of  $\lambda_{h-2}$; since $\lambda_{h-1} \lambda_{h-2}^2 = 0$ in fact we can have only one such factor. Hence by degree reasons again the remaining terms must be all of the form $\lambda_{h-3} \ell_j$. So we end up with
	\[
	\mu_g \lambda_{h-1} \lambda_{h-2} \lambda_{h-3}^{g-2} \sum_{i=1}^{h} \ell_1 \cdots \widehat{\ell_i} \cdots \ell_g \sum_{j \neq i} \ell_1 \cdots \widehat{\ell_i} \cdots \widehat{\ell}_j \cdots \ell_g.
	\]
	Using 
\[ \mu_g \mu_{g-1} \mu_{g-2} = \mu_g \sum_{i=1}^{h} \ell_1 \cdots \widehat{\ell_i} \cdots \ell_g \sum_{j \neq i} \ell_1 \cdots \widehat{\ell_i} \cdots \widehat{\ell}_j \cdots \ell_g, \] and $\lambda_{h-1} \lambda_{h-1} \lambda_{h-2}^{2} = 0$ for $h \geq 3$ the claim follows.
\end{proof}

\begin{lemma}
	For any $u < g$ as a formal expression in $\mu_1,\ldots,\mu_g, \lambda_1,\ldots,\lambda_h$ we have
	\[ c_{(g-u)h}( \BE_{\Mbar_{g}} \otimes \BE_{\CA_h}) = c_{(g-u)h}(\BE_{\Mbar_{g-u}} \otimes \BE_{\CA_h}) \]
	modulo $\mu_{g}, \mu_{g-1}, \ldots, \mu_{g-u+1}$.
\end{lemma}
\begin{proof}
	By the splitting principle we may assume $\BE_{\Mbar_{g,n}} = \BE' \oplus \CO^{\oplus u}$ where $c(\BE') = c(\BE_{\Mbar_{g,n}}) = 1 + \mu_1 + \ldots + \mu_{g-u}$, so $c_{(g-u)h}( \BE_{\Mbar_{g}} \otimes \BE_{\CA_h}) = c_{(g-u)h}( \BE' \otimes \BE_{\CA_h})$.
	Moreover, Mumfords relation applies also to the Chern classes of $E'$.
	By the proof of the last lemma applied to $\BE'$ we get hence the same result as for
	$c_{(g-u)h}(\BE_{\Mbar_{g-u}} \otimes \BE_{\CA_h})$.
	%It follows that the proof of the last proposition goes through exactly the same way as for $ 
\end{proof}

\section{Generalities on Gromov-Witten invariants}
\subsection{Conventions}
Consider a monic polynomial with integral coefficients of degree $h$ and with non-negative real roots
\[ \chi = t^h - e_1 t^{h-1} + e_2 t^{h-2} + \ldots + (-1)^h e_h. \]
We define the degree-refined Gromov-Witten invariants by
\[
\CC_{g,\chi}(\tau_{k_1 \ldots k_n}(\Gamma)) := 
\rho_{\ast}\left( \ev^{\ast}(\Gamma) \prod_{i=1}^{n} \psi_i^{k_i} \cap  [\Mbar_{g,n}(\pi,\chi) ]^{\vir} \right)
\]
The usual Gromov-Witten invariants are recovered by summation
\[
\CC_{g,d}(\tau_{k_1 \ldots k_n}(\Gamma)) =  \sum_{\substack{\chi = t^h - e_1 t^{h-1} + \ldots \\ e_1 = d}}
\CC_{g,\chi}(\tau_{k_1 \ldots k_n}(\Gamma))
\]
where the sum is finite by Lemma~\ref{lemma:bound on e_k}.

If the class $\Gamma \in \Chow^{\ast}(\CX^{\times n})$ splits as a product
\[ \Gamma = \pr_{I_1}^{\ast}(\Gamma_1) \pr_{I_2}^{\ast}(\Gamma_2) \]
for some decomposition $\{ 1, \ldots, n \} = I_1 \sqcup I_2$ we
write 
\[ \tau_{k_1 \ldots k_n}(\Gamma)
=
\tau_{(k_i)_{i \in I_1}}(\Gamma_1) \, \tau_{(k_i)_{i \in I_2}}(\Gamma_2), \]
so for example $\tau_{a} \tau_b(\theta_1 \theta_2) = \tau_a(\theta) \tau_b(\theta)$.

\subsection{String and dilaton equation}
We record the standard string and dilaton equations.
\begin{lemma} \label{lemma:string dilaton equation}
	Let $\Gamma \in H^{\ast}(\CX^{n})$ and let $\widetilde{\Gamma}$ be its pullback to $\CX^{n+1}$ via the projection that forgets the last factor. Let $\alpha=(\alpha_1,\ldots,\alpha_n)$. Then
	\[
	\CC_{g,\chi}(  \tau_{(\alpha_1,\ldots,\alpha_n,0)}(\widetilde{\Gamma}) )
	=
	\sum_{i=1}^{n} \CC_{g,\chi}\left( \tau_{(\alpha_1,\ldots,\alpha_{i-1}, \alpha_i - 1, \alpha_{i+1}, \ldots \alpha_n)}(\Gamma) \right)
	\]
	\[
	\CC_{g,\chi}\left( \tau_{(\alpha_1,\ldots,\alpha_n,0)}(\widetilde{\Gamma} \eta_{1, (n+1)}) \right)
	=
	\sum_{i=1}^{n} \CC_{g,\chi}\left( \tau_{(\alpha_1,\ldots,\alpha_{i-1}, \alpha_i - 1, \alpha_{i+1}, \ldots \alpha_n)}(\Gamma \eta_{1 i}) \right)
	\]
	
	\[
	\CC_{g,\chi}\left( \tau_{(\alpha_1,\ldots,\alpha_n,1)}(\widetilde{\Gamma}) \right)
	= (2g-2+n)
	\CC_{g,\chi}\left( \tau_{\alpha}(\Gamma) \right)
	\]
	where the terms involving $\tau_{\alpha}$ with $\alpha_i<0$ are defined to vanish.
\end{lemma}

\subsection{Vanishing relation I} \label{subsec:vanishing reln 1}
For $n \geq 1$ define the subspace
\[ \Mbar_{g,n}(\CX,\chi)^0 = \ev_1^{-1}(\sigma) \subset \Mbar_{g,n}(\CX,\chi) \]
which parametrizes stable maps $(f:C \to X,p_1,\ldots,p_n)$ such that $f(p_1)=0_X$.
There is a natural projection map
\[ t: \Mbar_{g,n}(\CX,\chi) \to \Mbar_{g,n}(\CX,\chi)^0, \quad (f,p_1,\ldots,p_n) \mapsto (t_{-f(p_1)} \circ f, p_1,\ldots,p_n) \]
which gives rise to the product decomposition
\[ \Mbar_{g,n}(\CX,\chi) = \Mbar_{g,n}(\CX,\chi)^0 \times_{\CA_h} \CX. \]
There is a fiber diagram
\[
\begin{tikzcd}
	\Mbar_{g,n}(\CX,\chi) \ar{r}{\ev} \ar{d}{t} & \CX^{\times n} \ar{d}{\xi}  \\
	\Mbar_{g,n}(\CX,\chi)^0 \ar{r}{\ev} & \CX^{n-1}
\end{tikzcd}
\]
where $\xi(a_1,\ldots,a_n) = (a_2-a_1, \ldots,a_n-a_1)$.
The virtual class is pulled back by $t$: 
\[ \exists \alpha \in \Chow(\Mbar_{g,n}(\CX,\chi)^0) : t^{\ast}(\alpha) = [\Mbar_{g,n}(\CX,\chi)]^{\vir}. \]

For $g \geq 2$ consider the morphism
\[ p_n : \Mbar_{g,n}(\CX,\chi) \to \Mbar_{g}(\CX,\chi)/\CX. \]
Following \cite{BOPY} we define the virtual class on the quotient by
\begin{equation} \label{virtual class quotient}
[ \Mbar_{g}(\CX,\chi)/\CX ]^{\vir} := \frac{1}{2g-2} p_{1 \ast}( \psi_1 \ev_1^{-1}(\sigma) \cap [\Mbar_{g,1}(\CX,\chi)]^{\vir}).
\end{equation}
By the same argument as in \cite[Lemma 17]{BOPY} we have
\begin{equation} \label{relation of virtual classes}
p_n^{\ast} [\Mbar_{g,n}(\CX,\chi)/\CX]^{\vir} = [\Mbar_{g,n}(\CX,\chi)]^{\vir}.
\end{equation}

The above leads to the following vanishing:
\begin{prop} \label{prop:vanishing in low multiplicity}
	Let $\Gamma \in \Chow^{\ast}(\CX^{\times n})$ be an eigenvector with $\mult(\Gamma) < 2h$. Then $\CC_{g,\chi}(\Gamma) = 0$.
\end{prop}
\begin{proof}
	By stability we assume $2g-2+n>0$. The case $g=0$ can be treated directly, see Section~\ref{subsec:genus zero class}.
	If $n = 0$, then $g \geq 2$ and we get $\CC_{g,\chi}() = (\tau \times \rho)_{\ast}( p_0^{\ast} [\Mbar_{g,n}(\CX,\chi)/\CX]^{\vir} ) = 0$ since $\rho$ factors through $p_0$ which is of relative dimension $h$.
	If $n>0$, then we have
	\begin{align*}
		\CC_{g,\chi}(\Gamma) 
		& = (\tau \times \rho)_{\ast}( \ev^{\ast}(\Gamma) \cap t^{\ast}(\alpha)) \\
		& = (\tau^0 \times \rho^0)_{\ast}( t_{\ast} \ev^{\ast}(\Gamma) \cap \alpha ) \\
		& = (\tau^0 \times \rho^0)_{\ast}( \ev^{\ast}(\xi_{\ast}(\Gamma)) \cap \alpha ),
	\end{align*}
	where we have written $\tau^0, \rho^0$ for the restriction of $\tau,\rho$ to $\Mbar_{g,n}(\CX,\chi)^0$.
	If $\mult(\Gamma) = \gamma$, then by $[m]_{\ast} [m]^{\ast} = m^{2hn} \id$ we get $[m]_{\ast} \Gamma = m^{2hn-\gamma} \Gamma$. Since $[m] \circ \xi = \xi \circ [m]$ (with $[m]$ the diagonal multiplication), we get
	$[m]_{\ast}(\xi_{\ast} \Gamma) = m^{2hn-\gamma} \xi_{\ast} \Gamma$, and hence
	$[m]^{\ast}( \xi_{\ast}\Gamma) = m^{\gamma - 2h} \xi_{\ast}(\Gamma)$. In particular, if $\mult(\Gamma) < 2h$, then $\xi_{\ast}(\Gamma) = 0$ by the motivic decomposition \eqref{motivic decomposition}. The claim follows.
\end{proof}

\begin{lemma} \label{lemma:vanishing1b}
	%Let $\Gamma \in R_{\mathsf{vert}}^{\ast}(\CX^{\times n})$.
	Let $\Gamma \in \Chow^{\ast}(\CX^{\times n})$.
	\begin{enumerate}
		\item[(i)] $\CC_{1,\chi}(\tau_{k_1 \ldots k_n}(\Gamma)) = 0$ if
		$\deg(\Gamma) + \sum_i k_i < h+n-1$.
		\item[(ii)] $\CC_{g, \chi}(\tau_{k_1 \ldots k_n}(\Gamma))= 0$ if
		$\deg(\Gamma) + \sum_i k_i < h+n$
		and $g \geq 2$
	\end{enumerate}
\end{lemma}
\begin{proof}
	%Since $\Gamma$ is a vertical tautological class we have $\deg(\Gamma) = \frac{1}{2} \mult(\Gamma)$,
	%so the case $n=0$ is a special case of the previous Lemma. 
	Assume first $g \geq 2$. We have
	\[ \CC_g(\tau_{k_1\ldots k_n}(\Gamma)) 
	= 
	\rho^0_{\ast}( p_{n \ast}( \prod_i \psi_i^{k_i} \ev^{\ast}(\Gamma)) \cap [\Mbar_{g}(\CX,\chi)/\CX]^{\vir} ). \]
	However, since $p_n$ is of relative dimension $h+n$ the pushforward vanishes if $\sum_i k_i + \deg(\Gamma) < h+n$.
	The case for $g=1$ is parallel, but here we pushforward to $\Mbar_{1,1}(\CX,d)^0$.
\end{proof}

\begin{cor} \label{cor:Vanishing for g >= h >= 4}
Let $g \geq h \geq 4$. Then 
$\taut\ \CC_{g,\chi}(\tau_{k_1 \ldots k_n}(\Gamma)) = 0$
for all $\Gamma,k_1,\ldots,k_n$.
\end{cor}
\begin{proof}
	By Lemma~\ref{lemma:vanishing1b} we can assume that $\deg(\Gamma) + \sum_i k_i \geq h+n$.
	Then $\CC_{g,\chi}(\tau_{k_1\ldots k_n}(\Gamma)) \in \Chow(\CA_h)$ is of dimension
	\[ \dim(\CA_h) + (h-3)(1-g) + n - \deg(\Gamma) - \sum_i k_i \]
	so of codimension
	\[ \deg(\Gamma) + \sum_i k_i -n -h + g (h-3) + 3 \geq g (h-3) + 3. \]
The top degree of the tautological ring of $\CA_h$ is $(h-1)h/2$ and
by an elementary check
\[ g (h-3) + 3 > \frac{(h-1)h}{2} \]
if $g,h \geq 4$,
so the tautological projection vanishes.
\end{proof}

\begin{rmk}
In fact, for $g \geq h$ by dimension reasons we have $\CC_{g,\chi}(\tau_{k_1 \ldots k_n}(\Gamma)) = 0$ if $h \geq 7$, and using the results of \cite[Corollary 3]{MR2946823} for $\CA_4$ and Borel-Serre \cite{MR387495} for $\CA_5$ one has the vanishing of $\CC_{g,\chi}(\tau_{k_1 \ldots k_n}(\Gamma))$ in cohomology for $h=4,5,6$.
\end{rmk}

The group structure on $\CX$ can be used to obtain relations among the Gromov-Witten invariants.
We record them here for convenience.
Parallel relations appeared in the study of the elliptic curve \cite{OP_VirCurves,HAE} or in abelian surfaces in \cite{BOPY}. Consider again the morphism
\[ \xi : \CX^{\times n} \to \CX^{n-1}, \quad \xi(x_1,\ldots,x_n) = (x_2-x_1, \ldots, x_n-x_1). \]
\begin{prop}(Abelian vanishing relation) Let $g \geq 2$ and $\alpha = (\alpha_1,\ldots,\alpha_n)$.
	Let $\Gamma_0 \in R_{\mathrm{vert}}^{\ast}(\CX^{n})$ with $\deg_{\BC}(\Gamma_0) < h$,
	and let $\Gamma_1 \in R_{\mathrm{vert}}^{\ast}(\CX^{n-1})$.
	Then
	\[ \CC_{g,\chi}( \tau_{\alpha}( \Gamma_0 \cdot \xi^{\ast}(\Gamma_1) ) ) = 0. \]
\end{prop}
\begin{proof}
	The same proof as in \cite[Lemma 4]{BOPY} applies.
\end{proof}

\subsection{Definition of product Gromov-Witten classes} \label{subsec:defn of product GW classes}
We define here the product of two Gromov-Witten classes
as used in the statement of the holomorphic anomaly equation (Conjecture~\ref{conj:HAE}).
If one is not used in the application of the HAE this section can be skipped.

For a decomposition $\{ 1, \ldots, n \} = I_1 \sqcup I_2$ and $d=d_1 + d_2$ consider the evaluation map
\[ \ev := \ev_1 \times \cdots \times \ev_n : \Mbar_{g_1,I_1}(\pi, d_1) \times_{\CA_h} \Mbar_{g_2,I_2}(\pi, d_2) \to \CX^{\times n} \]
where if $i \in I_k$ we let $\ev_i$ be the composition of the projection to the $k$-th factor and the evaluation map at the $i$-th marking.
Consider also the projection
\[ \rho :  \Mbar_{g_1,I_1}(\pi, d_1) \times_{\CA_h} \Mbar_{g_2,I_2}(\pi, d_2) \to \Mbar_{g_1,I_1} \times \Mbar_{g_2,I_2} \times \CA_h. \]
We then define for $\Gamma \in \Chow^{\ast}(\CX^{\times n})$ the class
\[
\CC_{g_1,I_1,d_1} \times \CC_{g_2,I_2,d_2}(\Gamma) 
=
\rho_{\ast}\left( 
\ev^{\ast}(\Gamma) \cdot [ \Mbar_{g_1,I_1}(\pi, d_1) \times_{\CA_h} \Mbar_{g_2,I_2}(\pi, d_2) ]^{\vir} \right)
\]
which is an element in $\Chow^{\ast}(\Mbar_{g_1,I_1} \times \Mbar_{g_2,I_2} \times \CA_h)$.
Here the virtual class is defined by
\[ [ \Mbar_{g_1,I_1}(\pi, d_1) \times_{\CA_h} \Mbar_{g_2,I_2}(\pi, d_2) ]^{\vir}
=
\Delta_{\CA_h}^{!} \left( 
[ \Mbar_{g_1,I_1}(\pi, d_1) ]^{\vir}  \times [\Mbar_{g_2,I_2}(\pi, d_2) ]^{\vir} \right)
\]
where $\Delta_{\CA_h} : \CA_h \to \CA_{h}\times \CA_h$ is the diagonal morphism.
%where we let
%\begin{itemize}
%\item $\rho :  \Mbar_{g_1,I_1}(\pi, d_1) \times_{\CA_h} \Mbar_{g_2,I_2}(\pi, d_2) \to \Mbar_{g_1,I_1} \times \Mbar_{g_2,I_2} \times \CA_h$
%be the forgetful map,
%\item the virtual class is defined by
%\end{itemize}
We form the generating series
\[
\CC_{g_1,I_1} \times \CC_{g_2,I_2}(\Gamma) 
=
\sum_{d_1, d_2 \geq 0} q^{d_1 + d_2} \CC_{g_1,I_1,d_1} \times \CC_{g_2,I_2,d_2}(\Gamma).
\]

\subsection{Proof of Corollary~\ref{cor:weight statement as a consequence of HAE}}
The proof is exactly the same as for a fixed elliptic curve,
see \cite[Sec.2.6]{HAE} except that we use a theorem of K\"unnemann that
\[ [e_{\theta}, f_{\theta}] = \sum_{i=0}^{2h} (i-h) \pi_i \]
where 
\[ \Delta = \sum_{i=0}^{2h} \pi_i \in \Chow^{\ast}(\CX \times_B \CX) \]
is the unique decomposition of the class of the diagonal
such that $(\id_{\CX} \times_B [m])^{\ast}(\pi_i) = m^i \pi_i$ for all $m \in \BZ$.
The $\pi_i$ are the projections on the subspaces of eigenvectors with multiplicity $i$.
In particular, $\Gamma$ is an eigenvector for $[e_{\theta}, f_{\theta}]$ with eigenvalue $\wt(\Gamma)$
if and only if it is a simultaneous eigenvector for $[m]$ with eigenvalue $m^{\mult(\Gamma)}$,
where $\mult(\Gamma) = \wt(\Gamma) + h$.
\qed
%
%For any $m \in \BZ$, let $[m] : \CX \to \CX$ be the multiplication by $m$.
%By \cite[Theorem 3.1]{DM} there exists a unique decomposition
%of the class of the diagonal
%\[ \Delta = \sum_{i=0}^{2h} \pi_i \in \Chow^{\ast}(\CX \times_B \CX) \]
%such that $(\id_{\CX} \times_B [m])^{\ast}(\pi_i) = m^i \pi_i$ for all $m \in \BZ$.
%The $\pi_i$ are the projections on the subspaces of eigenvectors with multiplicity $i$.
%The following result of K\"unnemann relates the Lefschetz action and the decomposition of Deninger-Murre:
%\begin{thm}[{\cite{Kuennemann}}]
%	$[e_{\theta}, f_{\theta}] = \sum_{i=0}^{2h} (i-h) \pi_i.$
%\end{thm}

\section{Quotient Gromov-Witten invariants}
\label{sec:quotient GW invariants}
Assume that $g \geq 2$ throughout.
Let $\chi = t^h - e_1 t^{h-1} + \ldots + (-1)^h e_h$ be a polynomial.
Recall the quotient Gromov-Witten invariants
\[ N_{g,\chi} := \rho_{\ast}([ \Mbar_{g}(\CX,\chi)/\CX ]^{\vir}) \in \Chow^{\ast}(\CA_h). \]
The goal of this section is to show that the descendent Gromov-Witten invariants are determined by the quotient invariants if a certain degree condition is satisfied.
Conversely, we show that the quotient invariants are determined by the (unrefined) Gromov-Witten invariants.
We then apply this to derive the modular properties of the quotient invariants.

\subsection{Reconstruction results}
We show here that the Gromov-Witten invariants
can be effectivly reconstructed from the quotient invariants $N_{g,\chi}^{\CA_h}$.
We begin with a case which is a bit easier and where we make the reconstruction explicit. 
The general case follows afterwards.

\begin{lemma}(Point descendent reconstruction) \label{lemma:tau1 sigma reduce to quotient}
	Let $\Gamma \in \Rvert(\CX^{\times n})$ be a tautological class such that
	$\Gamma|_{X^n} \in H^2(X)^{\otimes n}$.
	For all $g \geq 2$ we have
	\[ \CC_{g,\chi}\left( \tau_1(\sigma) \tau_{0^n}(\Gamma) \right)
	=
	(2g-2) Q_{\Gamma}(\chi) N_{g,\chi}
	\]
	where $Q_{\Gamma}$ is the unique polynomial
	such that for all classes $\beta \in H_2(X,\BZ)$
	with characteristic polynomial $\chi_{\beta} = t^h - e_1 t^{h-1} + \ldots + (-1)^h e_h$ we have 
	\[ \int_{\beta^{\times n}} \Gamma = Q_{\Gamma}(e_1,\ldots, e_h). \]
	%where we write $Q_{\Gamma}(\chi) := Q_{\Gamma}(a_1,\ldots,a_h)$ as before if $\chi = t^h - a_1 t^{h-1} + \ldots + (-1)^h a_h$
\end{lemma}
\begin{proof}
	This is just a special case of the divisor equation.
	We index the markings by $0,1,\ldots,n$. Consider the morphism
	$P : \Mbar_{g,n+1}(\CX,\chi) \to \Mbar_{g,1}(\CX,\chi)$.
	We have
	\[ \psi_0 = P^{\ast}(\psi_0) + \sum_{\substack{S \subset \{ 0, \ldots, n \} \\ 0 \in S, |S| \geq 2}} D_S \]
	where $D_S$ is the virtual Cartier divisor parametrizing stable maps with rational tails carrying the markings labeled by $S$ and the degree of the stable map on the rational tail being zero (which is automatic here).
	Since $\sigma_0 \Gamma_{1,\ldots n} \Delta_{0i} = 0$ by degree reasons for all $i > 0$ we find that
	\[ \ev^{\ast}(\sigma_0 \Gamma_{1 \cdots n}) \cap D_S \cap [\Mbar_{g,n+1}(\CX,\chi) ]^{\vir} = 0.\]
	Thus 
	\[
	\CC_{g,\chi}( \tau_1(\sigma) \tau_{0^n}(\Gamma) )
	=
	\int_{[\Mbar_{g,1}(\CX,\chi)]^{\vir}} \psi_0 \ev_0^{\ast}(\sigma) P_{\ast}(\ev_{1 \ldots n}^{\ast}(\Gamma)).
	\]
	Since $P_{\ast}(\ev_{1 \ldots n}^{\ast}(\Gamma))$ is of degree $0$ it can be evaluated by integrating over the fiber over a point $x=(f : C \to X,p_1)$.
	Let $\beta = f_{\ast}([C])$. Then the evaluation map $\ev_{1 \ldots n} : P^{-1}(x) \to X^{n}$ factors as the birational map $q : P^{-1}((f,p_1)) \to C^{n}$ and followed by $f^{\times n} : C^n \to X^n$, so the integral over the fiber is
	\[
	P_{\ast}(\ev_{1 \ldots n}^{\ast}(\Gamma))|_{x} = 
	\int_{C \times \cdots \times C} (f,\ldots,f)^{\ast}(\Gamma) = \int_{\beta^{\times n}} \Gamma = Q_{\Gamma}(\chi_{\beta}) = Q_{\Gamma}(\chi).
	\]
	The claim hence follows from the case $n=0$ which is the definition of the quotient invariants.
	%Essentially by definition we have
	%\[
	%\int_{[\Mbar_{g,1}(\CX,\chi)]^{\vir}} \psi_0 \ev_0^{\ast}(\sigma) = \langle \tau_1(\sigma) 
	%\]
	%
	%We argue as in Proposition~\ref{prop:reconstruction primary} just that now we need to integrate
	%\[
	%\int_{p_n^{-1}([f])} \psi_1 \alpha^{\ast} \ev^{\ast}(\sigma_1 \Gamma_{2\ldots n}).
	%\]
	%The fiber $p_n^{-1}([f])$ is isomorphic to $X \times C[n]$ where $C[n]$ is the fiber of $[C] \in \Mbar_g$ under the forgetfulmorphism $\Mbar_{g,n}\to \Mbar_g$. This is because all maps from rational curves to $X$ are constant, so the curve $C$ is stable.
	%Let $\pi_i : C[n] \to C$ be the maps forgetting all but the $i$-th marking and let $\pi  = \pi_1 \times \ldots \times \pi_n : C[n] \to C^n$.
	%We have the basic relation on $C[n]$:
	%\[ \psi_1 = \pi^{\ast}( K_1 ) + \sum_{\substack{ S \subset \{ 1, \ldots, n \} \\ 1 \in S, |S| \geq 2}} \]
	%where $D_S$ is the boundary divisor parametrizing degenerations of $C$ with a rational tail carrying the markings $S$
	%and $K = c_1(\omega_C)$ is the canonical class of the curve (with $K_1 = \pr_1^{\ast}(K)$ as usual).
	%Pushing forward by $\pi$ we obtain
	%\[ \pi_{\ast}(\psi_1) = K_1 + \sum_{i=2}^{n} \Delta_{1i}. \]
	%Pulling back to $p_n^{-1}([f])$ we find
	%\[ q_{\ast}(\psi_1) = K_1 + \sum_{i=2}^{n} \Delta_{1i}. \]
	%
	%Assume first that $C$ is smooth. Then $p_n^{-1}([f])$ is just the Fulton-Macpherson compatification of $C$.
\end{proof}

For any fixed principally polarized abelian variety $(X,\theta)$ consider the map
$\varphi : X^{n+1} \to X^n$ defined by 
\[ \varphi(x_0,\ldots,x_n) = (x_0+x_1, x_0+x_2, \ldots, x_0 + x_n). \]

\begin{thm} \label{thm:reconstruction primary}
For any $g \geq 2$, $h \geq 3$, $k_1,\ldots,k_n \geq 0$ and $\Gamma \in \Rvert(\CX^{\times n})$ satisfying $\deg(\Gamma) + \sum_i k_i = h+n$
there exists a polynomial $R_{g,\Gamma,k_1,\ldots,k_n}(z_1,\ldots,z_h)$ of weighted degree $n-\sum_i k_i$ with $\deg z_i = i$, which can be canonically constructed from the intersection theory of the moduli space of curves and the invariant theory of the symplectic group, such that
\[ \CC_{g,\chi}(\tau_{k_1 \ldots k_n}(\Gamma)) = R_{g,\Gamma,k_1,\ldots,k_n}(e_1,\ldots,e_h) N_{g,\chi}. \]

(Primary reconstruction) If $k_i=0$ for all $i$, $R_{g,\Gamma,0,\ldots,0}$ is the unique polynomial such that for all $\beta \in H_2(X,\BZ)$ we have
\[ \int_{[X]\times \beta \times \ldots \times \beta} \varphi^{\ast}(\Gamma) = R_{g,\Gamma,0,\ldots,0}(e_1,\ldots,e_h) \]
where $\chi_{\beta} = t^h - e_1 t^{h-1} + \ldots + (-1)^h e_h$ is the characteristic polynomial of $\beta$.
\end{thm}

\begin{proof}
The case $k_1=\ldots =k_n =0$ is essentially treated in \cite[Lemma 18]{BOPY};
we give a refinement of that argument.
	By relation \eqref{relation of virtual classes} 
	we have
	\[ \CC_{g,\chi}(\tau_{k_1\cdots k_n}(\Gamma)) =
	\rho_{\ast} p_{n \ast} \left(\prod_i \psi_i^{k_i} \ev^{\ast}(\Gamma) \right). \]
where $p_n : \Mbar_{g,n}(\CX,\chi) \to \Mbar_{g}(\CX,\chi)/\CX$ 
is the forgetful map followed by the quotient map.
	Since $\deg(\Gamma)+\sum_i k_i$ equals the relative dimension $h+n$ of $p_n$, the class $p_{n \ast} ( \prod_i \psi_i^{k_i} \ev^{\ast}(\Gamma))$ is of degree $0$. So to compute it, it suffices to integrate over a fiber.
	
	Let $[f : C \to X] \in \Mbar_{g}(\CX,\chi)$ which also defines a $\BC$-valued point in $\Mbar_{g}(\CX,\chi)/\CX$.
	We have a fiber diagram
	\[
	\begin{tikzcd}
		p_n^{-1}([f]) \ar{r}{\alpha} \ar{d}{\gamma} & \Mbar_{g.n}(\CX,\chi) \ar{r}{\ev} \ar{d} & \CX^{\times n} \\
		X \ar{r}{t} \ar{d} & \Mbar_{g}(\CX,\chi) \ar{d} \\
		\Spec(\BC) \ar{r}{[f]} & \Mbar_{g}(\CX,\chi)/\CX
	\end{tikzcd}
	\]
	where $t(a) = [t_a \circ f]$.
	
	The fiber product 
	$p_n^{-1}([f])$ parametrizes pairs $(a, [\tilde{f} : \tilde{C}\to X, p_1,\ldots,p_n])$ satisfying
	$t_a \circ f \circ \nu = \tilde{f}$, where $\nu : \tilde{C}\to C$ is the stabilization map that contracts the rational components which become unstable after forgetting the markings $p_1,\ldots,p_n$.
	The fiber $p_n^{-1}([f])$ can be described as follows.
	Since $X$ does not admit non-constant maps from rational curves, the curve $C$ is stable, 
	and the tuple $[\tilde{C},p_1,\ldots,p_n]$ defines a $n$-marked stable curve.
	Let $P_n : \Mbar_{g,n} \to \Mbar_g$ be the forgetful morphism on the moduli space of stable curves and let $C[n] = P_n^{-1}([C])$ be the fiber over $C$.
	Then the natural map
\[ p_n^{-1}([f]) \to X \times C[n] \]
given by sending $( [\tilde{f} : \tilde{C}\to X, p_1,\ldots,p_n], a)$ to 
$(a,[\tilde{C}, p_1,\ldots,p_n])$ is an isomorphism
with inverse $([\tilde{C}, p_1,\ldots,p_n], a) \mapsto (a,[t_a \circ f \circ \nu : \tilde{C}, p_1,\ldots,p_n])$.

Let $q_i : C[n] \to C[1] \cong C$ be the morphism that forgets all but the $i$-th marking.
Let $q = (q_1,\ldots,q_n) : C[n] \to C^n$.
Concretely, $q([\tilde{C},p_1,\ldots,p_n]) = (\nu(p_1), \ldots,\nu(p_n))$.
Then $q$ is birational.
By the standard intersection theory of $\psi$-classes on the moduli space of curves
\begin{equation} \label{143rwef}
q_{\ast}\left( \prod_{i=1}^{n} \psi_i^{k_i} \right) 
=
\sum_{\CP} \sum_{I \subset \{ 1, \ldots, \ell(\CP)\}} c_{\CP,I} \Delta_{\CP \ast}\left( \prod_{i \in I} \kappa_i \right) 
\end{equation}
where
\begin{itemize}
\item $\CP$ runs over set partitions $(\CP_{1}, \ldots, \CP_{\ell})$ of $\{ 1, \ldots, n \}$ with $\ell := \ell(\CP)$ many parts,
\item $\Delta_{\CP} : C^{\ell} \to C^n$ is the associated diagonal morphism,
\item $\kappa=c_1(\omega_C)$ and $\kappa_i = \pr_i^{\ast}(\kappa)$,
\item $c_{\CP,I} \in \BQ$ are certain universal combinatorial coefficients (which only depend on $(k_1,\ldots,k_n)$).
\end{itemize}
Note that each summand in the right hand side of \eqref{143rwef} is of dimension $\ell - |I|$.
Comparing with the dimension of $q_{\ast}\left( \prod_{i=1}^{n} \psi_i^{k_i} \right)$ we must have
\[ n-\sum_i k_i = \ell - |I|. \]

The map $\ev \circ \alpha$ sends the tuple $(a, [\tilde{f} : \tilde{C} \to X, p_1,\ldots,p_n])$ to $(a+f(\nu(p_1)), \ldots, a+f(\nu(p_n))$. Hence it factors as
\[ p_n^{-1}([f]) \cong X \times C[n] \xrightarrow{\id_X \times q} X \times C^n \xrightarrow{(\id,f,\dots,f)} X^{n+1} \xrightarrow{\varphi}
X^{n} \]

Hence we find
\begin{align*}
& \int_{[ p_n^{-1}([f]) ]} \prod_i \psi_i^{k_i} \cdot \alpha^{\ast} \ev^{\ast}(\Gamma)  \\
& = 
\int_{[ X \times C[n] ]} \prod_i \psi_i^{k_i} \cdot (\id \times q)^{\ast} (\id,f^{\times n})^{\ast} \varphi^{\ast}(\Gamma) \\
& = 
\int_{X\times X^n} (\id, f^{\times n})_{\ast} (\id \times q)_{\ast}\left( \prod_i \psi_i^{k_i} \right)  \varphi^{\ast}(\Gamma) \\
& = 
\sum_{\CP} \sum_{I} c_{\CP,I} \int_{X\times X^n} \pr_{1 \ldots n}^{\ast} f^{\times n}_{\ast} \Delta_{\CP \ast}\left( \prod_{i \in I} \kappa_i \right) 
\varphi^{\ast}(\Gamma) \\
& = 
\sum_{\CP} \sum_{I} c_{\CP,I} (2g-2)^{|I|} \int_{X\times X^n}  (\id_{X} \times \Delta^{X}_{\CP})_{\ast}\left( \prod_{i \in I} \sigma_i \prod_{i \in \{1 ,\ldots, \ell \} \setminus I} \beta_i \right) 
 \varphi^{\ast}(\Gamma) \\
 & = 
\sum_{\CP} \sum_{I} c_{\CP,I} (2g-2)^{|I|} \int_{X\times X^\ell}
\prod_{i \in I} \sigma_i \prod_{i \in \{1 ,\ldots,  \ell \} \setminus I} \beta_i
\cdot 
(\id_{X} \times \Delta^{X}_{\CP})^{\ast} \varphi^{\ast}(\Gamma)  \\
& = 
\sum_{\CP} \sum_{I} c_{\CP,I} (2g-2)^{|I|}
\int_{\beta^{\ell - |I|}} \pr_{\{1 ,\ldots, \ell \} \setminus I, \ast}\left( 
(\id_{X} \times \Delta^{X}_{\CP})^{\ast} \varphi^{\ast}(\Gamma) \prod_{i \in I} \sigma_i \right)
\end{align*}
where we used that $f_{\ast}(\kappa) = (2g-2) \sigma$ in cohomology. The integrand 
\begin{equation} \label{vfsdf354} 
\pr_{\{1 ,\ldots, \ell \} \setminus I, \ast}\left( 
(\id_{X} \times \Delta^{X}_{\CP})^{\ast} \varphi^{\ast}(\Gamma) \prod_{i \in I} \sigma_i \right)
%	\pr_{\{1 ,\ldots, n\} \setminus I, \ast}\left( 
%\varphi^{\ast}(\Gamma) \prod_{i \in I} \sigma_i \right) 
\end{equation} is again tautological,
so by Lemma~\ref{lemma:reconstruction}
this last expression is given by 
a polynomial $R_{\CP,I}(e_1,\ldots,e_h)$ of degree $\ell - |I| = n-\sum_i k_i$ in the coefficients of the characteristic polynomial of $\beta$. The polynomial depends here only on \eqref{vfsdf354}.
Setting $R_{g,\Gamma,k_1,\ldots,k_n} = \sum_{\CP,I}c_{\CP,I} (2g-2)^{|I|} R_{\CP,I}$ the claim follows.

In the special case $k_1 = \ldots k_n=0$ we have (since $q$ is birational)
\[ q_{\ast}(1) = [C^n] \]
so we obtain
\[
\int_{[ p_n^{-1}([f]) ]} \alpha^{\ast} \ev^{\ast}(\Gamma) =
\int_{[X] \times \beta^{\times n}} \varphi^{\ast}(\Gamma).
\]
\end{proof}

\subsection{Quasimodularity for quotient invariants}
An $r$-cycle was defined to be the class $\Gamma_r = \eta_{12} \eta_{23} \cdots \eta_{r-1, r} \eta_{r1} \in \Rvert(\CX^n)$.
\begin{cor} \label{cor:pk formulation}
	Let $r_1,\ldots,r_s \geq 1$. Then we have
	\[ \CC_{g,\chi}\left( \tau_1(\sigma) \tau_{0^{r_1}}(\Gamma_{r_1}) \cdots \tau_{0_{r_s}}(\Gamma_{r_s})  \right) 
	=
	(2g-2) \left( - \frac{1}{2} \right)^{\sum_i (r_i-1)} p_{r_1} \cdots p_{r_s} N_{g,\chi}^{\CA_{h}}. 
	\]
where $p_k$ are the power sum polynomials written as a polynomial in the coefficients $e_i$ of the characteristic polynomial $\chi$.
\end{cor}
\begin{proof}
This follows immediately from Lemma~\ref{lemma:tau1 sigma reduce to quotient} and Lemma~\ref{lemma:intersection computation}.
\end{proof}

Given a polynomial $P$ in variables $e_1,\ldots,e_h$, we let
\[ F^{\CA_h}_g(P(e_1,\ldots,e_h)) = \sum_{\chi} P(e_1,\ldots,e_h) N_{g,\chi} q^{e_1} \]
where we sum over all characteristic polynomials $\chi = t^h - e_1 t^{h-1} + \ldots + e_h (-1)^h$.

For example, $P$ will be below a power sum symmetric function $p_k$ which we view as a polynomial in the elementary symmetric polynomial $e_i$.

We often drop the superscript $\CA_h$ and we will also write $F_g$ for $F_g(1)$.

\begin{cor}
	If Conjectures~\ref{conj:quasimodularity} and~\ref{conj:HAE} hold, then we have
	\[
	F_{g}(e_{r_1} \cdots e_{r_s}) \in \QMod_{2gh + 2 \sum_{i=1}^{h} r_i} \otimes \Chow^{\ast}(\CA_h).
	\]
\end{cor}
\begin{proof}
By Corollary~\ref{cor:pk formulation} we have
	\begin{equation} \label{CG and quotient}
	\CC_{g}( \tau_1(\sigma) \tau_{0^{r_1}}(\Gamma_{r_1}) \cdots \tau_{0^{r_s}}(\Gamma_{r_s})) 
	= 
	(2g-2) \left( - \frac{1}{2} \right)^{\sum_i (r_i-1)} F_{g}(p_{r_1} \cdots p_{r_s}).
	\end{equation}
	By Conjecture~\ref{conj:quasimodularity}
	the left hand side is 
	%\[ F_{g}( \tau_1(\sigma) \tau_{0^{r_1}}(\Gamma_{r_1}) \cdots \tau_{0^{r_s}}(\Gamma_{r_s})) \]
	a quasimodular of weight $2gh + \sum_i 2 r_i$.
	The power sum polynomials are ring generators of the ring of symmetric function,
	so any polynomial in elementary symmetric polynomial is a polynomial in the $p_k$.
	Thus any quotient invariant is quasimodular.
	Since $p_r$ is of weighted degree $r$ in the $e_1,\ldots,e_h$ where $e_i$ has weight $i$, the weight statement follows as a consequence of Corollary~\ref{cor:weight statement as a consequence of HAE}.
\end{proof}

\begin{cor} \label{cor:converse modularity}
Assume that $F_{g}(e_{r_1} \cdots e_{r_s}) \in \QMod_{2gh + 2 \sum_{i=1}^{h} r_i}$ for all $s,r_1,\ldots,r_s$.
Then for all $\Gamma \in \Rvert(\CX^{\times n})$ satisfying
$\deg(\Gamma) + \sum_i k_i = h+n$ we have
\[ \CC_{g}(\tau_{k_1 \ldots k_n}(\Gamma)) \in \QMod_{(2g-2)h + \mult(\Gamma)} \otimes \Chow^{\ast}(\CA_h). \]
\end{cor}
\begin{proof}
By Theorem~\ref{thm:reconstruction primary}
we have
\[ 
\CC_{g}(\tau_{k_1 \ldots k_n}(\Gamma)) = F_g( R_{\Gamma,k_1,\ldots,k_n} )
\]
where the polynomial 
$R_{\Gamma,k_1,\ldots,k_n}$ is of degree $n-\sum_i k_i = \deg(\Gamma) - h$.
By the assumption the right hand side is quasimodular of weight
\[ 2gh + 2(\deg(\Gamma) - h) = (2g-2)h + \mult(\Gamma) \]
where $\mult(\Gamma) = 2 \deg(\Gamma)$ since $\Gamma$ is tautological.
\end{proof}

\begin{rmk}
Using the Vandermonde determinant the invariants $N_{g,\chi}^{\CA_h}$ are uniquely determined from the collection of series $F_g^{\CA_h}(P)$.
Hence \eqref{CG and quotient} shows that the (unrefined) descendent Gromov-Witten invariants $\CC_{g,d}(\tau_{k_1 \cdots k_n}(\Gamma))$ determine the quotient invariants.
\end{rmk}

\subsection{Holomorphic anomaly equation for quotient invariants}
The aim of this part is to prove the following result:

\begin{thm} \label{thm:hae for quotient text}
	Assume that $h \geq 3$ and $g \geq 2$. If Conjecture~\ref{conj:HAE} holds, then:
	\[
	\frac{d}{d G_2}F_g( e_{r_1} \cdots e_{r_s} )
	=
	2 \delta_{h=3} F_{g-1}( e_2 e_{r_1} \cdots e_{r_s}) 
	+ F_g( \CL_{g,h} e_{r_1} \cdots e_{r_s} )
	\]
	where $\CL_{g,h}$ is the following differential operator on the ring $\BC[e_1,\ldots,e_h]$, and where we set $e_0=1$ and $e_k = 0$ if $k \notin \{ 0, \ldots, h \}$,
	\begin{equation*}
		\begin{aligned}
			\CL_{g,h}
			={}&-2\sum_{i,j=1}^h
			\left(\sum_{a=0}^{\min(i,j)-1}(i+j-1-2a)e_a e_{i+j-1-a}\right)
			\frac{d}{d e_i} \frac{d}{d e_j} \\
			&\quad -4\sum_{i=1}^h (g-i+1)(h-i+1)e_{i-1} \frac{d}{de_i} .
		\end{aligned}
	\end{equation*}
\end{thm}

\begin{example}
	Modulo $e_3 = e_4 = \ldots = 0$ we have
	\[ \CL_{2,h}(e_2^r) = -2 r (2h-3+r) e_1 e_2^{r-1} \]
	where in the $r=0$ case the terms with negative exponents are defined to vanish. The same convention applies below. Thus we find
	for $h \geq 3$:
	\[
	\frac{d}{d G_2} F_2(e_2^r) = -2 r (2h-3+r) F_2(e_1 e_2^{r-1})
	\]
	In particular, all $F_2^{\CA_h}(1)$ are modular forms of weight $4h$ for all $h \geq 3$.
\end{example}

\begin{example} We consider $g=3$ and $h \geq 4$. Then we get
	\[ \frac{d}{d G_2} F_3( e_2^r e_3^s)
	=
	-2 r (r + 4s + 4h - 5) F_3( e_1 e_2^{r-1} e_3^s )
	-2 s (s + 2h - 5) F_3(e_2^{r+1} e_3^{s-1} )
	- 6 r (r-1) F_3(e_2^{r-2} e_3^{s+1})
	\]
	In particular, $F_3^{\CA_h}(1)$ is a modular form of weight $6h$ for $h \geq 4$.
\end{example}
\begin{example}
	For $h=3$ and setting $e_4 = e_5 = \ldots = 0$ we get for any $g \geq 2$:
	%\begin{align*}
	\[ \CL_{g,3}(e_2^r e_3^s)
	=
	-2 r (r + 4s + 4g-5) e_1 e_2^{r-1} e_3^s
	- 2 s (s + 2g-5) e_2^{r+1} e_3^{s-1}
	- 6 r (r-1) e_2^{r-2} e_3^{s+1}
	\]
	So we find for $h=3$:
	\begin{multline*}
		\frac{d}{d G_2} F_{g}^{\CA_3}(e_2^{r} e_3^s)
		=
		2 F_{g-1}^{\CA_3}(e_2^{r+1} e_3^s) \\
		-2 r (r + 4s + 4g-5) F_g^{\CA_3}(e_1 e_2^{r-1} e_3^s)
		- 2 s (s + 2g-5) F_g^{\CA_3}(e_2^{r+1} e_3^{s-1})
		- 6 r (r-1) F_g^{\CA_3}(e_2^{r-2} e_3^{s+1})
	\end{multline*}
\end{example}

\subsection{Proof of Theorem~\ref{thm:hae for quotient text}}
We consider the HAE (Conjecture~\ref{conj:HAE}) for the series
\[ \CC_{g}( \tau_1(\sigma) \tau_{0}(\Gamma_{r_1}) \cdots \tau_{0}(\Gamma_{r_s}))
=
(2g-2) \left( - \frac{1}{2} \right)^{\sum_i (r_i-1)} F_{g}(p_{r_1} \cdots p_{r_s}).
\]
We let $n = \sum_i r_i$ so that there are $n+1$ markings. 
There are three terms in the right hand side of the HAE which we denote by
\[ \frac{d}{dG_2} \CC_{g}( \tau_1(\sigma) \tau_{0}(\Gamma_{r_1}) \cdots \tau_{0}(\Gamma_{r_s}))  = I_1 + I_2 + I_3 \]
and consider one-by-one below.

\subsubsection{Term 1}
The first term is
\[ I_1 = \CC_{g-1}( \tau_1(\sigma) \tau_{0}(\Gamma_{r_1}) \cdots \tau_{0}(\Gamma_{r_s}) \tau_0( f_{\theta})) , \quad f_{\theta} = \frac{(\theta_1 + \theta_2 - 2 \eta_{12})^{h-1}}{(h-1)!} \] 
If $g \geq 3$, by Corollary \ref{cor:dim constraint} for this to contribute we must have
\[ h+n+3 = 1 + h + n + \deg(f_{\theta}) \Rightarrow h-1 = 2, \]
so $h=3$. If $g= 2$, by \eqref{cor:dim constraint} again we must have $h=2$ which we excluded from consideration.
Hence we can assume $h=3$ in which case after symmetrizing we get the following $4$ terms:
\begin{align*}
	I_1 & = \CC_{g-1}( \tau_1(\sigma) \tau_{0}(\Gamma_{r_1}) \cdots \tau_{0}(\Gamma_{r_s}) \tau_{00}(\theta_1^2)) \\
	& + 2 \CC_{g-1}( \tau_1(\sigma) \tau_{0}(\Gamma_{r_1}) \cdots \tau_{0}(\Gamma_{r_s}) \tau_{00}(\eta_{12}^2)) \\
	& + \CC_{g-1}( \tau_1(\sigma) \tau_{0}(\Gamma_{r_1}) \cdots \tau_{0}(\Gamma_{r_s}) \tau_0(\theta) \tau_0(\theta)) \\
	& - 4 \CC_{g-1}( \tau_1(\sigma) \tau_{0}(\Gamma_{r_1}) \cdots \tau_{0}(\Gamma_{r_s}) \tau_{00}(\theta_1 \eta_{12})).
\end{align*}

We consider these 4 terms one by one:

\noindent \textbf{Term 1a:} By the string equation (Lemma~\ref{lemma:string dilaton equation}) and the primary reconstruction (Theorem~\ref{thm:reconstruction primary}) we have
\begin{align*}
	\CC_{g-1,\chi}\left( \tau_1(\sigma) \tau_{0}(\Gamma_{r_1}) \cdots \tau_{0}(\Gamma_{r_s}) \tau_{00}(\theta_1^2) \right)
	& = \CC_{g-1,\chi}\left( \tau_0(\sigma) \tau_{0}(\Gamma_{r_1}) \cdots \tau_{0}(\Gamma_{r_s}) \tau_{0}(\theta_1^2) \right) \\
	& = \int_{X \times \beta^{\times (n+2)}} \varphi^{\ast}( \sigma \boxtimes \Gamma_{r_1} \boxtimes \cdots \boxtimes \Gamma_{r_s} \boxtimes \theta^2) N_{g-1,\chi}^{\CA_h} \\
	& \overset{(1)}{=} \int_{\beta^{\times n}} \Gamma_{r_1} \boxtimes \cdots \boxtimes \Gamma_{r_s} \int_{X \times \beta \times \beta} \varphi^{\ast}( \sigma_1 \theta_2^2 ) N_{g-1,\chi}^{\CA_h} \\
	& \overset{(2)}{=} 2 (p_1^2 - p_2) 
	\left( - \frac{1}{2} \right)^{\sum_i (r_i-1)} p_{r_1} \cdots p_{r_s} N_{g-1,\chi}^{\CA_h}.
\end{align*}
where in (1) we used that $\Gamma_{r}$ has degree $2$ in each factor and we integrate over $\beta$ in every factor, so the $\varphi$-pullback has to leave the classes invariant (see Lemma~\ref{lemma:phi pullback eval} for how the $\phi$-pullback acts). 
To see step (2) observe that $\varphi^{\ast}(\sigma_1) = \alpha_{\ast}(1)$ where $\alpha : X^2 \to X^3$ is given by $\alpha(x_0,x_2) = (x_0,-x_0,x_2)$.
Moreover the involution $\iota : X \to X, \iota(x)=-x$ acts by $+1$ on $H^{2h-2}(X)$. Thus
\begin{align*}
	\int_{X \times \beta \times \beta} \varphi^{\ast}( \sigma_1 \theta_2^2 )
	& = \int_{\beta \times \beta} \alpha^{\ast}(\varphi^{\ast}(\theta_2^2)) = \int_{\beta \times \beta} (\theta_1 + \theta_2 + 2 \eta_{12})^2
	= 2 (p_1^2 - p_2)
\end{align*}
where we used Lemma~\ref{lemma:intersection computation} in the last step.

Thus we find
\[
\CC_{g-1}( \tau_1(\sigma) \tau_{0}(\Gamma_{r_1}) \cdots \tau_{0}(\Gamma_{r_s}) \tau_{00}(\theta_1^2))
=
2 \left( - \frac{1}{2} \right)^{\sum_i (r_i-1)} F_{g-1}\left( (p_1^2 - p_2) p_{r_1} \cdots p_{r_s} \right)
\]

\noindent \textbf{Term 1b:}
The insertion is just another cycle of length $2$. So we find
\[
2 \CC_{g-1}( \tau_1(\sigma) \tau_{0}(\Gamma_{r_1}) \cdots \tau_{0}(\Gamma_{r_s}) \tau_{00}(\eta_{12}^2))
=
- (2g-4) \left( - \frac{1}{2} \right)^{\sum_i (r_i-1)} F_{g-1}\left( p_2 p_{r_1} \cdots p_{r_s} \right)
\]
\noindent \textbf{Term 1c:} By the divisor equation (or viewing the last two insertions as cycles of length $1$:
\[ \CC_{g-1}( \tau_1(\sigma) \tau_{0}(\Gamma_{r_1}) \cdots \tau_{0}(\Gamma_{r_s}) \tau_0(\theta) \tau_0(\theta))
=
(2g-4) \left( - \frac{1}{2} \right)^{\sum_i (r_i-1)} F_{g-1}\left( p_1^2 p_{r_1} \cdots p_{r_s} \right) \]
\noindent \textbf{Term 1d:}
\[ \CC_{g-1}( \tau_1(\sigma) \tau_{0}(\Gamma_{r_1}) \cdots \tau_{0}(\Gamma_{r_s}) \tau_{00}(\theta_1 \eta_{12})) = 0 \]
by the string equation.

\noindent \textbf{Sum:} As a sum of the above 4 terms we obtain:
\[
I_1 = (2g-2) \left( - \frac{1}{2} \right)^{\sum_i (r_i-1)} F_{g-1}\left( (p_1^2 - p_2) p_{r_1} \cdots p_{r_s} \right).
\]

\subsubsection{Term 2}
Consider the second term in the HAE which reads
\[
I_2 = 2 \sum_{\substack{g=g_1 + g_2 \\ \{ 1, \ldots, n \} = A \sqcup B }}
\CC_{g_1}(\tau_1(\sigma) \tau_0(\gamma_A) \tau_0((f_{\theta})_0) ) \CC_{g_2}(\tau_0((f_{\theta})_1) \tau_0(\gamma_B))
\]
where we used $\gamma = \Gamma_{r_1} \boxtimes \cdots \boxtimes \Gamma_{r_s}$
and we write $(f_{\theta})_{0}, (f_{\theta})_1$ informally as placeholder for the corresponding $f_{\theta}$ insertion.
Since $\deg \tau_0((f_{\theta})_1) \tau_0(\gamma_B) \leq h-1 + |B|$ by Lemma~\ref{lemma:vanishing1b} only the case $g_2=0$ can contribute.
Since there is no $\psi$-class on the component we must have $|B|=2$.
In this case, since all the insertions are tautological, this is evaluated by integration. So we find
\[
I_2 =2  \sum_{1 \leq i < j \leq n} \CC_g( \tau_1(\sigma) \tau_0\left( f_{\theta}^{(i)}( \Delta_{ij}^{\ast}( \Gamma_{r_1} \boxtimes \cdots \boxtimes \Gamma_{r_s})) \right)
\]
where $\Delta_{ij}: \CX^{n-1} \to \CX^n$ is the inclusion of the big $(i,j)$-diagonal and $f_{\theta}^{(i)}$ is the action of the Lefschetz operator on the $i$-th factor.

The evaluation of the action of $f_{\theta}$ is provided by the following two propositions. Define formally
\[ \Gamma_0 := -2h . \]
We also set $p_0 := h$ so that the the statement of Lemma~\ref{lemma:intersection computation} continues to hold for $r=0$.
%by setting
%\[ p_0 = h. \]
(This is the natural definition here since we are working with power sums in $h$ variables.)
\begin{prop}
	For $1 \leq i < j \leq r$ we have (up to applying suitable permutations)
	\[ f_{\theta}( \Delta_{ij}^{\ast}(\Gamma_r)) = -2 \Gamma_{r-1} - \frac{1}{2} \Gamma_{j-i} \Gamma_{r-(j-i)-1} - \frac{1}{2} \Gamma_{j-i-1} \Gamma_{r-(j-i)}. \]
\end{prop}
\begin{proof}
	Without loss of generality we can take $i=1$.
	We will use Lemma~\ref{lemma:F theta computation} below.
	Consider the case $j=2$. Then for $r=2$ this is
	\[ f_{\theta}(\Delta_{12}^{\ast}(\Gamma_2)) = f_{\theta}( \theta^2 ) = (2h-2) \theta = (2h-2) \Gamma_1 \]
	So the claim holds. If $r = 3$ we get the claim by
	\[
	f_{\theta}(\Delta_{12}^{\ast}(\Gamma_3)) = f_{\theta}( \eta_{31} \theta_1 \eta_{31} ) = (h-2) \eta_{13}^2 - \frac{1}{2} \theta_1 \theta_3
	= (h-2) \Gamma_2 - \frac{1}{2} \Gamma_1^2
	\]
	If $r \geq 4$ we get the claim by
	\[ f_{\theta}(\Delta_{12}^{\ast}(\Gamma_r))
	= f_{\theta}^{(1)}( \eta_{r1} \theta_1 \eta_{13} ) \eta_{34} \cdots \eta_{r-1,r}
	= ((h-2) \eta_{13} \eta_{1r} - \frac{1}{2} \theta_1 \eta_{r3}) \eta_{34} \cdots \eta_{r-1, r}
	= (h-2) \Gamma_{r-1} - \frac{1}{2} \Gamma_1 \Gamma_{r-2} \]
	For $j \geq 3$ and $r \geq j+1$ we get
	\[
	f_{\theta}(\Delta_{1j}^{\ast}(\Gamma_r))
	= f_{\theta}^{(1)}( \eta_{r1} \eta_{12} \eta_{1,j-1} \eta_{1,j+1} ) \eta_{2 3 } \cdots \eta_{j-2,j-1} \cdot \eta_{j+1, j+2} \cdots \eta_{r-1,r}
	\]
	The formula in Lemma~\ref{lemma:F theta computation} now says that the two chains $\eta_{2 3 } \cdots \eta_{j-2,j-1}$ and $\eta_{j+1, j+2} \cdots \eta_{r-1,r}$ are closed up, where in one of the two resulting loops the index $1$ is added, so cycles $\{ (23 \cdots j-1), (1 (j+1) \cdots r) \}$ and $\{ (123 \ldots j-1), (j+1, \ldots, r) \}$, giving $-\frac{1}{2}( \Gamma_{j-2} \Gamma_{r-j + 1} + \Gamma_{j-1} \Gamma_{r-j})$.
	Or they are joined to a one connected cycle of length $r-1$, which happens $4$ times, so gives $-2 \Gamma_{r-1}$.
	This shows the claim.
\end{proof}

\begin{lemma}
	For any $1 \leq i \leq r_1$ and $1 \leq j \leq r_2$ we get
	\[
	f_{\theta}^{(a)}( \Delta_{ab}^{\ast}( \Gamma_{r_1} \Gamma_{r_2}))
	=
	-2 \Gamma_{r_1 + r_2 - 1} - \frac{1}{2} \Gamma_{r_1} \Gamma_{r_2 - 1} - \frac{1}{2} \Gamma_{r_1 - 1} \Gamma_{r_2}.
	\]
\end{lemma}
\begin{proof}
	This follows likewise from Lemma~\ref{lemma:F theta computation}.
\end{proof}

We hence find the following second term in the holomorphic anomaly equation:
\begin{align*}
	I_2 
	& = 2 \sum_{a=1}^{s} \sum_{1 \leq i < j \leq r_a} \CC_g( \tau_{1}(\sigma) \tau_0( -2 \Gamma_{r_a-1} - \frac{1}{2} \Gamma_{j-i} \Gamma_{r_a - (j-i)-1} - \frac{1}{2} \Gamma_{j-i-1} \Gamma_{r_a-(j-i)} ) \prod_{\ell \neq a} \tau_{0}(\Gamma_{r_{\ell}})) \\
	& + 
	2 \sum_{1 \leq a < b \leq s} r_a r_b 
	\CC_{g}(\tau_1(\sigma) \tau_0\left[ -2 \Gamma_{r_a + r_b-1} - \frac{1}{2} \Gamma_{r_a} \Gamma_{r_b - 1} - \frac{1}{2} \Gamma_{r_a - 1} \Gamma_{r_b} \right] 
	\prod_{\ell \neq a,b} \tau_0(\Gamma_{\ell}) ) \\
	& = \sum_{a=1}^{s} \CC_{g}( \tau_1(\sigma) \tau_0\left[ -4 \binom{r_a}{2} \Gamma_{r_a-1} - r_a \sum_{k=1}^{r_a-1} \Gamma_k \Gamma_{r_a-k-1} \right] 
	\prod_{\ell \neq a} \tau_{0}(\Gamma_{r_{\ell}})) \\
	& + \sum_{1 \leq a < b \leq s} r_a r_b 
	\CC_{g}(\tau_1(\sigma) \tau_0\left[ -4 \Gamma_{r_a + r_b-1} - \Gamma_{r_a} \Gamma_{r_b - 1} - \Gamma_{r_a - 1} \Gamma_{r_b} \right] 
	\prod_{\ell \neq a,b} \tau_0(\Gamma_{\ell}) ) \\
	& = 
	(2g-2) \left( -\frac{1}{2} \right)^{\sum_a (r_a - 1)}
	4
	\sum_{a=1}^{s}
	F_g\left( \left[ r_a (r_a-1) p_{r_a - 1} - r_a \sum_{k=1}^{r_a-1} p_k p_{r_a-k-1} \right] \prod_{\ell \neq a} p_{\ell} \right)  \\
	& + 
	(2g-2) \left( -\frac{1}{2} \right)^{\sum_a (r_a - 1)}
	2 \sum_{1 \leq a < b \leq s} r_a r_b F_g\left( \left[ -2 p_{r_a + r_b - 1} + p_{r_a} p_{r_b - 1} + p_{r_a - 1} p_{r_b} \right] \prod_{\ell \neq a,b} p_{\ell} \right),
\end{align*}
where in the first step we used:
\[
\sum_{1 \leq i < j \leq r_a} (
\Gamma_{j-i} \Gamma_{r_a - (j-i)-1} + \Gamma_{j-i-1} \Gamma_{r_a-(j-i)-1} )
=
\sum_{k=1}^{r-1} (r-k) (\Gamma_{k} \Gamma_{r-k-1} + \Gamma_{k-1} \Gamma_{r-k})
= r \sum_{k=1}^{r-1} \Gamma_k \Gamma_{r-k-1}
\]

\subsubsection{Term 3}
Since $f_{\theta}(\sigma) = \theta^{h-1}/(h-1)!$ the third term in the holomorphic anomaly equation is:
\begin{align*}
	I_3 & = 
	-2 \CC_{g}( \tau_2\left( \frac{\theta^{h-1}}{(h-1)!} \right) \tau_0( \Gamma_{r_1}) \cdots \tau_0(\Gamma_{r_s})) \\
	& - 2 \sum_{a=1}^{s} \sum_{i=1}^{r_a} \CC_{g}\left(\tau_1(\sigma) \tau_{(0^{i-1} 1 0^{r_a-i})}(f_{\theta}^{(i)}\Gamma_{r_a}) \prod_{\ell \neq a} \tau_0(\Gamma_{r_{\ell}} ) \right)
\end{align*}

\noindent \textbf{Term 3a:}\
Using Lemma~\ref{lemma:tau2 reduce to quotient} below we have
\begin{align*}
	& - 2 \CC_{g}\left( \tau_2\left( \frac{\theta^{h-1}}{(h-1)!} \right) \tau_0( \Gamma_{r_1}) \cdots \tau_0(\Gamma_{r_s}) \right) \\
	& =
	\frac{-2}{(h-1)!} \sum_{a=1}^{s} 
	\sum_{i=1}^{r_a} 
	\CC_{g}\left( \tau_{0^{i-1} 1 0^{r_a-i}}( \theta_i^{h-1} \Gamma_{r_a} ) \prod_{\ell \neq a} \tau_0(\Gamma_{r_{\ell}}) \right) \\
	& = \sum_{a=1}^{s} r_a 
	\CC_{g}\left( \tau_1(\sigma) \tau_0(\Gamma_{r_a-1}) \prod_{\ell \neq a} \tau_{0}(\Gamma_{r_{\ell}}) \right) \\
	& = -2 (2g-2)  \left( -\frac{1}{2} \right)^{\sum_{a} (r_a -1 )} \sum_{a=1}^{s} r_a  F_{g}\left( p_{r_a-1} \prod_{\ell \neq a} p_{r_\ell} \right)
\end{align*}
where we used the relation
\[ \frac{\theta_1^{h-1}}{(h-1)!} \Gamma_r = -\frac{1}{2} \sigma_1 \cdot (\Gamma_{r-1})_{2 \cdots r}. \]
which follows from Example~\ref{example:some eta relation} (and is valid also if $r=1$ since $\Gamma_0=-2h$).

\noindent \textbf{Term 3b:}
The action of the Lefschetz operator on a cycle is given by
\[ f_{\theta}^{(1)}( \Gamma_r) = -\frac{1}{2} (\Gamma_{r-1})_{2 \cdots r}. \]
Inserting and using the dilaton equation we obtain the term
\begin{align*}
	& - 2 \sum_{a=1}^{r_s} \sum_{i=1}^{r_a} \CC_{g}\left(\tau_1(\sigma) \tau_{(0^{i-1} 1 0^{r_a-i})}(f_{\theta}^{(i)}\Gamma_{r_a}) \prod_{\ell \neq a} \tau_0(\Gamma_{r_{\ell}} ) \right) \\
	& = (2g-2+n) \sum_{a=1}^{r_s} r_a \CC_{g}\left(\tau_1(\sigma) \tau_{0}(\Gamma_{r_a-1}) \prod_{\ell \neq a} \tau_0(\Gamma_{r_{\ell}} ) \right) \\
	& =
	-2 (2g-2 + n) (2g-2) \left( -\frac{1}{2} \right)^{\sum_a (r_a - 1)} \sum_{a=1}^{s} r_a F_g( p_{r_a-1} \prod_{\ell \neq a} p_{r_{\ell}} )
\end{align*}

We hence find the following third term in the holomorphic anomaly equation:
\begin{align*}
	I_3 =
	%& = (-2) (2g-2)  \left( -\frac{1}{2} \right)^{\sum_{a} (r_a -1 )} \sum_{a=1}^{s} r_a  F_{g}\left( p_{r_a-1} \prod_{\ell \neq a} p_{r_\ell} \right) \\
	& -2 (2g-2 + n + 1) (2g-2) \left( -\frac{1}{2} \right)^{\sum_a (r_a - 1)} \sum_{a=1}^{s} r_a F_g( p_{r_a-1} \prod_{\ell \neq a} p_{r_{\ell}} )
\end{align*}

We have used above the following result:
\begin{lemma} \label{lemma:tau2 reduce to quotient}
	With the same condition as in Lemma~\ref{lemma:tau1 sigma reduce to quotient} on $\Gamma$, for all $g \geq 2$ we have
	\[ \CC_{g,\chi}\left( \tau_2\left(  \theta^{h-1} \right) \tau_{0^n}(\Gamma) \right)
	=
	\sum_{i=1}^{n} \CC_{g,\chi}\left( \tau_{0^{i-1} 1 0^{n-i}}( \theta_i^{h-1} \Gamma) \right). \]
\end{lemma}
\begin{proof}
As in the proof of Lemma~\ref{lemma:tau1 sigma reduce to quotient} 
we consider the morphism $P : \Mbar_{g,n+1}(\CX,\chi) \to \Mbar_{g,1}(\CX,\chi)$ forgetting all except the marking labeled $0$.
	As before we have
	\[ \psi_0 = P^{\ast}(\psi_0) + \sum_{\substack{S \subset \{ 0, \ldots, n \} \\ 0 \in S, |S| \geq 2}} D_S. \]
	By the degree condition on $\Gamma$ (and since all classes are tautological), if $|S| \geq 3$, then
	\[ D_S \cdot \ev_0^{\ast}(\theta^{h-1}) \ev_{1 \cdots n}^{\ast}(\Gamma) \cap [\Mbar_{g,n+1}(\CX,\chi)]^{\vir} = 0. \]
	Moreover, $\psi_0 \cdot D_S = 0$ if $|S| = 2$. Thus
	\begin{align*}
		& \psi_0^2 \cdot \ev_0^{\ast}(\theta^{h-1}) \ev_{1 \cdots n}^{\ast}(\Gamma) \cap [\Mbar_{g,n+1}(\CX,\chi)]^{\vir} \\
		& = \psi_0 \cdot P^{\ast}(\psi_0) \cdot \ev_0^{\ast}(\theta^{h-1}) \ev_{1 \cdots n}^{\ast}(\Gamma) \cap [\Mbar_{g,n+1}(\CX,\chi)]^{\vir} \\
		& = \left( P^{\ast}(\psi_0) + \sum_{i=1}^{n} D_{0i} \right) \cdot P^{\ast}(\psi_0) \cdot \ev_0^{\ast}(\theta^{h-1}) \ev_{1 \cdots n}^{\ast}(\Gamma) \cap [\Mbar_{g,n+1}(\CX,\chi)]^{\vir}
	\end{align*}
	
	We have
	\[
	\int_{[\Mbar_{g,n+1}(\CX,\chi)]^{\vir}} P^{\ast}(\psi_0)^2 \ev_0^{\ast}(\theta^{h-1}) \ev_{1 \cdots n}^{\ast}(\Gamma) 
	=
	\int_{[\Mbar_{g,1}(\CX,\chi)]^{\vir}} \ev_0^{\ast}(\theta^{h-1}) \psi_0^{2} P_{\ast}(  \ev_{1 \cdots n}^{\ast}(\Gamma)  ) = 0
	\]
	by Proposition~\ref{prop:vanishing in low multiplicity}, since $P_{\ast}(  \ev_{1 \cdots n}^{\ast}(\Gamma)  )$ is a class of degree $0$.
	
	Thus we find
	\begin{align*}
		\CC_{g,\chi}\left( \tau_2\left(  \theta^{h-1} \right) \tau_{0^n}(\Gamma) \right)
		& =
		\sum_{i=1}^{n} 
		\int_{[\Mbar_{g,n+1}(\CX,\chi)]^{\vir}} D_{0i}  P^{\ast}(\psi_0) \ev_0^{\ast}(\theta^{h-1}) \ev_{1 \cdots n}^{\ast}(\Gamma) \\
		& = 
		\sum_{i=1}^{n} \int_{[\Mbar_{g,n}(\CX,\chi)]^{\vir}} \ev^{\ast}( \theta_i^{h-1} \Gamma) P_i^{\ast}(\psi_i) \\
		& = 
		\sum_{i=1}^{n} \int_{[\Mbar_{g,n}(\CX,\chi)]^{\vir}} \ev^{\ast}( \theta_i^{h-1} \Gamma) \left( \psi_i - \sum_{i \in S} D_S \right) \\
		& = 
		\sum_{i=1}^{n} \int_{[\Mbar_{g,n}(\CX,\chi)]^{\vir}} \ev^{\ast}( \theta_i^{h-1} \Gamma) \psi_i,
	\end{align*}
	where $P_i : \Mbar_{g,n}(\CX,\chi) \to \Mbar_{g,1}(\CX,\chi)$ forgets all but the $i$-th marking
	and we used again the degree condition to conclude $\Delta_{S}^{\ast}(\theta_i^{h-1} \Gamma) = 0$.
\end{proof}

\subsubsection{Conclusion}
Summing up Terms 1-3 we find the following explicit expression:
\begin{align*}
	\frac{d}{dG_2} F_g(p_{r_1} \cdots p_{r_s})
	& =  \delta_{h=3} F_{g-1}\left( (p_1^2 - p_2) p_{r_1} \cdots p_{r_s} \right) \\
	& + 
	4
	\sum_{a=1}^{s}
	F_g\left( \left[ r_a (r_a-1) p_{r_a - 1} - r_a \sum_{k=1}^{r_a-1} p_k p_{r_a-k-1} \right] \prod_{\ell \neq a} p_{\ell} \right)  \\
	& + 
	2 \sum_{1 \leq a < b \leq s} r_a r_b F_g\left( \left[ -2 p_{r_a + r_b - 1} + p_{r_a} p_{r_b - 1} + p_{r_a - 1} p_{r_b} \right] \prod_{\ell \neq a,b} p_{\ell} \right), \\
	& -2 (2g-2 + n + 1) \sum_{a=1}^{s} r_a F_g( p_{r_a-1} \prod_{\ell \neq a} p_{r_{\ell}} )
\end{align*}

Let $\partial_r = \frac{d}{d p_r}$ and $p_0=h$ as before, and define the operators
\[
\mathsf{WT} = \sum_{r \geq 1} r p_r \partial_r, \quad 
\Psi = \sum_{r \geq 1} r p_{r-1} \partial_r.
\]
Then the above formula can be rewritten as:
\begin{align*}
	\frac{d}{dG_2} F_g(p_{r_1} \cdots p_{r_s})
	= & \delta_{h=3} F_{g-1}\left( (p_1^2 - p_2) p_{r_1} \cdots p_{r_s} \right) \\
	& + 
	F_g\left( \left[ 4 \sum_{r \geq 2} r (r-1) p_{r-1} \partial_r - 4 \sum_{r \geq 2} \sum_{k=1}^{r-1} r p_k p_{r-1-k} \partial_r \right] p_{r_1} \cdots p_{r_s} \right) \\
	& + F_g\left( \left[ 2 \Psi \Wt - 2 \sum_{r \geq 1} r^2 p_{r-1} \partial_r - 2 \sum_{k, \ell \geq 1} k \ell p_{k + \ell - 1} \partial_k \partial_{\ell} \right] p_{r_1} \cdots p_{r_s} \right) \\
	& - 2 F_g\left( \Psi (2g-1 +\Wt) p_{r_1} \cdots p_{r_s} \right) \\
	=& \delta_{h=3}  F_{g-1}\left( (p_1^2 - p_2) p_{r_1} \cdots p_{r_s} \right) \\
	& + 
	F_g\left( \left[ 4 \sum_{r \geq 2} r (r-1) p_{r-1} \partial_r - 4 \sum_{r \geq 2} \sum_{k=1}^{r-1} r p_k p_{r-1-k} \partial_r \right] p_{r_1} \cdots p_{r_s} \right) \\
	& + F_g\left( \left[ - 2 \sum_{r \geq 1} r^2 p_{r-1} \partial_r - 2 \sum_{k, \ell \geq 1} k \ell p_{k + \ell - 1} \partial_k \partial_{\ell} \right] p_{r_1} \cdots p_{r_s} \right) \\
	& - 2 F_g\left( \Psi (2g-1) p_{r_1} \cdots p_{r_s} \right) .
\end{align*}

In other words, if we define the Holomorphic anomaly operator
\begin{align*}
	\CL_{g,h}
	& :=
	4 \sum_{r \geq 2} r (r-1) p_{r-1} \partial_r - 4 \sum_{r \geq 2} \sum_{k=1}^{r-1} r p_k p_{r-1-k} \partial_r \\
	& - 2 \sum_{r \geq 1} r^2 p_{r-1} \partial_r - 2 \sum_{k, \ell \geq 1} k \ell p_{k + \ell - 1} \partial_k \partial_{\ell} - 2 (2g-1) \Psi \\
	& = 
	2 \sum_{r \geq 1} r (r-2g-1) p_{r-1} \partial_r - 4 \sum_{r \geq 2} \sum_{k=1}^{r-1} r p_k p_{r-1-k} \partial_r
	- 2 \sum_{k, \ell \geq 1} k \ell p_{k + \ell - 1} \partial_k \partial_{\ell}.
\end{align*}
Then we have (for $g \geq 2$ and $h \geq 3$)
	\[ \frac{d}{dG_2} F_g(p_{r_1} \cdots p_{r_s})
	=  \delta_{h=3} F_{g-1}\left( (p_1^2 - p_2) p_{r_1} \cdots p_{r_s} \right) 
	+ F_g( \CL_{g,h}  (p_{r_1} \cdots p_{r_s})).
	\]
The claim of Theorem~\ref{thm:hae for quotient text} 
then follows from rewriting the operator $\CL_{g,h}$ in terms of the $e_i$.
This is done in Appendix~\ref{appendix:an operator identity}. \qed

\section{Noether-Lefschetz cycles and Gromov-Witten theory}
\label{sec:NL and GW theory}
\subsection{Overview}
In this section we express the Gromov-Witten classes of $\CA_h$ as pushforwards of explicit cycles from the Noether-Lefschetz loci.
We start with some basics on Noether-Lefschetz loci and the behaviour of characteristic polynomials of curve classes under pushfoward. In Section~\ref{subsec:NLdecomp} we then decompose the moduli space of stable maps according to the dimension and polarization type of the smallest abelian subvariety such that a translate contains the image of the curve and derive the precise obstruction classes.
The key formula we obtain is \eqref{full rank decomposition}.
This is used in Section~\ref{subsec:vanishings II} to derive vanishing statements for Gromov-Witten classes.

\subsection{Noether-Lefschetz loci} \label{subsec:NL loci}
We give some standard definitions on Noether-Lefschetz loci following \cite{AitorNL}.
A polarization $\theta$ on an $u$-dimensional abelian variety $B$ has {\em type} $\delta = (\delta_1,\ldots,\delta_u)$ with $\delta_1 | \delta_2 | \cdots | \delta_u$ if the kernel $\ker(\varphi_{\theta})$ of the associated homomorphism
\[ \varphi_{\theta} : B \to \hat{B},\quad  b \mapsto t_{b}^{\ast} \CO(\theta) \otimes \CO(\theta)^{-1} \]
is isomorphic to the finite abelian group
\[ K(\delta) := (\BZ / \delta_1)^2 \times \ldots \times (\BZ/\delta_u)^2. \]
One has
\[ \int_{B} \frac{\theta^{u}}{u!} = \delta_1 \cdots \delta_u. \]

There is a natural symplectic pairing on $\ker(\varphi_{\theta})$ which under a change of basis corresponds to the standard symplectic pairing on $K(\delta)$.
A level structure for $\theta$ is a symplectic isomorphism of groups $f : K(\delta) \xrightarrow{\cong} \ker(\varphi_{\theta})$.
We let $\CA_{u,\delta}$ be the moduli stack of abelian varieties with a polarization of type $\delta$,
and let $\CA_{u,\delta}^{\mathrm{lev}}$ be the moduli stack of triples $(B,\theta,f)$ where $(B,\theta) \in \CA_{u,\delta}$ and $f$ is a level structure.

Given an integer $u \leq h/2$ and a polarization type $\delta = (d_1,\ldots,d_u)$ we define\footnote{This defines the Noether-Lefschetz loci as a set, see below for the definition as a cycle.} the Noether-Lefschetz loci to be
\[ \NL_{h,\delta} = \left\{ (X,\theta) \in \CA_h \middle|  \begin{array}{c} X \text{ has an abelian subvariety } B \text{ of dimension } u \\ \text{ and } \theta|_{B} \text{ is of type } \delta \end{array} \right\}. \]
These loci are irreducible, closed and of codimension $u \cdot (h-u)$, see \cite{DL}.\footnote{The second type of Noether-Lefschetz loci in $\CA_h$ are the abelian varieties with real multiplication,
	which exists for every divisor $k | h$ with $k \neq h$, and are of codimension $h(h-k)/2$.
	We will not require them in this paper.}

For any abelian subvariety $B \subset X$ there exists a complementary abelian subvariety $C \subset X$ of dimension $h-u$. If $\theta|_{B}$ of type $\delta$, then $\theta|_{C}$ has type $\tilde{\delta} = (1^{h-2u}, \delta)$. Moreover, there exists a canonical anti-symplectic isomorphism $\xi : \ker(\theta|_{B}) \cong \ker(\theta|_{C})$ and $X$ is obtained as the quotient $(B \times C)/\Lambda$, where $\Lambda$ is the graph of $\xi$. This has a converse description.
Fix any anti-symplectic isomorphism $\xi_0 : K(\delta) \to K(\tilde{\delta})$. Then there is a morphism
\[
\varphi_{h,\delta} : \CAlev_{u,\delta} \times \CAlev_{h-u,\tilde{\delta}} \to \CA_{h}
\]
given by sending $(B,\theta_B,f), (C,\theta_C,f')$ to the pair $(X,\theta)$, where $X=(B \times C)/\Lambda$, the lattice $\Lambda$ is the graph of the anti-symplectic isomorphism $f' \circ \xi_0 \circ f^{-1}$, and $\theta$ is the descent of $(\theta_B)_1 + (\theta_C)_2$ along the quotient.

We also let $\CA_{u,h-u,\delta}$ be the moduli stack of triples $(X,\theta,B)$ where $(X,\theta)$ is a principal polarized abelian variety and $B \subset X$ is a subvariety of dimension $u$ with $\theta|_{B}$ of type $\delta$. Then
\[ \CA_{u,h-u,\delta} \cong [ \CAlev_{u,\delta} \times \CAlev_{h-u,\tilde{\delta}} / \Sp(K(\delta)) ] \]
and the morphism $\varphi_{h,\delta}$ factors as the quotient map followed by the forgetful map
\[ \varphi'_{h,\delta} : \CA_{u,h-u,\delta} \to \CA_{h}, (X,\theta,B) \to (X,\theta). \]
The map $\varphi'_{h,\delta}$ is generically of degree $1$, except if $h=2u$ where it is degree $2$ because we can interchange $B$ and its complement $C$.
We define $\NL_{h,\delta}$ as a cycle as the pushforward of the fundamental class of $\CA_{u,h-u,\delta}$ by $\varphi_{h,\delta}'$, see \cite{AitorNL}.

We write $\BE_{u}$ for the Hodge bundle on $\CA_{u,\delta}$.
The normal bundle of the maps $\varphi_{h,\delta}$ is
\[ \BE_{h-u}^{\vee} \otimes \BE_{u}^{\vee}, \]
where suppress the pullback under the projections to $\CA_{u,\delta}$ and $\CA_{h-u,\tilde{\delta}}$.
The same holds then also for $\varphi'_{h,\delta}$ by descent.

\subsection{Characteristic polynomial} \label{subsec:char polynomial in the non-principally polarized case}
Let $(B,\theta_B)$ be an abelian variety of dimension $u$ with polarization type $\delta = (\delta_1,\ldots,\delta_u)$. The characteristic polynomial of a class $v \in H^2(B,\BZ)$ is
\[
\chi_{v}(t) = \frac{1}{\delta_1 \cdots \delta_u} \int_{B} \frac{(t \theta_B - v)^u}{u!}.
\]
If $\gamma \in H_2(B,\BZ)$ then its characteristic polynomial $\chi_{\beta}$ is defined to be the characteristic polynomial of $v_{\gamma}$, where $v_{\gamma}$ is the image of $\gamma$ under the map
\[ H_2(B,\BZ) \equiv H^2(\widehat{B},\BZ) \xrightarrow{\varphi_{\theta}^{\ast}} H^2(B,\BZ). \]

Assume now that $B$ is embedded as a subvariety of a principally polarized abelian variety $(X,\theta)$ of dimension $h$ such that $\theta_B= \theta|_{B}$. Let $j : B \to X$ be the inclusion.
Let $\beta = j_{\ast}\gamma$ for a class $\gamma \in H_2(B,\BZ)$.
We have the following relationship between the characteristic polynomials of $\beta$ and $\gamma$:
\begin{lemma} \label{lemma:cbar comparison}
$\chi_{\beta} = \chi_{\gamma}(t) t^{h-u}$.
\end{lemma}
\begin{proof}
Let $\widehat{j}: \widehat{X} \to \widehat{B}$ be the dual map.
We have the commutative diagram
\[
\begin{tikzcd}
		H_2(X,\BZ) \ar{r}{\equiv} & H^2(\widehat{X},\BZ) \ar{r}{\varphi_{\theta}^{\ast}} & H^2(X,\BZ) \ar{d}{j^{\ast}} \\
		H_2(B,\BZ) \ar{u}{j_{\ast}} \ar{r}{\equiv} & H^2(\widehat{B},\BZ) \ar{u}{\widehat{j}^{\ast}} \ar{r}{\varphi_{\theta_{B}}^{\ast}} & H^2(B,\BZ).
	\end{tikzcd}
	\]
	The commutativity of the second square follows here by \cite[Corollary 2.4.6 (iii)]{BL}.
	In particular, 
	\[ j^{\ast} v_{\beta} = v_{\gamma} \]
	where $v_{\gamma}$ is the image of $\gamma$ under the bottom row map.
	
Choose a complementary abelian subvariety $C$ to $B$ and let $\theta_C := \theta|_{C}$. Then $(C,\theta_C)$ has polarization type $\tilde{\delta}$ and the quotient map $B \times C \to X$ has degree $(\delta_1 \ldots \delta_u)^2$. We find that
	\begin{align*}
\chi_{\beta}(t) 
		& = \int_{X} \frac{(t \theta - v_{\beta})^h}{h!} \\
		& = \frac{1}{(\delta_1 \cdots \delta_u)^2} \int_{B \times C} \frac{ (t \theta_B + t \theta_C - \pr_1^{\ast}(v_{\gamma}))^{h}}{h!} \\
		& = \frac{t^{h-u}}{\delta_1 \cdots \delta_u} \int_{B} \frac{(t \theta_B - v_{\gamma})^u}{u!} \\
		& = t^{h-u} \chi_{\gamma}.
	\end{align*}
\end{proof}

\begin{defn}
	A class $\beta \in N_1(X)$ is of {\em rank $u$} if its characteristic polynomial  has rank $u$, i.e. it is of the form
	$\chi_{\beta} = t^h - a_1 t^{h-1} + \ldots + a_{h-u} t^{h-u}$ with $a_{h-u} \neq 0$.
\end{defn}

By Lemma~\ref{lemma:cbar comparison} the pushforward $\beta$ from a class on an abelian subvariety of dimension $u$ is of rank $\leq u$. For algebraic classes the converse is also true.

\begin{lemma}
	If $\beta \in N_1(X)$ is of rank $u$, then it is the pushforward of a class
	from an abelian subvariety $B$ of dimension $u$.
	Hence the rank of $\beta$ is the smallest dimension of an abelian subvariety $j : B \hookrightarrow X$ such that $\beta = j_{\ast} \gamma$ for some $\gamma \in N_1(B)$.
\end{lemma}
\begin{proof}
	Let $\alpha \in \NS(\widehat{X})$ be the class under the natural isomorphism $N_1(X) \cong \NS(\widehat{X})$ corresponding to $\beta$.
	Then by a result of Kempf \cite[Appendix]{MR282975} (see also \cite[Theorem 2.1 (c)]{DebarreCurves} for the exact statement we need here) there exists a polarized abelian variety $(Y,\alpha')$ and a surjective morphism $f : \widehat{X} \to Y$ such that $f^{\ast}(\alpha') = \alpha$.
	Let $B = \widehat{Y}$ which embedds into $X$ via the dual morphism $j = \hat{f} : B \to X$. Consider again the commutative diagram
	\[
\begin{tikzcd}
	N_1(X) \ar{r}{\equiv} & \NS(\widehat{X}) \\
	N_1(B) \ar{u}{j_{\ast}} \ar{r}{\equiv} & \NS(Y) \ar{u}{\hat{j}^{\ast} = f^{\ast}}
\end{tikzcd}
\]
Let $\gamma \in N_1(B)$ the class corresponding to $\alpha'$. Then it follows $j_{\ast} \gamma = \beta$.
\end{proof} 

\begin{lemma} \label{lemma:rank vanishing}	If $\beta$ is of rank $u$, then
$\CC_{g,\chi}(\tau_{k_1 \ldots k_n}(\Gamma))$ vanishes for all $g < u$.
\end{lemma}
\begin{proof}
If $f:C \to X$ is a stable map of genus $g$, then the image of $f$ is contained in (a translate of) an abelian subvariety of dimension $\leq g$. Indeed, it is contained in the image of the Jacobian of the normalization of $C$.
	Thus the rank of $f_{\ast}[C]$ is $\leq g$ by Lemma~\ref{lemma:cbar comparison}.
\end{proof}

\subsection{Decomposing the Gromov-Witten classes according to Noether-Lefschetz strata}
\label{subsec:NLdecomp}
Let $f : C \to X$ be a $n$-marked genus $g$ stable map with characteristic polynomial $\chi$ of rank $u$, where $(X,\theta)$ is principally polarized abelian variety of dimension $h$. We assume that $n>0$.

By the proof of Lemma~\ref{lemma:rank vanishing} we must have $u \leq g$.

Let $B \subset X$ be the smallest abelian subvariety such that $f$ factors through a translate of $B$.
Then $B$ is of dimension $u$. Let $(\delta_1,\ldots,\delta_u)$ be the type of $\theta|_{B}$.
Let $\Mbar_{g,n}(\CX,\chi)_{\delta}$ be the (open and closed) component\footnote{The type $\delta$ of $\theta|_{B}$ is constant in connected families.} of $\Mbar_{g,n}(\CX,\chi)$ such that $\theta|_{B}$ has type $\delta$.
Associating to a stable map the subvariety $B$ induces a morphism to the Noether-Lefschetz locus:\footnote{Here if $u \leq h/2$ this is the standard definition from Section~\ref{subsec:NL loci}.
	If $u > h/2$ we require $\delta = (1^{2h-u},\delta'_1,\ldots,\delta'_u)$ 
	and set $\CA_{u,h-u,\delta} := \CA_{h-u,u,\tilde{\delta}}$
	where $\tilde{\delta} = (\delta'_1,\ldots,\delta'_u)$.}
\[ \rho : \Mbar_{g,n}(\CX,d)_{\delta} \to \CA_{u,h-u,\delta}. \]
Consider the pullback
$\varphi_{h,\delta}^{\prime \ast} \pi : \varphi_{h,\delta}^{\prime \ast} \CX \to \CA_{u,h-u,\delta}$ of the universal family $\pi : \CX\to \CA_h$.
Then we have an isomorphism
\begin{equation} \label{isom moduli 1}
\Mbar_{g,n}(\pi,\chi)_{\delta} \cong \Mbar_{g,n}( \varphi_{h,\delta}^{\prime \ast} \pi, \chi)_{\mathrm{main}}.
\end{equation}
where the subscript '$\mathrm{main}$' stands for taking the (open and closed) component which parametrizes stable maps $(f : C \to X,p_1,\ldots,p_n)$ over a point $(X,\theta,B)$ such that $f$ maps to a translate of $B$ (since $f_{\ast}[C]$ has rank $u$ this implies that $B$ is precisely the smallest abelian subvariety such that $f$ factors through a translate of it).

Recall the gluing morphism
\[ \varphi_{h,\delta} : \CA_{u,\delta}^{\mathrm{lev}} \times \CA_{h-u,\delta}^{\mathrm{lev}} \to 
\CA_{u,h-u,\delta} \]
which is finite and \'etale.
Consider the pullback
$\varphi_{h,\delta}^{\ast} \pi: \varphi_{h,\delta}^{\ast}\CX \to  \CA_{u,\delta}^{\mathrm{lev}} \times \CA_{h-u,\delta}^{\mathrm{lev}}$ of the universal family $\pi$ via $\varphi_{h,\delta}$.
We obtain a fiber diagram:
\[
\begin{tikzcd}
	\Mbar_{g,n}(\varphi_{h,\delta}^{\ast} \pi,\chi)_{\mathrm{main}} \ar{r}{\varphi_{h,\delta}} \ar{d}{\rho} & \Mbar_{g,n}(\CX,\chi)_{\mathrm{main}} \ar{d}{\rho} \\
	\CA_{u,\delta}^{\mathrm{lev}} \times \CA_{h-u,\tilde{\delta}}^{\mathrm{lev}} \ar{r}{\varphi_{h,\delta}} &
	\CA_{u,h-u,\delta}.
\end{tikzcd}
\]
where the subscript '$\mathrm{main}$' stands for the component where the stable maps $f$ over a point $((B,\theta_B,f),(C,\theta_C,f'))$ map to a translate of $B$.
By base change the top horizontal map is again finite and \'etale.

Consider the projections
\[
p_{\delta} : \CA_{u,\delta}^{\mathrm{lev}} \times \CA_{h-u,\tilde{\delta}}^{\mathrm{lev}} \to \CA_{u,\delta},
\quad
p_{\tilde{\delta}} : \CA_{u,\delta}^{\mathrm{lev}} \times \CA_{h-u,\tilde{\delta}}^{\mathrm{lev}} \to \CA_{h-u,\tilde{\delta}}
\]
and let $\pi_{u,\delta} : \CX_{u,\delta} \to \CA_{u,\delta}$ and $\pi_{h-u,\tilde{\delta}} : \CX_{h-u,\tilde{\delta}} \to \CA_{h-u,\tilde{\delta}}$
be the universal families.
There is a canonical short exact sequence of families of abelian varieties
\[ 0 \to p_{\delta}^{\ast}\CX_{u,\delta} \to \varphi_{h,\delta}^{\ast}\CX
\xrightarrow{\epsilon} p_{\tilde{\delta}}^{\ast} \CX_{h-u,\tilde{\delta}}/\Lambda \to 0
\]
where we think of $\Lambda$ here as a constant local system embedded via the level structure.

We may view
$\epsilon : \varphi_{h,\delta}^{\ast}\CX
\to p_{\tilde{\delta}}^{\ast} \CX_{h-u,\tilde{\delta}}/\Lambda$
as a family of $u$-dimensional polarized abelian varieties polarized of type $\tilde{\delta}$ by the restriction of  $\varphi_{h,\delta}^{\ast}(\theta)$ to the fibers of $\epsilon$.
We can hence form the moduli space of stable maps $\Mbar_{g,n}(\epsilon,\chi')$
of stable maps with characteristic polynomial $\chi'$ with respect to the fiberwise polarization.
Since all stable maps in $\Mbar_{g,n}(\varphi_{h,\delta}^{\ast} \pi,\chi)_{\mathrm{main}}$ are contracted by the quotient map $\epsilon$ we have an isomorphism of moduli spaces:
\begin{equation} \label{ABC} 
\Mbar_{g,n}(\varphi_{h,\delta}^{\ast} \pi,\chi)_{\mathrm{main}}
\cong \Mbar_{g,n}(\epsilon, \chi/t^{h-u} )
\end{equation}
where we use the characteristic polynomial $\chi' = \chi/t^{h-u}$ here by Lemma~\ref{lemma:cbar comparison}. We do not need the subscript $\mathrm{main}$ here on the right since the characteristic polynomial has rank $u$.

To simplify further, we use the fiber diagram
\[
\begin{tikzcd}
	p_{\delta}^{\ast} \CX_{u,\delta} \times p_{\tilde{\delta}}^{\ast} \CX_{h-u,\tilde{\delta}} \ar{r}{\zeta} \ar{d}{\widetilde{\epsilon}} & \varphi_{h,\delta}^{\ast} \CX \ar{d}{\epsilon} \\
	p_{\tilde{\delta}}^{\ast} \CX_{h-u,\tilde{\delta}} \ar{r} & p_{\tilde{\delta}}^{\ast} \CX_{h-u,\tilde{\delta}}/\Lambda
\end{tikzcd}
\]
where $\zeta$ is again finite and \'etale of degree $|\Lambda| = |K(\delta)|$. On the level of moduli spaces of stable maps we get hence the fiber diagram
\[
\begin{tikzcd}
\Mbar_{g,n}(\widetilde{\epsilon}, \chi/t^{h-u} ) \ar{r}{\zeta} \ar{d} 
& \Mbar_{g,n}(\epsilon, \chi/t^{h-u} ) \ar{d} \\
p_{\tilde{\delta}}^{\ast} \CX_{h-u,\tilde{\delta}} \ar{r}{\zeta} 
& p_{\tilde{\delta}}^{\ast} \CX_{h-u,\tilde{\delta}}/\Lambda
\end{tikzcd}
\]
The advantage of the family $\widetilde{\epsilon}$ is that its just the base change of the universal family $\pi_{u,\delta}: \CX_{u,\delta} \to \CA_{u,\delta}$. Hence we find the same for the moduli spaces of stable maps, that is we get a square of fiber diagrams
\begin{equation} \label{fiber diagram square}
\begin{tikzcd}
\Mbar_{g,n}(\widetilde{\epsilon},\chi') \ar{rr} \ar{d} &  & 
\Mbar_{g,n}(\pi_{u,\delta}^{\mathrm{lev}},\chi') \ar{d} \ar{r}{\mathrm{Fg}} & \Mbar_{g,n}(\pi_{u,\delta},\chi') \ar{d}\\
p_{\tilde{\delta}}^{\ast} \CX_{h-u,\tilde{\delta}} \ar{r} & \CA_{u,\delta}^{\mathrm{lev}} \times \CA_{h-u,\tilde{\delta}}^{\mathrm{lev}} \ar{r}{p_{\delta}} & \CA^{\mathrm{lev}}_{u,\delta} \ar{r}{\mathrm{Fg}}  & \CA_{u,\delta},
\end{tikzcd}
\end{equation}
where we write
$\pi_{u,\delta}^{\mathrm{lev}}: \CX_{u,\tilde{\delta}}^{\mathrm{lev}} \to \CA_{u,\delta}^{\mathrm{lev}}$ for the universal family over $\CA_{u,\delta}^{\mathrm{lev}}$, which is just the pullback of the universal family over $\CA_{u,\delta}$ under the morphism $\mathrm{Fg}$ forgetting the level structure. In other words, since $p_{\delta}$ is just the projection to the factors,
\begin{equation} \label{product decomp}
\Mbar_{g,n}(\widetilde{\epsilon},\chi') \cong \Mbar_{g,n}(\pi_{u,\delta}^{\mathrm{lev}},\chi') \times \CX_{\tilde{\delta}}^{\mathrm{lev}}.
\end{equation}

Thus we have found a finite and etale map 
\[
\alpha = \varphi_{h,\delta} \circ \widetilde{\zeta} : 
 \Mbar_{g,n}(\pi_{u,\delta}^{\mathrm{lev}},\chi') \times \CX_{\tilde{\delta}}^{\mathrm{lev}}
 \to \Mbar_{g,n}(\pi,\chi)_{\delta}.
\]

Since the map $\alpha$ is \'etale the perfect obstruction theory on $\Mbar_{g,n}(\pi,\chi)_{\delta}$ pulls back to a perfect obstruction theory on the source.
The next step is to identify the corresponding virtual tangent bundle.
This happens essentially in two steps.

In the first step we compare the $K$-theory classes of the natural perfect obstruction theories
under the isomorphism \eqref{isom moduli 1}.
As discussed in Appendix \ref{app:subsec reduced} the obstruction sheaf of $\Mbar_{g,n}(\pi,\chi)_{\delta}$ is obtained from removing a copy of the normal bundle $N = \BE_{u}^{\vee} \otimes \BE_{h-u}^{\vee}$ of $\varphi'_{h,\delta} : \CA_{u,h-u,\delta} \to \CA_{h}$ from the obstruction sheaf of $\Mbar_{g,n}(\varphi_{h,\delta}^{\prime \ast} \pi, \chi)_{\mathrm{main}}$. In $K$-theory this yields, see \eqref{Ktheory compare2},
\[
T^{\vir}_{\Mbar_{g,n}(\pi,\chi)_{\delta}} = T_{\Mbar_{g,n}(\varphi_{h,\delta}^{\prime \ast} \pi,\chi)_{\mathrm{main}}}^{\vir} + \BE_{u}^{\vee} \otimes \BE_{h-u}^{\vee}.
\]
Pulling back via $\varphi_{h,\delta}$ we obtain
\begin{equation} \label{arewfef}
\varphi_{h,\delta}^{\ast} T^{\vir}_{\Mbar_{g,n}(\pi,\chi)_{\delta}} = T_{\Mbar_{g,n}(\varphi_{h,\delta}^{\ast} \pi,\chi)_{\mathrm{main}}}^{\vir} + \BE_{u}^{\vee} \otimes \BE_{h-u}^{\vee}.
\end{equation}

Next we compare the virtual tangent bundles under the isomorphismn \eqref{ABC}
which brings us exactly in the situation discussed in Section~\ref{app:subsec on comparing two relative}.
By \eqref{Ktheory compare1} we have
\[
T^{\vir}_{\Mbar_{g,n}(\varphi_{h,\delta}^{\ast} \pi,\chi)_{\mathrm{main}}}
= T^{\vir}_{\Mbar_{g,n}(\epsilon, \chi/t^{h-u} )} - \BE_{\Mbar_{g,n}}^{\vee} \otimes \BE_{h-u}^{\vee}
\]

Combining with \eqref{arewfef} we hence get
\[
\varphi_{h,\delta}^{\ast} T^{\vir}_{\Mbar_{g,n}(\pi,\chi)_{\delta}}
=
T^{\vir}_{\Mbar_{g,n}(\epsilon, \chi/t^{h-u} )} + \BE_{u}^{\vee} \otimes \BE_{h-u}^{\vee}
 - \BE_{\Mbar_{g,n}}^{\vee} \otimes \BE_{h-u}^{\vee}
\]
Pulling back by $\widetilde{\zeta}$ we hence have 
%and using \eqref{product decomp} hence
\begin{equation} \label{some rel}
\alpha^{\ast} T^{\vir}_{\Mbar_{g,n}(\pi,\chi)_{\delta}}
=
T^{\vir}_{\Mbar_{g,n}(\widetilde{\epsilon}, \chi/t^{h-u} )} + \BE_{u}^{\vee} \otimes \BE_{h-u}^{\vee}
- \BE_{\Mbar_{g,n}}^{\vee} \otimes \BE_{h-u}^{\vee}.
\end{equation}

Because the image of the stable maps in $\Mbar_{g,n}(\widetilde{\epsilon}, \chi/t^{h-u} )$ do not admit maps to translates of proper abelian subvarieties of $B$ pulling back differentials by the stable maps yields the inclusion
\[ \BE_{u} \hookrightarrow \BE_{\Mbar_{g,n}} \]
over the moduli space $\Mbar_{g,n}(\widetilde{\epsilon}, \chi/t^{h-u} )$.
Dually this yields the exact sequence
\[
0 \to (\BE_{\Mbar_{g,n}}/\BE_{u})^{\vee} \to 
\BE_{\Mbar_{g,n}}^{\vee} \to E_{u}^{\vee} \to 0.
\]
The equation \eqref{some rel} hence can be rewritten
\[
\alpha^{\ast} T^{\vir}_{\Mbar_{g,n}(\pi,\chi)_{\delta}}
=
T^{\vir}_{\Mbar_{g,n}(\widetilde{\epsilon}, \chi/t^{h-u} )} 
- (\BE_{\Mbar_{g,n}} / \BE_u)^{\vee} \otimes \BE_{h-u}^{\vee}
\]
Thus the pullback perfect obstruction theory may intuitively be viewed as removing a factor of 
$(\BE_{\Mbar_{g,n}} / \BE_u)^{\vee} \otimes \BE_{h-u}^{\vee}$ from the obstruction sheaf.\footnote{To make this rigorous requires a more careful study of the corresponding comparison maps. We only require the $K$-theory statements here, so we do not attempt to give such a discussion.}
Using Siebert's formula \eqref{Sieberts formula} we find the identity of virtual classes
\[
\alpha^{\ast} [ \Mbar_{g,n}(\pi,\chi)_{\delta}]^{\vir} 
=
e\left( (\BE_{\Mbar_{g,n}} / \BE_u)^{\vee} \otimes \BE_{h-u}^{\vee} \right) 
\cap [\Mbar_{g,n}(\widetilde{\epsilon}, \chi/t^{h-u} )]^{\vir}
\]

The moduli space $\Mbar_{g,n}(\widetilde{\epsilon}, \chi' )$ is just the product of
$\Mbar_{g,n}(\pi^{\mathrm{lev}}_{u,\delta}, \chi')$ with the smooth space $\CX^{\mathrm{lev}}_{h-u,\tilde{\delta}}$. Thus we have
\[
[\Mbar_{g,n}(\widetilde{\epsilon}, \chi' )]^{\vir}
=
[ \Mbar_{g,n}(\pi^{\mathrm{lev}}_{u,\delta}, \chi') ]^{\vir} \times [ \CX^{\mathrm{lev}}_{h-u,\tilde{\delta}} ].
\]
Moreover, using the right square fiber diagram \eqref{fiber diagram square} and that $\mathrm{Fg}$ is finite \'etale we have
\begin{equation} \label{efwefwe-} [ \Mbar_{g,n}(\pi^{\mathrm{lev}}_{u,\delta}, \chi') ]^{\vir}  = 
\mathrm{Fg}^{\ast} [ \Mbar_{g,n}(\pi_{u,\delta}, \chi') ]^{\vir}. \end{equation}

Inserting this we thus arrive at
\[
\alpha^{\ast} [ \Mbar_{g,n}(\pi,\chi)_{\delta}]^{\vir} 
=
(-1)^{(g-u)(h-u)}
e\left( (\BE_{\Mbar_{g,n}} / \BE_u) \otimes \BE_{h-u} \right)
\cap 
\mathrm{Fg}^{\ast} [ \Mbar_{g,n}(\pi_{u,\delta}, \chi/t^{h-u}) ]^{\vir} \times [ \CX^{\mathrm{lev}}_{h-u,\tilde{\delta}} ]
\]
Pushing forward we get the explicit description:
\begin{prop}(Noether-Lefschetz decompositon)
	Let $\chi$ be a characteristic polynomial of rank $u$. Then
\begin{multline} \label{full rank decomposition}
[ \Mbar_{g,n}(\pi,\chi)_{\delta}]^{\vir} 
	=
\sum_{\substack{\delta = (\delta_1,\ldots,\delta_u) \\ \delta_u > 0}}
\frac{(-1)^{(g-u)(h-u)}}{|\deg \varphi_{h,\delta}| \cdot |K(\delta)|}  \\
	\cdot \alpha_{\ast}\left( 
	e\left( (\BE_{\Mbar_{g,n}} / \BE_u) \otimes \BE_{h-u} \right)
	\cap 
	\mathrm{Fg}^{\ast} [ \Mbar_{g,n}(\pi_{u,\delta}, \chi/t^{h-u}) ]^{\vir} \times [ \CX^{\mathrm{lev}}_{h-u,\tilde{\delta}} ]
	\right)
\end{multline}
\end{prop}

We hence have reduced the computations of the contributions from the rank $u$ stable maps to the study of the Gromov-Witten classes of $\CA_{u,\delta}$.

	\subsection{Vanishing relations II} \label{subsec:vanishings II}
As a first consequence of \eqref{full rank decomposition} we will prove some vanishing results.
We start with Gromov-Witten invariants in genus $g=1,2,3$.
\begin{prop} Let $h \geq 2$ and $\Gamma \in R_{\mathrm{vert}}^{\ast}(\CX^{\times n})$.
Then for all characteristic polynomials $\chi$ we have
		\begin{enumerate}
			\item[(i)] $\CC_{1,\chi}(\tau_{k_1 \ldots k_n}(\Gamma))=0$ if $\deg(\Gamma) + \sum_i k_i > h+n-1$.
			\item[(2)] $\taut\ \CC_{2,\chi}(\tau_{k_1 \ldots k_n}(\Gamma))=0$ if $\deg(\Gamma) + \sum_i k_i > h+n$ and $h \geq 2$.
			\item[(3)] $\taut\ \CC_{3,\chi}(\tau_{k_1 \ldots k_n}(\Gamma))=0$ if $\deg(\Gamma) + \sum_i k_i > h+n$ and $h \geq 3$.
		\end{enumerate}
	\end{prop}
	\begin{proof}
Assume first that the rank of $\chi$ is positive.
		We will use the decomposition \eqref{full rank decomposition} and an analysis of the virtual dimension.
		Let $u$ be the rank of $\chi$.
		Capping \eqref{full rank decomposition} with $\ev^{\ast}(\Gamma)$ and using the push-pull formula gets
		\begin{multline} \label{full rank decomposition2}
			\CC_{g,\chi}(\Gamma)
			=
			\sum_{\delta=(\delta_1,\ldots,\delta_u)} 
			\frac{(-1)^{(g-u)(h-u)}}{|\deg \varphi_{h,\delta}| \cdot |K(\delta)|} \\
			\varphi_{h,\delta \ast}
			(\rho_{\delta} \times \pi^{\mathrm{lev}}_{h-u,\tilde{\delta}})_{\ast}
			\left(  \prod_i \psi_i^{k_i} e\left( (\BE_{\Mbar_{g,n}}/\BE_{u}) \otimes \BE_{h-u} \right)
			(\zeta^{\times n} \circ (\ev \times \Delta))^{\ast}(\Gamma)
			[ \Mbar_{g,n}(\pi_{u,\delta}^{\mathrm{lev}},\chi/t^{h-u}) ]^{\vir} \times [\CX^{\mathrm{lev}}_{h-u,\tilde{\delta}} ] \right).
		\end{multline}
	where
\begin{itemize}
\item  $\rho_{\delta} : \Mbar_{g,n}(\widetilde{\epsilon},\chi/t^{h-u}) \to \CA_{u,\delta}^{\mathrm{lev}}$ is the base map,
\item $\Delta : \CX_{h-u,\tilde{\delta}}^{\mathrm{lev}} \to (\CX_{h-u,\tilde{\delta}}^{\mathrm{lev}})^{\times n}$ is the diagonal morphism
\end{itemize}
		
		If $\theta_{\delta}$, $\theta_{\tilde{\delta}}$ are the theta divisors on $\CX_{\delta}$, $\CX_{\tilde{\delta}}$, we have
		\[ \zeta^{\ast}(\theta) = \theta_{\delta} + \theta_{\tilde{\delta}}. \]
		 It follows that
		\[ (\zeta \times \zeta)^{\ast}(\eta) = \eta_{\delta} + \eta_{\tilde{\delta}} \]
		where 
		\[ \eta_{\delta} := \frac{1}{2}( \mathsf{m}^{\ast}(\theta_{\delta}) - (\theta_{\delta})_1 - (\theta_{\delta})_2 ) , \quad \eta_{\tilde{\delta}} = \frac{1}{2}( \mathsf{m}^{\ast}(\theta_{\tilde{\delta}}) - (\theta_{\tilde{\delta}})_1 - (\theta_{\tilde{\delta}})_2 )  \]
		are defined via $\theta_{\delta}, \theta_{\tilde{\delta}}$.
		The pullback $(\zeta^{\times n})^{\ast}(\Gamma)$ is hence a linear combination of elements $\Gamma_1 \otimes \Gamma_2$ in $\Rvert(\CX_{u,\delta}^{\times n}) \otimes \Rvert(\CX_{h-u,\tilde{\delta}}^{\times n})$.
		In the pushforward \eqref{full rank decomposition2} only those terms contribute where the degree of $\Gamma_2$ 
		%in $\Rvert(\CX_{\tilde{\delta}}^{\times n})$
		is $h-u$ (since $\theta_{\tilde{\delta}}^{h-u+1} = 0$), and so we may assume that 
		\begin{equation} \deg(\Gamma_1) = \deg(\Gamma) - (h-u). \label{deg gamma1} \end{equation}
		
		Moreover, by \eqref{efwefwe-} it suffices to work with $\pi_{u,\delta} : \CX_{u,\delta} \to \CA_{u,\delta}$ (not over $\CA_{u,\delta}^{\mathrm{levl}}$).
		
		It hence suffices to show the vanishing of
		\begin{equation} \label{pushforward class} I = (\rho_{\delta})_{\ast} \left( \prod_i \psi_i^{k_i} \cdot e\left( (\BE_{\Mbar_{g,n}}/\BE_{u}) \otimes \BE_{h-u} \right)  \ev^{\ast}(\Gamma_1) \cap [\Mbar_{g,n}(\pi_{u,\delta},\chi/\chi')]^{\vir} \right)
		\end{equation}
		after tautological projection 		where 
$\rho_{\delta} : \Mbar_{g,n}(\pi_{u,\delta},\chi/t^{h-u}) \to \CA_{u,\delta}$ is the forgetful morphism.

		Recall that $u \leq g$ here. We consider the cases $g \in \{ 1,2, 3 \}$ one by one:
		
		($g=1$) We have $u=g=1$ and $\delta \geq 1$. In particular, the Euler factor $e( (\BE_{\Mbar_{g,n}}/\BE_{u}) \otimes \BE_{h-u})$ does not contribute. The moduli space $\CA_{1,\delta} \cong \CA_{1}$ is a finite quotient of $\BA^1$. Thus the class \eqref{pushforward class} is a multiple of the fundamental class
		and can be computed by restricting to a fiber $E$ of a point. 
		The fiber $E$ is an elliptic curve here.
		Thus for it to be non-zero the integral
		\[ \int_{[\Mbar_{1,n}(E,d/\delta)]^{\vir}} \ev^{\ast}(\Gamma_1) \prod_i \psi_i^{k_i} \]
		must be non-zero.
		Since $\Mbar_{1,n}(E,d/\delta)$ has virtual dimension $n$, this can only happen if the dimension constraint
		\[ n = \deg(\Gamma_1) + \sum_i k_i = \deg(\Gamma) - (h-1) + \sum_i k_i \]
		or equivalently
		\[ \deg(\Gamma) + \sum_i k_i = n+h-1 \]
		is satisfied. This shows (1).
		
		($g=2$) Here we have two cases $u=1$ and $u=2$.
		In the case $u=1$ the Euler factor will contribute a factor of $\mu_1=c_1(\BE_{\Mbar_{g,n}})$, so if the pushforward is non-zero, we require the non-vanishing of
		\[ \int_{[\Mbar_{2,n}(E,d/\delta)]^{\vir}} \mu_1 \ev^{\ast}(\Gamma_1) \prod_i \psi_i^{k_i} \]
		which forces us to have
		\[ 2 + n = 1 + \deg(\Gamma_1) + \sum_i k_i = 1 + (\deg(\Gamma) - (h-1)) + \sum_i k_i \]
		and hence
		\[ n+h = \deg(\Gamma) + \sum_i k_i. \]
				
In the case $u=2$, 
if $\deg(\Gamma) + \sum_i k_i > h+n$, then by \eqref{deg gamma1} we get
\[
\deg I = \deg(\Gamma_1) + \sum_i k_i - n - 1 > 1.
\]
Since the socle is in degree $1$ (see \cite[Theorem 1]{AitorNL} for the description of the tautological ring of $\CA_{h,\delta}$ also in the non-principally polarized case), 
we find that $\taut (I) = 0$.
The claim hence follows since
by extending $\varphi_{h,\delta}$ to a map $\overline{\varphi_{h,\delta}}$ on a compactification as in \cite[Section 3.1]{AitorNL}
and since $\lambda$-classes pullback to $\lambda$-classes under $\overline{\varphi_{h,\delta}}$ it is straightforward to show following the ideas of \cite[Section 3.1]{AitorNL} that
for all $\gamma \in \Chow^{\ast}(\CA_{u,\delta})$ and $\widetilde{\gamma} \in R^{\ast}(\CA_{h-u,\tilde{\delta}})$ we have
\begin{equation} \label{taut vs taut} \taut\ (\varphi_{h,\delta})_{\ast}( p_{\delta}^{\ast}(\gamma) \cdot p_{\tilde{\delta}}^{\ast}(\widetilde{\gamma}))  = \taut\ (\varphi_{h,\delta})_{\ast}( p_{\delta}^{\ast}(\taut\ \gamma) \cdot p_{\tilde{\delta}}^{\ast}(\widetilde{\gamma})), \end{equation}
where the tautological projection on $\Chow^{\ast}(\CA_{u,\delta})$ is defined
by a straightforward generalization of \cite{CMOP} (see e.g. \cite{AitorNL} for the methods).

($g=3$) The case is very similar to the genus $2$ case and omitted: In case $u=1$ it follows by the virtual dimension of $\Mbar_{g,n}(E,d)$. In case $u=2$ and $u=3$ one finds that the Euler factors make the degree of $I$ larger than the socle dimension (which is $1$ and $3$ respectively.\\

It remains to consider the case $\chi = t^h$, so a curve class of degree zero.
We have
\[ [\Mbar_{g,n}(\pi,t^h) ]^{\vir} = e(\BE_h^{\vee} \otimes \BE_{\Mbar_{g,n}}^{\vee} ) \cap [\Mbar_{g,n}] \times [\CX]. \]
The evaluation map is the diagonal morphism $\Delta : \CX \to \CX^{\times n}$.
Then we find
\[
\CC_{g,t^h}( \tau_{k_1 \cdots k_n}(\Gamma))
=
(-1)^{gh} \pr_{2 \ast}\left( \prod_i \psi_i^{k_i} \cdot e( \BE_{h} \otimes \BE_{\Mbar_{g,n}}) \right)  \pi_{\ast}(\Delta^{\ast}(\Gamma))
\]
where $\pr_2 : \Mbar_{g,n} \times \CA_h \to \CA_h$ is the projection.

For $\pi_{\ast}(\Delta^{\ast}(\Gamma))$ to be non-zero,
$\Delta^{\ast}(\Gamma)$ must have have multiplicity $2h$ by Example~\ref{example:pushforward lemma},
so $\Gamma$ must be of multiplicity $2h$ by Example~\ref{example:multiplicity under homomorphism}.
Since $\Gamma$ is tautological, it follows that $\deg(\Gamma) = h$.

Assume now that $g=3$ and $h \geq 3$. By Lemma~\ref{lemma:e(Hodge Hodge) eval} we have
\[ \pr_{2 \ast}\left( \prod_i \psi_i^{k_i} \cdot e( \BE_{h} \otimes \BE_{\Mbar_{g,n}}) \right) 
=
\left( \int_{\Mbar_{3,n}} \mu_3 \mu_2 \mu_1 \prod_i \psi_i^{k_i} \right)  \lambda_{h-1} \lambda_{h-2} \lambda_{h-3}. \]
Since $\Mbar_{3,n}$ has dimension $6+n$ this can be non-zero only if $\sum_i k_i = n$.
Thus if $\deg(\Gamma) + \sum_i k_i > h+n$ also the degree zero invariant vanishes.
The cases $g=1$ and $g=2$ are parallel.
	\end{proof}

	\begin{prop}
Assume that $g,h \geq 4$, $\Gamma \in \Rvert(\CX^{\times n})$.Then $\taut\ \CC_{g,\chi}(\tau_{k_1\ldots k_n}(\Gamma)) = 0$ for all $\chi$.
	\end{prop}
	\begin{proof}
The case where $\chi$ has rank $0$ follows from the vanishing of $e(\BE_{\Mbar_{g,n}} \otimes \BE_h)$ of Lemma~\ref{lemma:e(Hodge Hodge) eval}.
If the rank $u$ is positive, then $u \leq g$ and we use the decomposition \eqref{full rank decomposition}.
By the same dimension analysis as in Corollary~\ref{cor:Vanishing for g >= h >= 4} and using \eqref{taut vs taut} we can have non-zero terms only for $u=1,2,3$.
Then a case-by-case analysis as in the last proposition shows that all invariants in these cases vanish.
	\end{proof}

We consider the remaining pairs $(g,h)$ with $g \geq 2$ that we have not yet considered:
\begin{lemma}
If $h=2$ and $g \geq 2$, then $\CC_{g,\chi}(\tau_{k_1 \ldots k_n}(\Gamma))=0$ unless $\deg(\Gamma) + \sum_i k_i = g+n$.
\end{lemma}
\begin{proof}
The Chow ring of $\CA_2$ is $\BQ[\lambda_1]/(\lambda_1^2)$.
If $\deg(\CC_g(\tau_{k_1 \ldots k_n}(\Gamma))) = 0$, then $\CC_g(\tau_{k_1 \ldots k_n}(\Gamma))$ is a multiple of the fundamental class. The multiple can be evaluated by restricting to a fiber.
However, all (non-reduced) Gromov-Witten invariants of an abelian surface are zero in $g \geq 2$.
\end{proof}

\begin{lemma}
If $h=3$, then $\CC_{g,\chi}(\tau_{k_1 \ldots k_n}(\Gamma))=0$ unless $\sum_i k_i + \deg(\Gamma) = h+n$.
\end{lemma}
\begin{proof}
By Lemma~\ref{lemma:vanishing1b} for non-vanishing we must have $\sum_i k_i + \deg(\Gamma) \geq 3+n$.
Thus $\deg_{\BC} \CC_g(\tau_{k_1 \ldots k_n}(\Gamma)) = \sum_i k_i + \deg(\Gamma) -n \geq 3$.
But $\Chow^i(\CA_3) = 0$ if $i \geq 4$, see \cite{MR1642753}, so we must have equality.
\end{proof}
%	\begin{rmk}
%		For $h=2$ we have that $\CC_g(\tau_{k_1 \ldots k_n}(\Gamma))=0$ unless $\deg(\Gamma) + \sum_i k_i = g+n$.
%	\end{rmk}

Combining the vanishing statements above we have the following uniform description:
\begin{cor} \label{cor:dim constraint}
Let $h \geq 3$ and let $\Gamma \in \Rvert(\CX^{\times n})$.
\begin{itemize}
	\item[(i)] If $g=1$, then $\taut\ \CC_{g,\chi}(\tau_{k_1 \ldots k_n}(\Gamma))$ is non-zero only if $\sum_i k_i + \deg(\Gamma) = h + n - 1$.
\item[(ii)] If $g \geq 2$, then $\taut\ \CC_{g,\chi}(\tau_{k_1 \ldots k_n}(\Gamma))$ is non-zero only if $\sum_i k_i + \deg(\Gamma) = h + n$. 
\end{itemize}
\end{cor}

\section{Genus $0$ and $1$ Gromov-Witten classes of $\CA_h$}
As a first example we consider here the Gromov-Witten classes in genus $0$ and $1$.

\subsection{g=0} \label{subsec:genus zero class}
In the case of genus $0$ stable maps (which are all constant) we have
\[ \CC_0(\Gamma) = \pr_2^{\ast} \pi_{\ast}(\Delta^{\ast}(\Gamma)) \]
where $\Delta : \CX \to \CX^{\times n}$ is the diagonal morphism and $\pr_2: \Mbar_{0,n} \times \CA_h \to \CA_h$ is the projection.
It follows by Examples~\ref{example:multiplicity under homomorphism} and~\ref{example:pushforward lemma} that $\CC_0(\Gamma)$ is zero unless $\mult(\Gamma) = \mult(\Delta^{\ast}(\Gamma)) = 2h$, which 
confirms Conjectures~\ref{conj:quasimodularity} and~\ref{conj:HAE}.

\subsection{g=1}
Consider for $h > 0$ the codimension $h-1$ class
\[ \CC_{1}(\tau_0(\sigma)) \in \Chow^{h-1}(\Mbar_{1,1} \times \CA_h), \quad \sigma = \frac{\theta^h}{h!}. \]
The contribution from the constant maps is
\[ \CC_{1,0}(\tau_0(\sigma)) = \frac{1}{24} (-1)^h \lambda_{h-1}. \]
Since the Noether-Lefschetz divisors $\NL_{h,\delta}$ are also of codimension $h-1$,
$\CC_{1,d}(\tau_0(\sigma))$ is a linear combinations of the classes $\NL_{h,\delta}$ for some $\delta|d$.
In fact, the analysis \eqref{full rank decomposition} yields the
that the coefficient of $\NL_{h,\delta}$ is the Gromov-Witten invariant of the elliptic curve $E$ given by
\[ \langle \tau_0(\sigma) \rangle_{1,d/\delta}^{E} = \sigma(d/\delta). \]
Hence we obtain
\[ \CC_{1}(\tau_0(\sigma)) = \frac{1}{24} (-1)^h \lambda_{h-1}
+ \sum_{d \geq 1} \sum_{\delta|d} \sigma(d/\delta) [\NL_{h,\delta}]. \]
Conjecture~\ref{conj:quasimodularity} yields that this is a quasimodular form of weight $2h$,
in agreement with \cite{AitorNL,GreerLian}.

We can also apply the holomorphic anomaly equation,
where all terms vanish:
$\CC_0(\sigma, f_{\theta}) = 0$ because $\mult(\sigma \boxtimes f_{\theta}) = 2h + (2h-2) > 2h$,
and $\CC_1(f_{\theta} \sigma) = 0$ since $f_{\theta}(\sigma)$ has too low multiplicity.

%
%\begin{rmk}
%	For the $g=1$ case one has the identity
%	\[ \sum_{\delta | d} \sigma_1(d/\delta) \delta^{2h-1} \prod_{p|\delta} (1-p^{2-2h}) = \sigma_{2h-1}(d). \]
%\end{rmk}

\section{Modular forms and the Shimura lift}
We introduce the Weil representation following \cite{OWilliams, GWNL}.

\subsection{Modular forms for the Weil representation}
%The metaplectic group $\mathrm{Mp}_2(\mathbb{R})$ is the group of pairs $(M, \phi(\tau))$, where $M = \begin{pmatrix} a & b \\c & d \end{pmatrix} \in \mathrm{SL}_2(\RR)$ and $\phi(\tau)$ is a holomorphic square root of $\tau \mapsto c \tau + d$ on the upper half-plane.
%The group structure on $\mathrm{Mp}_2(\mathbb{R})$ is defined by
%$$(M_1,\phi_1(\tau)) \cdot (M_2, \phi_2(\tau)) = (M_1 M_2, \phi_1(A_2 \tau) \phi_2(\tau)).$$
%The integral metaplectic group $\mathrm{Mp}_2(\mathbb{Z})$ is 
%the preimage of $\SL_2(\BZ)$ under the projection $\mathrm{Mp}_2(\mathbb{R}) \to \SL_2(\BR)$. It is generated by the elements
%\begin{align*}
%	S &= \left( \begin{pmatrix} 0 & -1 \\ 1 & 0 \end{pmatrix}, \sqrt{\tau} \right), \\
%	T & = \left( \begin{pmatrix} 1 & 1 \\ 0 & 1 \end{pmatrix}, +1 \right),
%\end{align*}
%where the square root $\sqrt{\tau}$ has positive imaginary part, modulo the relations 
%$S^8 = (ST)^6 = \id$.
Let $M$ be an even integral lattice of signature $(p, q)$ and rank $n = p+q$ and signature $\sigma = p-q$. The finite Weil representation $\rho$ associated to $M$ is the representation of the integral metaplectic group $\mathrm{Mp}_2(\mathbb{Z})$ on the group algebra $\mathbb{C}[M\dual/M] = \mathbb{C}[e_\gamma: \, \gamma \in M\dual/M]$, defined by
\begin{align*} 
	\rho(S) e_x & = \frac{e^{-\pi i \sigma / 4}}{\sqrt{|M\dual/M|}} \sum_{y \in M\dual/M} e^{-2\pi i \langle x, y \rangle} e_y \\ \rho(T) e_x & = e^{\pi i \langle x, x \rangle} e_x.
\end{align*}
where $S = \left( \begin{pmatrix} 0 & -1 \\ 1 & 0 \end{pmatrix}, \sqrt{\tau} \right)$ and $T  = \left( \begin{pmatrix} 1 & 1 \\ 0 & 1 \end{pmatrix}, +1 \right)$ are the standard generators.

For a half-integer $k \in \frac{1}{2} \BZ$ a modular form for the Weil representation $M$ is a holomorphic function $f : \BH \to \BC[M\dual/M]$, with $\BH=\{ \tau \in \BC | \mathrm{Im}(\tau) > 0 \}$ such that
\[ \forall g = (A,\phi(\tau)) \in \Mp_2(\BZ) \ : \ f(A \tau) = \phi(\tau)^{2k} \cdot \rho(g)(f(\tau)). \]
Since $e_{\gamma}$ for $\gamma \in M\dual/M$ is an eigenbasis for the action of $T$, we have a Fourier expansion
\[ f(\tau) = \sum_{\gamma \in M\dual/M} \sum_{\substack{n \in \BQ \\ n - \frac{1}{2} \langle \gamma, \gamma \rangle \in \BZ}} c(n,\gamma) q^{n} e_{\gamma},
\quad q=e^{2 \pi i \tau}. \]
%where $q=e^{2 \pi i \tau}$. 
We say $f$ is {\em holomorphic} if $c_{n,\gamma} = 0$ for all $n<0$.
Let $\Mod_k(\rho_M)$ be the vector space of holomorphic modular forms of weight $k$ for the Weil representation on $M$. The direct sum
\[ \Mod(\rho_M) = \bigoplus_{k} \Mod_k(\rho_M). \]
is then a module over the algebra of classical modular forms $\Mod= \BC[G_4,G_6]$.
The $\QMod$-module of quasimodular forms for the Weil representation is
\[ \QMod(\rho_M) = \Mod(\rho_M) \otimes_{\Mod} \QMod = \Mod(\rho_M) \otimes \BC [G_2]. \]

\begin{example} \label{example eisenstein series}
The main example we will need is the Eisenstein series $E_{5/2}^{(\delta)}$ of weight $5/2$ for the lattice
$M= (2 \delta)$, which is a holomorphic modular form for $\rho_M$, see \cite{BruinierKuss}.
We have
\[ M^{\vee}/M \cong ( \frac{1}{2 \delta} \BZ ) / \BZ \]
with intersection form $\langle a/2 \delta, b/2 \delta \rangle = ab/2\delta$.
%We write $e_{a/2 \delta}$ for the basis  vector of $\BC[M\dual/M]$ corresponding to $a/2 \delta$.

In case $\delta=1$ the Fourier expansion is
\[ E_{5/2}^{(1)} = 120 \sum_{d \geq 0} H(2,4d) q^d e_{0}
+ 120 \sum_{d \geq 0} H(2,4d+1) q^{d+\frac{1}{4}} e_{1/2} \]
where $H(2,N)$ is the Cohen function given in \cite[Prop.4.1]{Cohen} by
\[ H(2,N) = -\frac{1}{5} \sum_{s \in \BZ} \sigma_1\left( \frac{N - s^2}{4} \right) - \delta_{N \text{ square }} \frac{N}{10} \]
with the convention $\sigma_1(0)=\frac{1}{2} \zeta(-1) = -\frac{1}{24}$ and $\sigma_1(u) = 0$ if $u \notin \BZ_{\geq 0}$.
The Fourier coefficients of $E_{5/2}^{(\delta)}$ for $\delta>1$ are determined from the $\delta=1$ by Hecke operators,
see Theorem~\ref{thm:coeff of e_{5/2}delta}.
\end{example}

\subsection{Shimura lift}
Consider the lattice $M = A_1(\delta) = (2 \delta)$ for some $\delta \geq 1$
and identify $M^{\vee}/M$ with the quotient of $\frac{1}{2 \delta} \BZ$ by $\BZ$.
Consider a formal series
\[ F = \sum_{\gamma \in \frac{1}{2 \delta}\BZ / \BZ } \sum_{\substack{n \in \BQ \\ n - \delta \gamma^2 \in \BZ}} c(n,\gamma) q^n e_{\gamma}. \]
The (formal) {\em Shimura lift} of $F$ in weight $k$ is defined by
\begin{align*}
	\Lift_k(F) 
	%& = -\frac{c(0,0) B_k}{2k} + \sum_{a \geq 1} c\left(\frac{a}{2 \delta}, \frac{a^2}{4 \delta} \right) \sum_{\ell \geq 1} \ell^{k-1} q^{\ell a} \\
	& = -\frac{c(0,0) B_k}{2k} + \sum_{n \geq 1} \sum_{\ell|n} \ell^{k-1} c\left( \frac{n^2}{4 \delta \ell^2}, \frac{n}{2 \delta \ell} \right) q^n.
\end{align*}

\begin{thm} \label{thm:shimura lift}
	Let $k  \geq 2$. If $F \in \QMod_{k+\frac{1}{2}}(\SL_2(\BZ),\rho_M)$ is a quasimodular form of weight $k+ \frac{1}{2}$ for the Weil representation of $M$, then $\Lift_k(F)$ is a quasimodular form of weight $2k$ for
	the congruence subgroup  $\Gamma_0(\delta)$. 
	Moreover, if $k\geq 4$,  then
	\[
	\frac{d}{dG_2} \Lift_k(F) = D_{\tau} \Lift_{k-2}\left( \frac{d}{dG_2} F \right),
	\]
where $D_{\tau} = q \frac{d}{dq}$.
\end{thm}
\begin{proof}
	If $F$ is a modular form, this is the well-known Shimura lift viewed here as the theta lifting of Kudla-Milson under the special isomorphism $\mathfrak{o}(2,1) \cong \mathfrak{sl}_2$, see \cite[Section 14]{Borcherds}.
	The case of quasimodular form is treated in \cite{OWilliams} and the compatibility with $\frac{d}{dG_2}$ follows from \cite[Thm 5.9]{OWilliams}.
\end{proof}

\section{Genus 2 Gromov-Witten invariants of $\CA_h$}
The goal of this section is to compute the quotient Gromov-Witten invariants
in genus $2$ over $\CA_h$ for all $h \geq 2$,
and deduce the modularity of the associated generating series.

\subsection{Formula}
By Lemma~\ref{lemma:rank vanishing} it suffices to consider quotient Gromov-Witten invariants for characteristic polynomials of the form
\[ \chi = t^h - d t^{h-1} + a t^{h-2} = t^{h-2}( t^2 - d t + a) \]
where $d$ is the degree of the curve class and $0 \leq a \leq \frac{1}{4} d^2$.
Consider the Fourier expansion of the Eisenstein series $E_{5/2}^{(\delta)}$ introduced in Example~\ref{example eisenstein series}:
\[ E_{5/2}^{(\delta)} = \sum_{ \frac{r}{2 \delta} \in \frac{1}{2\delta} \BZ / \BZ} \sum_{\substack{n \in \BQ_{\geq 0} \\ n - \frac{r^2}{4 \delta} \in \BZ}} c^{(\delta)}(n,\gamma) q^{n} e_{\gamma} . \]

We will use the notation from Section~\ref{sec:NL and GW theory}.
In particular, for $h \geq 4$ and $\delta = (\delta_1,\delta_2)$ we have the diagram
	\[
\begin{tikzcd}
	\CA_{2,\delta}^{\mathrm{lev}} \times \CA_{h-2,\tilde{\delta}}^{\mathrm{lev}} \ar{r}{\varphi_{h,\delta}} \ar{d}{ p_{\delta} \times p_{\tilde{\delta}}} & \CA_h \\
	\CA_{2,\delta} \times \CA_{h-2, \tilde{\delta}} & 
\end{tikzcd}
\]
where $\tilde{\delta} = (1^{h-4},\delta_1,\delta_2)$ is the complementary polarization type.

\begin{thm}  \label{thm:main formula genus 2}
	\begin{enumerate}
		\item[(i)] (curve class zero) For all $h \geq 2$
		\[ N_{2,t^{h}}^{\CA_h} = \frac{1}{5760}\lambda_{h-1} \lambda_{h-2}. \]
		\item[(ii)] (curve classes of rank 1) For all $h \geq 2$:
		\[ N_{2,t^{h}-d t^{h-1}}^{\CA_h}
		=
		\frac{(-1)^{h-1} }{24}
		\sum_{\delta \geq 1}
		\sigma_3(d/\delta)
		\NL_{h, \delta} \lambda_{h-2}
		\]
		\item[(iii)] (curve classes of rank $2$) Let $a > 0$.
		For $h \geq 4$ we have
\[
			N_{2,t^{h}-d t^{h-1} + a t^{h-2}}^{\CA_h} 
	= - \sum_{\substack{\delta_1, \delta_2 > 0 \\ \delta_1 | \delta_2 \\ \delta_1 | d \\
			\delta_1 \delta_2 | a}}
	\sum_{\substack{k \geq 1 \\ k | d/\delta_1 \\ k^2 | \frac{a}{\delta_1 \delta_2}}}
	k^3 \sigma_1\left( \frac{a}{\delta_1 \delta_2 k^2} \right)
	c^{(\delta_2/\delta_1)}\left( \frac{d^2 - 4 a}{4 \delta_1 \delta_2 k^2}, \frac{d}{2 k \delta_2} \right) \cdot 
		\frac{\varphi_{h,(\delta_1,\delta_2) \ast}( p_{(\delta_1,\delta_2)}^{\ast}\lambda_1)}{ \deg(\varphi_{h,(\delta_1,\delta_2)})}
\]
For $h=3$ we have:
\begin{align*}
	N_{2,t(t^{2}-d t + a)}^{\CA_3}
	= - \sum_{\substack{\delta > 0 \\ 
			\delta | a}}
	\sum_{\substack{k \geq 1 \\ k | d, k^2 | \frac{a}{\delta}}}
	k^3 \sigma_1\left( \frac{a}{\delta k^2} \right)
	c^{(\delta)}\left( \frac{d^2 - 4a}{4 \delta k^2}, \frac{d}{2 k \delta} \right) \cdot 
	\NL_{3,\delta} \lambda_1
\end{align*}
For $h=2$ we have:
\begin{align*}
	N_{2,t^{2}-d t + a}^{\CA_2}
	= - \sum_{\substack{k \geq 1 \\ k | d, k^2 | a}}
	k^3 \sigma_1\left( \frac{a}{k^2} \right)
	c^{(1)}\left( \frac{d^2 - 4a}{4 k^2}, \frac{d}{2 k} \right) \cdot 
\lambda_1
\end{align*}

\end{enumerate}

\end{thm}

\begin{cor} 
\label{cor:genus 2 quotient invariants}	
	For any $r \geq 0$ and $h \geq 2$ we get
\[ \taut \ F_{2}^{\CA_h}(e_2^r)
=
-\frac{1}{360} \frac{(h-1) h}{|B_{2h-2}| |B_{2h}|}
\Lift_{2(h+r)}( D_{\tau}^r(G_{2h-2}) E_{5/2}^{(1)}) \lambda_{h-1} \lambda_{h-2}.
\]
In particular, we have the holomorphic anomaly equation:
\[
\frac{d}{dG_2} \taut\ F_{2}^{\CA_h}(e_2^r)
=
-2 r (r+2h-3)
\taut\ F_{2}^{\CA_h}(e_1 e_2^{r-1})
\]
\end{cor}
\begin{proof}
The tautological projections of the Noether-Lefschetz cycles in $\CA_h$ were computed by \cite{AitorNL} and \cite{CMOP}. One has
\[	\taut( \NL_{h,(\delta)} )
	=
	\frac{h \delta^{2h-1}}{6 |B_{2h}|} \prod_{p|\delta} (1-p^{2-2h}) \lambda_{h-1} \]
and for $h \geq 3$
\begin{multline*}
	\taut( \NL_{h,(\delta_1,\delta_2)} ) = 
	\frac{1}{360} \frac{h(h-1)}{|B_{2h-2}| |B_{2h}|} \\ \cdot
	\delta_1^{2h-5} \delta_2^{2h-1} \prod_{p|\delta_1} (1-p^{6-2h}) \prod_{p|\delta_2} (1-p^{4-2h})
	\prod_{p|\delta_2/\delta_1} \frac{1-p^{-4}}{1-p^{-2}} 
	\cdot \lambda_{h-1} \lambda_{h-3} 
	\end{multline*} 
Moreover the tautological projection of
$\varphi_{h,\delta\, \ast}( p_{(\delta_1,\delta_2)}^{\ast} \lambda_1 )$
is computed in Lemma~\ref{lemma:some taut projections} below.
The first part follows hence from Theorem~\ref{thm:modular identity} for $h \geq 3$, and directly from the definition of the Shimura lift in case $h=2$.
The holomorphic anomaly equation follows from Theorem~\ref{thm:shimura lift} and a straightforward computation using that the commutator $[\frac{d}{dG_2}, D_{\tau}]$ acts by $-2k$ on quasimodular forms of weight $k$, see \cite{vIOP}.
\end{proof}

\begin{proof}[Proof of Theorem~\ref{thm:genus 2 quotient quasimodularity} and Corollary~\ref{cor:genus 2}]
Theorem~\ref{thm:genus 2 quotient quasimodularity} follows from Corollary~\ref{cor:genus 2 quotient invariants}
and observing that
\[ F_g(e_1 P) = q \frac{d}{dq} F_g(P). \]
To prove Corollary~\ref{cor:genus 2}
observe that by Proposition~\ref{prop:dim constraint} we must have $h+n = \sum_i k_i + \deg(\Gamma)$,
so the claim follows from Corollary~\ref{cor:converse modularity}.
\end{proof} 

\begin{lemma} \label{lemma:some taut projections}
For $h \geq 2$ we have:                                
	\begin{multline*}
		\taut \frac{\varphi_{h,(\delta_1,\delta_2)\, \ast}( p_{(\delta_1,\delta_2)} \lambda_1 )}{|\mathrm{Sp}(K(\delta))|}  \\
		=
		\frac{1}{360}
		\frac{(h-1) h}{|B_{2h-2}| |B_{2h}|}
		\cdot 
		\delta_1^{2h-5} \delta_2^{2h-1} \prod_{p|\delta_1} (1-p^{6-2h}) \prod_{p|\delta_2} (1-p^{4-2h})
		\prod_{p|\delta_2/\delta_1} \frac{1-p^{-4}}{1-p^{-2}}
		%\delta_1^{2h-1} \delta_2^{2h-5} \prod_{p|\delta_1} (1-p^{6-2h}) \prod_{p|\delta_2} (1-p^{4-2h})
		%\prod_{p|\delta_2/\delta_1} \frac{(1-p^{-2})}{(1-p^{-4})}
		\cdot  \lambda_{h-1} \lambda_{h-2}.
	\end{multline*}
\end{lemma}
\begin{proof}
Let $\delta = (\delta_1,\delta_2)$ and let $\tilde{\delta}$ be the complementary polarization.
We first check the statement for $\delta=(1,1)$, where the morphism is of the form $\varphi_{h,(1,1)} : \CA_2 \times \CA_{h-2}\to \CA_h$.
In $\CA_2$ we have $[\CA_1 \times \CA_1] = 10 \lambda_1$ (Here we use that $[\CA_1 \times \CA_1]$ is twice the underlying locus \cite{AitorNL}.) Thus in $\CA_2 \times \CA_{h-2}$ we have $p_{\delta}^{\ast} c_1(\BE_{2}) = \frac{1}{10} [ \CA_1 \times \CA_1 \times \CA_{h-2} ]$.
	In \cite{CMOP} it was computed that
	\[ \taut( \CA_1 \times \CA_1 \times \CA_{h-2}) =
	\frac{\gamma_1 \gamma_1 \gamma_{h-2}}{\gamma_h} \lambda_{h-1} \lambda_{h-2} = 
	\frac{1}{36} \frac{(h-1) h}{|B_{2h-2}| |B_{2h}|} \lambda_{h-1} \lambda_{h-2}, \]
	where $\gamma_k = \prod_{i=1}^{k} \frac{|B_{2i}|}{4i}$. The claim follows for $\delta=(1,1)$.
	
	In general, one can use the strategy parallel to \cite{AitorNL} to reduce to $\delta=(1,1)$.
	Consider the diagram
	\[
	\begin{tikzcd}
		\CA_{2,\delta}^{\mathrm{lev}} \times \CA_{h-2,\tilde{\delta}}^{\mathrm{lev}} \ar{r}{\varphi_{h,\delta}} \ar{d}{ \varphi_{\delta} \times \varphi_{\tilde{\delta}} } & \CA_h \\
		\CA_2 \times \CA_{h-2} & 
	\end{tikzcd}
	\]
	where the morphism $\varphi_{\delta} : \CA_{u,\delta}^{\mathrm{lev}} \to \CA_u$ sends a triple $(B,\theta_B,f)$ to $B / f(\{ 0 \} \times \BZ^u / \delta)$ with the induced principal polarization, see \cite[Section 2.1]{AitorNL}.

	Then by the same argument as in \cite[Proof of Theorem 3]{AitorNL} one has
	\[
	\frac{1}{|\mathrm{Sp}(K(\delta))|}
	\taut\left( \varphi_{h,\delta\, \ast}( p_{\delta}^{\ast}\lambda_1 ) \right)
	= 
	\frac{\deg(\phi_{\delta}) \deg(\phi_{\tilde{\delta}})}{|\mathrm{Sp}(K(\delta))|}
	\taut\left( \varphi_{h,(1,1)\, \ast}( p_{(1,1)}^{\ast} \lambda_1 ) \right)
	\]
	By Proposition 26,27 in \cite{AitorNL} we have
	\[
	\deg(\phi_{\tilde{\delta}}) = (\delta_1 \delta_2)^{2h-3} \prod_{p|\delta_1} (1-p^{6-2h}) \prod_{p|\delta_2} (1-p^{4-2h}) \]
	and
	\[
	\frac{|\mathrm{Sp}(K(\delta))|}{\deg(\phi_{\delta})}
	= \delta_1^{2} \delta_2^{-2} \prod_{p|\delta_2/\delta_1} \frac{1-p^{-2}}{1-p^{-4}}
	\]
	Hence we find (Note that there is a typo in Theorem 3 of \cite{AitorNL})
	\[
	\frac{\deg(\phi_{\delta}) \deg(\phi_{\tilde{\delta}})}{|\mathrm{Sp}(K(\delta))|}
	=
	\delta_1^{2h-5} \delta_2^{2h-1} \prod_{p|\delta_1} (1-p^{6-2h}) \prod_{p|\delta_2} (1-p^{4-2h})
	\prod_{p|\delta_2/\delta_1} \frac{1-p^{-4}}{1-p^{-2}}. \qedhere
	\]
\end{proof}

\subsection{A basic lemma}
\begin{lemma} \label{lemma:a and s}
Let $j : B \to X$ be the inclusion of an abelian surface into a principally polarized abelian variety $(X,\theta)$ of dimension $h$.
Let $\beta = j_{\ast}\gamma$ for a class $\gamma \in H_2(B,\BZ)$.
Let $\theta_B = \theta|_{B}$ be of type $(\delta_1, \delta_2)$.
Then $\beta$ has characteristic polynomial
\[
\chi_{\beta}/t^{h-2} = t^2 - (\gamma \cdot \theta_B) t + \delta_1 \delta_2 \frac{(\gamma \cdot \gamma)}{2}.
\]
%In particular, $\delta_1 \delta_2 | \frac{1}{2} (\gamma,\gamma)$.
\end{lemma}
\begin{proof}
By Lemma~\ref{lemma:cbar comparison} we have
\[ \frac{1}{t^{h-2}} \chi_{\beta}(t) = \chi_{\gamma}(t) 
= \frac{1}{\delta_1 \delta_2} \int_{B} \frac{(t \theta_B + v_{\gamma})^2}{2} \\
= \frac{1}{\delta_1 \delta_2} \left( \delta_1 \delta_2 t^2 - (\theta_B\cdot v_{\gamma})t + \frac{1}{2} (v_{\gamma} \cdot v_{\gamma}) \right).
\]
It suffices hence to prove
\[ 
v_{\gamma} \cdot \theta_{B} = \delta_1 \delta_2 (\gamma \cdot \theta_B), \quad 
v_{\gamma} \cdot v_{\gamma} = \left(\frac{\theta_B \cdot \theta_B}{2} \right)^2 (\gamma \cdot \gamma). \]
To do so, it is enough to check this on a Zariski open subset of $H_2(B,\BR)$. So we may assume
\[ \theta_B= a_1 e_{1}^{\ast} \wedge f_1^{\ast} + a_2 e_2^{\ast} \wedge f_2^{\ast} \]
\[ \gamma = b_1 e_1 \wedge f_1 + b_2 e_2 \wedge f_2 \]
where $e_1,f_1,e_2,f_2$ is a basis of $H_1(B,\BZ)$ and $e_i^{\ast},f_i^{\ast}$ denotes the dual basis.
In particular, $\gamma \cdot \gamma = 2 b_1 b_2$ and $\theta_B\cdot \theta_B = 2 a_1 a_2$.
Hence
\begin{align*}
\varphi_{\theta_B}^{\ast}(\gamma) & = b_1 \theta_B(e_1) \wedge \theta_B(f_1) + b_2 \theta_B(e_2) \wedge \theta_B(f_2) \\
& = -b_1 a_1^2 f_1^{\ast} \wedge e_1^{\ast} - b_2 a_2^2 f_2^{\ast} \wedge e_2^{\ast} \\
& = b_1 a_1^2 e_1^{\ast} \wedge f_1^{\ast} + b_2 a_2^2 e_2^{\ast} \wedge f_2^{\ast}.
\end{align*}
We find hence
\[ v_{\gamma} \cdot v_{\gamma} = 2 b_1 b_2 a_1^2 a_2^2 = (\gamma,\gamma) \cdot \left( \frac{(\theta_B \cdot \theta_B)}{2} \right)^2 \]
and 
\[ v_{\gamma} \cdot \theta_{B} = a_1 a_2 (a_1 b_1 + a_2 b_2) = \delta_1 \delta_2 (\gamma \cdot \theta_B) \]
as desired.
\end{proof}

\subsection{Humbert divisors}
Let $U = \binom{0\ 1}{1 \ 0}$ be the hyperbolic lattice, let $v_0 \in U^3$ be a vector of square $2 \delta$,
and let $M = v_0^{\perp}$ be the orthogonal complement.
By transitivity of the orthogonal group on primitive vectors of a given square,
we may assume that $v_0$ lies in the third copy of $U$, so
\[ M = U \oplus U \oplus (-2 \delta). \]
The period domain is
\[ \CD = \{ Z \in \p(M) | Z \cdot \overline{Z} > 0, Z \cdot Z = 0\}^{+} \]
where $+$ stands for one of the two connected component.
Let $\Gamma := \widehat{\mathrm{SO}}^{+}(M) \leq SO(M)$ be the subgroup of the special orthogonal group which has spinor norm $1$ and acts trivially on the discriminant $M^{\vee}/M$.
The group $\Gamma$ acts naturally on $\CD$ and the quotient
is the moduli space of abelian surfaces of polarization type $(1,\delta)$,
\begin{equation} \label{isomo} \CA_{2,(1,\delta)} \cong \CD/\Gamma, \end{equation}
see for example \cite[Section 1.3]{GN2}.
Indeed, if $(X,\theta)$ is a polarized abelian surface with $\theta$ of type $(1,\delta)$, then $H^2(X,\BZ)$ together with the intersection pairing is isomorphic to $U^3$, so there exists an isometry $\varphi : H^2(X,\BZ) \to U^3$ sending $\theta$ to $v_0$.
Then the period mapping $\CA_{2,(1,\delta)} \to \CD/\Gamma$ which realizes the isomorphism \eqref{isomo} sends $(X,\theta)$ to
\[ \varphi([\sigma_X]) \in \CD/\Gamma, \]
where $\sigma_X \in H^0(X,\Omega_X^2)$ is the symplectic form on $X$.
The quotient by $\Gamma$ shows the map is independent of the choice of $\varphi$.

Let $n > 0$ and $\gamma \in M^{\vee}/M$. 
The Hegner divisors on $\CD/\Gamma$ are defined by 
\[ H_{n,\gamma} = \left( \bigcup_{\substack{v \in M^{\vee} \\ [v]=\gamma, \frac{1}{2} v \cdot v=-n}} v^{\perp} \right)/\Gamma. \]
We also set
$H_{0,0} = -\lambda_1$ amd $H_{0,\gamma} = 0$ if $\gamma \neq 0$.
By a result of Bruinier and K\"uhn \cite[Prop. 4.8]{BruinierKuehn} the degrees of the Hegner divisors are recorded by the Eisenstein series $E_{5/2}^{(\delta)}$ (which can be viewed as the Eisenstein series of $U \oplus U \oplus (-2 \delta)$ with respect to the dual Weil representation):
\[ \forall \gamma \in M^{\vee}/M\ \forall n \geq 0: \quad H_{n,\gamma} = -c(n,\gamma) \lambda_1. \]

For our geometric applications we also need a further construction.
Given integers $s \geq 0$ and $d \geq 0$ consider the determinant
\[ \Delta = \det \begin{pmatrix} 2 \delta & d \\ d & 2s \end{pmatrix} = 4 \delta s-d^2. \]
The matrix is meant to record the intersection numbers of a pair of classes $(\theta,\beta)$ in $H^2(X,\BZ)$ where $\beta$ is degree $d$ and self-intersection $\beta^2=2s$. The determinant $\Delta$ vanishes if and only if $\beta$ is colinear to $\theta$.
We define the Humbert divisors\footnote{These are just the Noether-Lefschetz divisors in the moduli space of abelian surfaces, but we use the classical terminology of Humbert divisor here to distinguish them from the classes $\NL_{h,(\delta_1,\ldots,\delta_u)}$ in $\CA_h$ that we defined earlier.} $\Hum_{s,d}$ in $\CA_{2,(1,\delta)}$ as follows:
\begin{itemize}
	\item If $\Delta \neq 0$ we set
	\[ \Hum_{s,d} = \left( \bigcup_{\substack{\beta \in U^3 \\ \beta \cdot v_0 = d \\ \beta \cdot \beta = 2s }}
	\beta^{\perp} \cap \CD \right)/\Gamma. \]
	\item If $\Delta = 0$ we set
	\[ \Hum_{s,d} = c_1(\BE^{\vee}) = -\lambda_1. \]
\end{itemize}

\begin{lemma} \label{lemma:Humbert degree}
	$\Hum_{s,d} = H_{n,\gamma}$ where $n=\frac{d^2-4 \delta s}{4 \delta}$ and $\gamma =d/2 \delta$ modulo $\BZ$. In particular,
	\[ \Hum_{s,d}  = -c^{(\delta)}\left( \frac{d^2 - 4 \delta s}{4 \delta}, \frac{d}{2 \delta}\right) \lambda_1. \]
\end{lemma}
\begin{proof}
	See \cite[Lemma 3]{GWNL}. For fixed $s,d$ with $\Delta \neq 0$, there is a bijection between the set
	\[ \{ \beta \in U^3 | \beta \cdot \beta = 2s, \beta \cdot \theta = d \} \]
	and the set $v \in M^{\vee}$ such that $-\frac{1}{2} v \cdot v = \frac{d^2-4 \delta s}{4 \delta}$ and $[v]=\{ d/ 2 
	\delta \}$.
	Indeed, $\beta$ is sent to the projection onto $M$, with inverse sending $v$ to $v + \frac{d}{2 \delta} v_0$.
\end{proof}
\begin{example}
Let $\delta=1$ so consider the case $\CA_2$.
	The divisor $\Hum_{0,1}$ parametrizes abelian surfaces $(X,\theta)$ which contain a class in $\NS(X)$ of degree $1$ and square $0$, so an elliptic curve of degree $1$, which forces $X = E \times F$.
	Since there are two curve classes on $E \times F$ of norm $0$ degree $1$ against the polarization $H$,
	$\Hum_{0,1}$ is two times the underlying reduced divisor. Using \eqref{E52_1} the degree is
	\[ \Hum_{0,1} = 10 \lambda_1. \]
	Indeed, $\Hum_{0,1}$ is precisely the zero divisor of the weight $10$ Igusa cusp form. 
\end{example}

We can also defined refined Noether-Lefschetz divisors $\Hum_{s,d,\ell}$ in $\CA_{2,(1,\delta)}$ which also record the multiplicity $\ell$ of the curve class as follows:
\begin{itemize}
	\item If $\Delta \neq 0$ we set
	\[ \Hum_{s,d,\ell} = \left( \bigcup_{\substack{\beta \in U^3, \div(\beta)=\ell \\ \beta \cdot v_0 = d \\ \beta \cdot \beta = 2s }}
	\beta^{\perp} \cap \CD \right)/\Gamma. \]
	%where $\mu(s,d ; L_v)$ is the number of classes $\beta \in L_v$ such that $\beta \cdot v_0 = d$ and $\beta \cdot \beta = 2s$.
	\item If $\Delta = 0$, we define $\Hum_{s,d,\ell} = c_1(\BE^{\vee}) = -\lambda_1$ in case $s=\delta \ell^2$ and equal to zero otherwise.
\end{itemize}

The refined Noether-Lefschetz divisors satisfy the basic relations:
\begin{align}
	\label{basicrel1} \Hum_{s,d,\ell} & = 
	\begin{cases} \Hum_{s/\ell^2, d/\ell, 1} & \text{ if } \ell|d, \ell^2 | s \\ 0 & \text{ otherwise} \end{cases} \\
	\label{basicrel2} \Hum_{s,d} & = \sum_{\ell \geq 1} \Hum_{s,d,\ell}.
\end{align}
In fact, these determine the classes of the refined Humbert divisors in terms of the classical ones via
\[ \Hum_{s,d,1} = \sum_{\substack{m \geq 1 \\ m | d, m^2|s}} \mu(m) \Hum_{\frac{s}{m^2},\frac{d}{m}} \]
where $\mu(m)$ is the M\"obius function, but we will not need this below.

\begin{example}
	On $\CA_2$ (case $\delta=1$) the Humbert divisor $\Hum_{0,d,1}$ parametrizes abelian surface with a primitive curve of arithmetic genus $1$ and degree $d$ against the polarization, so in other words, an elliptic curve embedded of degree $d$. Thus
	\[ \Hum_{0,d,1} = \NL_{2,(d)} \]
	and therefore
	\[ \Hum_{0,d} = \sum_{\ell|d} \NL_{2,(\ell)}. \]
	Recall that 
	\[ \NL_{2,(\ell)} = 10 \ell^{3} \prod_{p|\ell} (1-p^{-2}) \lambda_1 \]
	thus one gets 
	\[ -c^{(1)}\left( \frac{d^2}{4}, \frac{d}{2} \right) = \sum_{\ell | d} 10 \ell^{3} \prod_{p|\ell} (1-p^{-2}). \]
\end{example}

\subsection{Proof of Theorem~\ref{thm:main formula genus 2}}
Recall that
\[ N_{2,\chi}^{\CA_h} = \frac{1}{2} \CC_{2,\chi}( \tau_1(\sigma) ). \]
%which are also the quotient Gromov-Witten invariants.

\subsubsection{Curve class zero}
Let $h \geq 2$. We have
\[ [\Mbar_{2,1}(\CX,0)]^{\vir} = e(T_{\CX/\CA_h} \otimes \BE_{\Mbar_{2,1}}^{\vee}) \cap [\Mbar_{2,1} \times \CX]. \]
By Lemma~\ref{lemma:e(Hodge Hodge) eval} we have
\[ e(T_{\CX/\CA_h} \otimes \BE_2^{\vee}) = c_{2h}(\BE_{\CA_h}^{\vee} \otimes \BE_{\Mbar_{2}}^{\vee}) 
%= c_{2h}( \BE_{\CA_h}\otimes \BE_{\Mbar_{2}})
= \lambda_{h-1} \lambda_{h-2} \mu_1 \mu_2 \]
One has
\[
\int_{\Mbar_{2,1}} \psi_1 \mu_1 \mu_2 = 2 \int_{\Mbar_2} \mu_1 \mu_2  = \frac{1}{2880}.
\]
Thus we get
\[ N_{2,t^{h}}^{\CA_h} = \frac{1}{2} \cdot \frac{1}{2880} \lambda_{h-1} \lambda_{h-2}. \]

\subsubsection{Curve classes of rank $1$}
We use \eqref{full rank decomposition} to obtain the evaluation
\[ N_{2,t^{h} - d t^{h-1}}^{\CA_h}
=
\frac{1}{2}
\sum_{\delta \geq 1}
(-1)^{(h-1)} 
\left\langle \tau_{1}(\pt) \mu_1 \right\rangle^{E}_{2,d/\delta}
\cdot \frac{\varphi_{1,\delta \, \ast}(p_{\tilde{\delta}}^{\ast}(\lambda_{h-2}))}{|\deg \varphi_{1,\delta}|}.
\]
The Gromov-Witten invariants of the elliptic curve $E$
are computed in Lemma~\ref{lemma:GW E evaluation} below. 
Moreover, since $H^{\ast}(\CA_1) = \BQ$ we have
\[ \varphi_{1,\delta}^{\ast}(\lambda_{k}) = \pr_2^{\ast}(\lambda_k). \]
Thus we obtain
\[ N_{2,t^{h} - d t^{h-1}}^{\CA_h}
=
\frac{(-1)^{h-1} }{24}
\sum_{\delta \geq 1}
\sigma_3(d/\delta)
\NL_{h, \delta} \lambda_{h-2}
\]

\begin{lemma} \label{lemma:GW E evaluation} For any $a > 0$ we have $\langle \tau_1(\pt) \mu_1 \rangle^{E}_{2,a} = \frac{1}{12} \sigma_3(a)$.
\end{lemma}
\begin{proof}
	By \cite[Prop 6.8]{QHilbEll1} the series
	\[ \sum_{a \geq 0} \langle \tau_1(\pt) \mu_1 \rangle^{E}_{2,a} q^a \]
	is a multiple of the Eisenstein series $E_4(q) = 1 + 240 \sum_{n \geq 1} \sigma_3(n) q^n$. 
	The constant term is
	\[
	\langle \tau_1(\pt) \mu_1 \rangle^{E}_{2,0} = \int_{\Mbar_{2,1}} \psi_1 \mu_1 \mu_2 = \frac{1}{2880}
	\]
	so the claim follows.
\end{proof}

\subsubsection{Curve classes of rank $2$}
Let $\chi = t^{h-2}(t^2 - dt + a)$ with $a > 0$, and set $\chi' = \chi/t^{h-2}$.
Assume first $h \geq 4$
We use the formula \eqref{full rank decomposition} to get
\[
\CC_{2,\chi}( \tau_1(\sigma) )
=
\sum_{\substack{\delta_1, \delta_2 > 0 \\ \delta_1 | \delta_2 \\ \delta_1 | d \\ \delta_1 \delta_2 | a}}
\frac{1}{\deg(\varphi_{h,\delta})} \varphi_{h,\delta \ast}\left( \underbrace{\rho_{\ast}\left(
\ev_1^{\ast}(\sigma) \psi_1 \cap  [\Mbar_{2,1}(\pi_{2,\delta}^{\mathrm{lev}}, \chi' ) ]^{\vir} \right)}_{I} \times [\CA^{\mathrm{lev}}_{h-2,\tilde{\delta}}] \right)
\]
where $\delta = (\delta_1,\delta_2)$ and
$\rho : \Mbar_{2,1}(\pi_{2,\delta}^{\mathrm{lev}}, \chi' ) \to \CA_{2,\delta}^{\mathrm{lev}}$
is the map to the base.
The conditions $\delta_1 | d$ and  $\delta_1 \delta_2 | a$ here follow from Lemma~\ref{lemma:a and s} (for the first observe that $\theta_B$ is divisible by $\delta_1$).

Let $\delta_{\mathrm{prim}} := \delta_2/\delta_1$. We have the composition of maps
\begin{equation} \label{4sdf33} 
\CA_{2,\delta}^{\mathrm{lev}} \xrightarrow{\mathrm{Fg}} \CA_{2,\delta} \cong \CA_{2,(1,\delta_{\mathrm{prim}})}
\end{equation}
where the first map forgets the level structure and the second is an isomorphism.

The class $I$ is the pullback under \eqref{4sdf33} of the class
\[
I' = \rho_{\ast}\left(
\ev_1^{\ast}(\sigma) \psi_1 \cap  [\Mbar_{2,1}(\pi_{2,(1,\delta_{\mathrm{prim}})}, \chi'_{\mathrm{prim}}) ]^{\vir} \right)
\]
where
\begin{itemize}
	\item $\pi_{2,(1,\delta_{\mathrm{prim}})} : \CX_{2,(1,\delta_{\mathrm{prim}})} \to \CA_{2,(1,\delta_{\mathrm{prim}})}$ is the universal family,
	\item $\rho : \Mbar_{g,1}(\pi_{2,(1,\delta_{\mathrm{prim}})},\chi_{\mathrm{prim}}') \to \CA_{2,(1,\delta_{\mathrm{prim}})}$ is the base map,
	\item $\chi'_{\mathrm{prim}}$ is the characteristic polynomial with respect to $\theta_B/\delta_1$ given by
	\[ \chi'_{\mathrm{prim}}(t) = \frac{1}{\delta_1^2} \chi'( \delta_1 t) = t^2 - \frac{d}{\delta_1} t + \frac{a}{\delta_1^2}. \]
\end{itemize}

By Lemma~\ref{lemma:a and s} the curve class $\beta = f_{\ast}[C]$
of a stable map in $\Mbar_{g,1}(\pi_{2,(1,\delta_{\mathrm{prim}})},\chi_{\mathrm{prim}}')$
has degree $d/\delta_1$ and self-intersection
\[ s = \frac{1}{2} \beta \cdot \beta = \frac{a}{\delta_1 \delta_2}. \]

We now use the fact proven in \cite{GWNL} that the Gromov-Witten invariants of family of smooth symplectic surfaces can be expressed through the Humbert divisors and the reduced Gromov-Witten invariants of a fiber.
The relation reads
\[
I' = \sum_{\substack{\ell \geq 1 \\ \ell| \frac{d}{\delta_1}, \ell^2 |s}} \Hum_{s,d/\delta_1,\ell} \cdot \langle \tau_1(\sigma) \rangle^{\mathrm{red}}_{2,s,\ell}
\]
where we write
\begin{equation} \label{Red invariants}
\left\langle \tau_1(\sigma) \right\rangle^{\mathrm{red}}_{g,s,\ell}
=
\int_{[\Mbar_{g,1}(X,\beta)]^{\mathrm{red}}} \psi_1 \ev^{\ast}(\sigma).
\end{equation}
for the reduced Gromov-Witten invariants of a principally polarized abelian surface $(X,\theta)$
in a curve class $\beta$ satsifying
\[ \beta \cdot \beta = 2s, \quad \div(\beta) = \ell. \]
This relationship was proven in \cite{GWNL} for K3 surfaces but the proof works the same also for abelian surfaces,
see \cite{O_Ab2} for an extended discussion. By monodromy invariance
the integral on the right of \eqref{Red invariants} only depends upon $(X,\beta)$ through $s,\ell$, which explains the choice of notation on the left, see \cite[Sec 2.2]{OP_MCF} and \cite[Section 1]{BOPY}.

By \cite[Theorem 6]{BOPY} we have
\[
\langle \tau_1(\sigma) \rangle^{\mathrm{red}}_{2,s,\ell}
=
2\sum_{\substack{k \geq 1 \\ k|\ell, k^2|s}} k^3 \sigma_1\left( \frac{s}{k^2} \right).
\]

Inserting and using Lemma~\ref{lemma:Humbert degree} we find
\begin{align*}
I' & = 2 \sum_{\substack{\ell \geq 1 \\ \ell| \frac{d}{\delta_1}, \ell^2 |s}} \Hum_{s,d/\delta_1,\ell}
\sum_{\substack{k \geq 1 \\ k|\ell, k^2|s}} k^3 \sigma_1\left( \frac{s}{k^2} \right) \\
& = 2 \sum_{\substack{k \geq 1 \\ k| \frac{d}{\delta_1}, k^2|s}} k^3 \sigma_1\left( \frac{s}{k^2} \right)
\sum_{\tilde{\ell} \geq 1} \Hum_{s,d/\delta_1,k \tilde{\ell}} \\
& = 2 \sum_{\substack{k \geq 1 \\ k| \frac{d}{\delta_1}, k^2|s}} k^3 \sigma_1\left( \frac{s}{k^2} \right)
\sum_{\tilde{\ell} \geq 1} \Hum_{s/k^2,d/k \delta_1,\tilde{\ell}} \\
& = 2 \sum_{\substack{k \geq 1 \\ k| \frac{d}{\delta_1}, k^2|s}} k^3 \sigma_1\left( \frac{s}{k^2} \right)
\Hum_{s/k^2,d/k \delta_1} \\
& = -2\sum_{\substack{k \geq 1 \\ k| \frac{d}{\delta_1}, k^2|s}} k^3 \sigma_1\left( \frac{s}{k^2} \right)
c^{(\delta_{\mathrm{prim}})}\left( \frac{d^2/\delta_1^2 - 4 \delta_{\mathrm{prim}} s}{4 k^2 \delta_{\mathrm{prim}}}, \frac{d}{2 \delta_1 \delta_{\mathrm{prim}} k} \right) \lambda_1 \\
& = - 2\sum_{\substack{k \geq 1 \\ k| \frac{d}{\delta_1}, k^2| \frac{a}{\delta_1 \delta_2} }} k^3 \sigma_1\left( \frac{a}{\delta_1 \delta_2 k^2} \right)
c^{(\delta_2/\delta_1)}\left( \frac{d^2 - 4 a}{4 \delta_1 \delta_2 k^2}, \frac{d}{2 \delta_2 k} \right) \lambda_1.
\end{align*}
where the Humbert divisors $\Hum_{s,d,\ell}$ are taken on $\CA_{2,(1,\delta_{\mathrm{prim}})}$.
Inserting this for $I$ we thus find
\[
\CC_{2,\chi}( \tau_1(\sigma) )
=
-2
\sum_{\substack{\delta_1, \delta_2 > 0 \\ \delta_1 | \delta_2 \\ \delta_1 | d \\ \delta_1 \delta_2 | a}}
\sum_{\substack{k \geq 1 \\ k| \frac{d}{\delta_1} \\ k^2| \frac{a}{\delta_1 \delta_2} }} k^3 \sigma_1\left( \frac{a}{\delta_1 \delta_2 k^2} \right)
c^{(\delta_2/\delta_1)}\left( \frac{d^2 - 4 a}{4 \delta_1 \delta_2 k^2 }, \frac{d}{2 \delta_2 k} \right) 
\frac{\varphi_{h,\delta \ast}\left( p_{\delta}^{\ast}\lambda_1 \right)}{\deg(\varphi_{h,\delta})}
\]
which is exactly what we wanted to prove.

The case $h=3$ is the same except that
if $B \subset X$ is an abelian subsurface, then the complimentary abelian variety is an elliptic curve $E$. Hence $\theta|_{B}$ must have type $(1,\delta)$ where $\delta$ is the degree of $\theta|_{E}$.
Thus we obtain the same terms but sum only over the terms with $\delta_1=1$.
Moreover, since $\CA_{1,\tilde{\delta}}$ is a quotient of $\BA^1$, we have
$\varphi_{h,\delta}^{\ast}(\lambda_1) = p_{(1,\delta)}^{\ast}\lambda_1$, so
\[
\frac{ \varphi_{h,\delta \ast}\left( p_{(1,\delta)}^{\ast}\lambda_1\right)}{\deg(\varphi_{h,\delta})}
=
\lambda_1 \frac{\varphi_{h,\delta \ast}\left( p_{(1,\delta)}^{\ast}(1) \right)}{\deg(\varphi_{h,\delta})}
=
\lambda_1 \NL_{3,\delta}.
\]

In the case $h=2$ only the term $\delta_1=\delta_2=1$ contributes and $\varphi_{h,\delta}=\id$.
\qed

\appendix

\section{Perfect obstruction theories}
\label{app:pot}
\subsection{Definitions}
Let $\pi : \CX\to B$ be a smooth projective morphism to a smooth variety $B$.
Let $\pi : M \to B$ be a connected component of the moduli space of genus $g$ stable maps to the fibers of $\pi$.
Then $M$ comes equipped with a relative perfect obstruction theory 
\[ \phi : E^{\bullet} \to \BL_{M/B}, \]
which is compatible with base change.
In particular, for $M_b$ the fiber of $M$ over a point $b \in B$ the restriction
\[ \phi|_{M_b} : E^{\bullet}|_{M_b} \to \BL_{M/B}|_{M_b} \cong \BL_{M_b} \]
 is just the usual perfect obstruction theory of the component $M_b$ of the moduli space of stable maps to $X_b$. In particular, $\rk(E^{\bullet}) = \mathrm{vdim}(M_b)$.

The virtual class on $M$ is constructed as follows. One sets
\[ E^{\bullet}_{\mathrm{abs}} := \mathrm{Cone}( E^{\bullet} \to \BL_{M/B} \to \pi^{\ast} \Omega_B[1] )[-1] \]
which fits into a morphism of distinguished triangles
\[
\begin{tikzcd}
\pi^{\ast} \Omega_B \ar{r} \ar{d}{=} &  E^{\bullet}_{\mathrm{abs}} \ar{r} \ar{d} & E^{\bullet} \ar{d}{\phi} \ar{r} & \ldots [1] \\
\pi^{\ast} \Omega_B \ar{r} & \BL_{M} \ar{r} & \BL_{M/B} \ar{r} & \ldots [1]
\end{tikzcd}
\]
Then $E^{\bullet}_{\mathrm{abs}} \to \BL_M$ is an (absolute) perfect obstruction theory of virtual dimension $\mathrm{vd} := \rk (T^{\vir}_{M/B,\mathrm{abs}} = \dim(B) + \rk(E^{\bullet})$. (This is a special case of the virtual pullback construction, but can be seen entirely elementary).
We denote the relative and absolute virtual tangent bundles by
\[ T_{M/B}^{\vir} = E^{\vee}, \quad T^{\vir}_{M/B,\mathrm{abs}} = (E_{\mathrm{abs}})^{\vee}. \]

The virtual fundamental class associated to the absolute perfect obstruction theory
can be expressed in terms of the virtual tangent bundle by Siebert's formula:
\begin{equation} \label{Sieberts formula}
[M]^{\vir} = \left[ \frac{1}{c(T^{\vir}_{M/B,\mathrm{abs}})} \cap c_F(M) \right]_{\mathrm{vd}}
\end{equation}
where  $[ - ]_{m}$ stands for taking the dimension $m$ component and $c_F(M)$ is the Fulton canonical class, see \cite[Sec.C1]{PTKKV}. In particular, the virtual class only depends on the $K$-theory class of the obstruction theory. The comparison of virtual tangent bundles derived below hence yield immediately formulas comparing the associated virtual fundamental classes.

\subsection{Comparison to relative obstruction theory obtained from a fibration $B \to C$}
\label{app:subsec on comparing two relative}
Let $p : B \to C$ be a smooth projective morphism to a smooth variety $C$.
Then $M$ is also isomorphic to a component in the moduli space of stable maps to the fibers of $p \circ \pi$. Consider the relative perfect obstruction theory
\[ F^{\bullet} \to \BL_{M/C} \]
associated to the morphism $p \circ \pi$. Let $F_{\mathrm{abs}}^{\bullet}$ be the associated absolute perfect obstruction theory as before. Let
\[ T_{M/C}^{\vir} = F^{\vee}, \quad T^{\vir}_{M/C,\mathrm{abs}} = (F_{\mathrm{abs}})^{\vee}. \]
be the associated virtual tangent bundles.
We want to compare here $T^{\vir}_{M/C,\mathrm{abs}}$ and $T^{\vir}_{M/B,\mathrm{abs}}$.

If $\rho : \CC \to M$ is the universal family and $f : \CC \to \CX$ is the universal map, there is a distinguished triangle
\[ R \rho_{\ast}( f^{\ast} T_{\CX/B} ) \to R \rho_{\ast}( f^{\ast} T_{\CX/C}  ) \to R \rho_{\ast}(f^{\ast} \pi^{\ast} T_{B/C}) \xrightarrow{[1]} \ldots . \]
This compares the obstruction theories relative to $B$ and $C$, relative to the moduli stack of prestable curves, and yields the distinguished triangle
\[ 
R \rho_{\ast}( f^{\ast} \pi^{\ast} T_{B/C} )^{\vee} \to F^{\bullet} \to E^{\bullet} \xrightarrow{[1]} \ldots \,. \]
%or dually
%\[ T_{E}^{\vir} \to T_{F}^{\vir} \to R \rho_{\ast}(f^{\ast} \pi^{\ast} T_{B/C}) \xrightarrow{[1]} \ldots . \]
Observe that
\[
R \rho_{\ast}(f^{\ast} \pi^{\ast} T_{B/C}) = \pi^{\ast} T_{B/C} \otimes R \rho_{\ast}(\CO_{\CC}).
\]
and by relative Serre dualty
\[
\rho_{\ast} \CO_{\CC} \cong \CO, \quad \quad 
R^1 \rho_{\ast}(\CO_{\CC}) \cong \BE_{g}^{\vee}
\]
where $\BE_g$ is the Hodge bundle over the moduli space $M$.
By taking the cones of the vertical maps in the following map of exact triangles
\[
\begin{tikzcd}
R \rho_{\ast}(f^{\ast} \pi^{\ast} T_{B/C})^{\vee} \ar{r} \ar{d} & F^{\bullet} \ar{r} \ar{d} & E^{\bullet} \ar{d} \ar{r}{[1]} & \ldots \\
\pi^{\ast} \Omega_{B/C} \ar{r} & (p \circ \pi)^{\ast} \Omega_{C}[1] \ar{r} & \pi^{\ast} \Omega_{B}[1] \ar{r}{[1]} & \ldots 
\end{tikzcd}
\]
we obtain the exact triangle
\[ \pi^{\ast}(\Omega_{B/C}) \otimes \BE_{g}[1] \to F_{\mathrm{abs}}^{\bullet} \to E_{\mathrm{abs}}^{\bullet} \to \]
and hence in $K$-theory
\begin{equation} \label{Ktheory compare1} T^{\vir}_{M/C,\mathrm{abs}} = T^{\vir}_{M/B,\mathrm{abs}} - \pi^{\ast}(T_{B/C}) \otimes \BE_{g}^{\vee}. \end{equation}

\subsection{Comparison to relative obstruction theory obtained from an inclusion $B_0 \hookrightarrow B$}
\label{app:subsec reduced}
Let $B_0 \subset B$ be a smooth closed subscheme and assume that $M$ is isomorphic to a component of the moduli space of stable maps to the fibers of the restriction $\CX|_{B_0} \to B_0$.
In particular, for maps $f : C \to X_b$ for $b \in B_0$ in $M$ all deformations of $f$ outside of the $B_0$ locus are (even infinitesimally) obstructed. 
Consider the relative perfect obstruction theory
\[ F^{\bullet} \to \BL_{M/B_0} \]
associated to the morphism $M \to B_0$ and let
$F_{\mathrm{abs}}^{\bullet}$ be the associated absolute perfect obstruction theory
with virtual tangent bundle 
\[ T_{M/B_0,\mathrm{abs}}^{\vir} = (F_{\mathrm{abs}}^{\bullet})^{\vee}. \]
Our goal here is to compare
$T_{M/B,\mathrm{abs}}^{\vir}$ and $T_{M/B_0,\mathrm{abs}}^{\vir}$.

Consider the exact sequence
\[ 0 \to N^{\ast} \to \Omega_{B}|_{B_0} \to \Omega_{B_0} \to 0 \]
where $N$ is the normal bundle of $B_0 \subset B$.
Moreover, since relative perfect obstruction theories are compatible under base change
$E^{\bullet} \cong E^{\bullet}|_{B_0} \cong F^{\bullet}$.
Using the definition of the absolute perfect obstruction theory this gives an exact triangle
\[ N^{\ast} \to E^{\bullet}_{\mathrm{abs}} \to F^{\bullet}_{\mathrm{abs}} \to \pi_{B_0}^{\ast} N^{\ast}[1] \to \ldots \,. \]
We hence get in $K$-theory
\begin{equation} \label{Ktheory compare2}
T_{M/B,\mathrm{abs}}^{\vir} = T_{M/B_0,\mathrm{abs}}^{\vir} + N.
\end{equation}

\begin{rmk}
Since deformations away from $B_0$ are obstructed, we have pointwise a short exact sequence
\[ 0 \to N \to h^1(T^{\vir}_{M/B_0,\mathrm{abs}}) \to h^1(T_{M/B,\mathrm{abs}}^{\vir}) \to 0. \]
The obstruction theory $T_{M/B,\mathrm{abs}}^{\vir}$ hence may be viewed as a reduced obstruction theory of $T^{\vir}_{M/B_0,\mathrm{abs}}$. 
We refer to \cite{KT1} for more details on this construction.
\end{rmk}

\section{Modular form identities (joint with Brandon Williams)}
\label{appendix:modular form identity}
\subsection{Formula}
Let $M_{\delta} = A_1(\delta) = (2 \delta)$ and let $E_{5/2}^{(\delta)}$ be the Eisenstein series of weight $5/2$ for the Weil representation associated to the lattice $M_{\delta}$, see \cite{BruinierKuss}.
The Eisenstein series for the finite Weil representation of $M_{\delta}$ has a Fourier expansion of the form
\[
E_{5/2}^{(\delta)}(\tau)
=
\sum_{\gamma\in M_{\delta}^{\vee}/M_{\delta}}
\sum_{\substack{n \geq 0 \\ n - \frac{1}{2} \langle \gamma, \gamma \rangle \in \BZ}}
c^{(\delta)}(n,\gamma)q^n\e_\gamma .
\]
%The coefficients were explicitly written down in \cite{BruinierKuss}
%and an implementation can be found in \cite{WeilRepProgram}.
We normalize the Eisenstein series to have constant coefficient $1$, i.e. $c^{(\delta)}(0,0)=1$.

Let $\sigma_k(n) = \sum_{d|n} d^k$ be the standard divisor functions.

For $h \geq 3$ and $r \geq 0$ define power series
\[ F_{h,r}(q) \in \BQ[[q]] \]
as follows:
\begin{align*}
	F_{3,r}(q) := & \delta_{r0} \left( \frac{1}{4}
	\frac{B_{6} B_{6-2}}{6 (6-2)}
	- \frac{1}{24} 
	\sum_{d \geq 1} q^d
	\sum_{\delta|d}
	\sigma_3(d/\delta)
	\delta^5 \prod_{p|\delta} (1-p^{-4}) \right) \\
	& +\sum_{d \geq 1} q^d \sum_{\delta \geq 1} 
	\sum_{k | d} k^{3+2r} \sum_{i \geq 1} c^{(\delta)}\left( \frac{d^2}{4 \delta k^2} - i, \frac{d}{2 k \delta} \right) \sigma_1(i) i^r \delta^{5+r} \prod_{p|\delta} (1-p^{-4}) 
\end{align*}
and for $h \geq 4$ 
\begin{align*}
	F_{h,r}(q) & = 
	\delta_{r0} \left(
	\frac{1}{4}
	\frac{B_{2h} B_{2h-2}}{2h(2h-2)} 
	+
	5 \frac{B_{2h-2}}{(2h-2)}
	\sum_{d \geq 1} q^d
	\sum_{\delta | d}
	\sigma_3(d/\delta) \cdot
	\delta^{2h-1} \prod_{p|\delta} (1-p^{2-2h}) \right) \\
	&  
	+ \sum_{d \geq 1} q^d 
	\sum_{\substack{\delta_1,\delta_2 \geq 1 \\ \delta_1 | \delta_2 \\ \delta_1 | d \\ }}
	%\sum_{\substack{(\delta_1,\delta_2) \\ \delta_1 | \delta_2 \\ \delta_1 | d \\ \delta_2 > 0}}
	\sum_{k | d/\delta_1} k^{3+2r} 
	\sum_{i \geq 1} c^{(\delta_2/\delta_1)}\left( 
	\frac{d^2}{4 \delta_1 \delta_2 k^2} - i, \frac{d}{2 \delta_2 k} \right) \sigma_1(i) i^r \\
	& \quad \quad \quad \quad 
	\cdot 
	\delta_1^{2h-5+r} \delta_2^{2h-1+r} \prod_{p|\delta_1} (1-p^{6-2h}) \prod_{p|\delta_2} (1-p^{4-2h})
	\prod_{p|\delta_2/\delta_1} \frac{1-p^{-4}}{1-p^{-2}}.
\end{align*}

Recall for even $k$ the weight $k$ Eisenstein series
$G_k(q) = - \frac{B_k}{2 k} + \sum_{n \geq 1} \sum_{d|n} d^{k-1} q^n$.

Recall that for $k \geq 2$ the Shimura lift maps
a (quasi)modular form $F$ of weight $k+\frac{1}{2}$ for the Weil representation of $M_1$ with Fourier expansion
\[ F = \sum_{n \geq 0} \sum_{\gamma \in M^{\vee}/M} c(n,\gamma) q^n e_{\gamma} \]
to the (quasi)modular form for $\SL_2(\BZ)$ of weight $2k$ with Fourier expansion:
\begin{align*}
	\Lift_k(F) 
	& = -\frac{c(0,0) B_k}{2k} + \sum_{d \geq 1} \sum_{\ell|d} \ell^{k-1} c\left(\frac{d^2}{4 \ell^2}, \frac{d}{2\ell} \right)  q^{d}
\end{align*}
where we identify $M^{\vee}/M = \frac{1}{2} \BZ / \BZ$.
See \cite{OWilliams}.

Let $D_{\tau} = q \frac{d}{dq}$ where $q=e^{2 \pi i t}$.

The goal of this appendix is to prove the following result:
\begin{thm} \label{thm:modular identity} For any $h \geq 3$ and $r \geq 0$ we have
	\[
	F_{h,r}(q) = \Lift_{2(r+h)}( D_{\tau}^r(G_{2h-2}) E_{5/2}^{(1)})
	\]
	In particular, $F_{h,r}$ is a quasimodular form
	of weight $4(h+r)$.
\end{thm}

\subsection{Acknowledgements}
The stated identity in Theorem~\ref{thm:modular identity} was found by the authors based on numerical computations.
During the SwissMAP workshop {\em Moduli of curves and abelian varieties} in Les Diablerets, May 2026, the task to find a proof was submitted as a problem to the AI-benchmark site IMProofBench \cite{schmitt2025improofbench}. With further support of Johannes Schmitt the model ChatGPT 5.5 (using an early version of the harness developed in \cite{ProofCouncil}) provided then a 14-page document claiming a proof. We did not fact check the AI document, but after understanding the proof strategy, we then wrote our own cleaner, but essentially equivalent version.
In particular, the key numer-theoretic identites Propositions~\ref{prop:identity1}
and ~\ref{prop:identity2} were suggested by AI. 
%The work shows the potential of AI in mathematics for concrete algebro/combinatorical problems.
We thank Johannes Schmitt very much for his help as well as discussions on using AI in mathematics.

\subsection{Fourier coefficients of Eisenstein series}
The first step in the proof is to write the coefficients of the Eisenstein series $E_{5/2}^{(\delta)}$ in terms of the coefficients of $E_{5/2}^{(1)}$:

Write an element of $M_{\delta}^{\vee}/M_{\delta}$ as
$\gamma=\frac{r}{2\delta}$, where $r\in \Z/2\delta\Z$ and
for $n \geq 0$ let
\[
N=N(n,r):=n-\frac{r^2}{4\delta}.
\]
If \(N\notin \Z\), then the coefficient $c^{(\delta)}(n,\gamma)$ is zero. Hence assume from now on that
\(N\in\Z\).

\begin{theorem} \label{thm:coeff of e_{5/2}delta}
	For $\gamma=r/(2\delta)\in M_\delta^{\vee}/M$, one has
	\[
	c^{(\delta)}\!\left(n,\frac{r}{2\delta}\right)
	=
	\delta^{-2}
	\prod_{p\mid \delta}\left(1+p^{-2}\right)^{-1}
	\sum_{\substack{a>0\\ a^2\mid \delta\\ a\mid r}}
	\mu(a)
	\sum_{\substack{b>0\\ b\mid N\\ b\mid r/a\\ b\mid \delta/a^2}}
	b^2\,
	c^{(1)}\!\left(
	\frac{\delta n}{a^2 b^2},
	\frac{r}{2ab}
	\right),
	\]
	where $\mu$ is the M\"obius function.
\end{theorem}
\begin{proof}
	This is a special case of a general identity relating the vector-valued Eisenstein series attached to a lattice $L$, with quadratic form $Q$, and the rescaled lattice $L(\delta)$ with quadratic form $\delta \cdot Q$. Suppose $k \in \frac{1}{2}\mathbb{Z}$, $k \ge 5/2$ is a weight for which $\kappa = k + 1 - \mathrm{rank}(L)/2$ is an integer and $k + 1 + \mathrm{sgn}(L)/2$ is even, and $\beta \in L'/L$ is a coset of norm $0$ mod $\mathbb{Z}$. Then one defines the Eisenstein series $$E_{k, L, \beta}(\tau) := \sum_{\gamma \in \mathrm{Mp}_2(\mathbb{Z}) / \Gamma_{\infty}} (e_{\beta} + e_{-\beta}) \Big|_{k, \rho} \gamma,$$ where $e_{\beta}$ are basis elements of the group ring $\mathbb{C}[L'/L]$, and where $|_{k, \rho}$ is the Petersson slash operator with respect to the Weil representation $\rho$. Similarly to classical modular forms, the Eisenstein series $E_{k, L, \beta}$ span the orthogonal complement of the space of cusp forms as $\beta$ runs through the isotropic cosets of $L'/L$. 
	
	Bouchard, Creutzig and Joshi \cite{BCJ19} have defined natural index-raising Hecke operators that are analogous to the operators $U_N$, $V_N$ on Jacobi forms, and we also use the notation $U_N$ and $V_N$ for them. Explicitly, if $f$ has Fourier expansion $$f(\tau) = \sum_{\gamma \in L'/L} \sum_{n \in \mathbb{Z} - Q(\gamma)} c_f(n, \gamma) q^n e_{\gamma} \in M_k(\rho_L),$$ then $f|V_N$ is a vector-valued modular form for the lattice $L(N)$ with Fourier expansion $$f \Big| V_N = \sum_{\beta \in L(N)' / L(N)} \sum_{n \in \mathbb{Z} - Q_N(\beta)} a(n, \beta) q^n e_{\beta} \in M_k(\rho_{L(N)}),$$ where $$a(n, \beta) = \sum_{\substack{ad= N \\ d \beta \in L', nd/a \in \mathbb{Z} - Q(d\beta)}} a^{\kappa} c_f \Big( \frac{nd}{a} - Q\left( d\beta \right), d \beta \Big).$$ $f|U_N$ is a vector-valued modular form for the lattice $L(N^2)$ with Fourier expansion $$f \Big| U_N = \sum_{\beta \in L(N^2)' / L(N^2)} \sum_{n \in \mathbb{Z} - Q_{N^2}(\beta)} a(n, \beta) q^n e_{\beta} \in M_k(\rho_{L(N^2)}),$$ where $$a(n, \beta) = \begin{cases} c(n, N \cdot \beta): & N\beta \in L'; \\ 0: & \text{otherwise}. \end{cases}$$ By the construction of \cite{BCJ19}, the operators $U_N$ and $V_N$ map cusp forms into cusp forms and the Eisenstein space (or orthogonal complement of cusp forms) into the Eisenstein space. To prove the general identity $$E_{k, L(\delta), 0} = \Big( \delta^{1-\kappa} \prod_{p | \delta} (1 + p^{1 - \kappa})^{-1} \Big) \cdot \sum_{a^2 | \delta} \mu(a) E_{k, L, 0} \Big| U_a \Big| V_{\delta / a^2},$$ since both sides belong to the Eisenstein space, one only has to check that the constant terms match. Theorem~\ref{thm:coeff of e_{5/2}delta} is the special case of this identity for the lattice $L = M = A_1$ and $k = 5/2$.
\end{proof}

We will also need the following evaluation:
\begin{lemma} \label{lemma:c1 square coefficients} For all $d \geq 1$ we have
	\[ c^{(1)}\left( \frac{d^2}{4}, \frac{d}{2} \right)
	= -10 \sum_{a|d} \mu(a) a \sigma_3(d/a). \]
\end{lemma}
\begin{proof}
	Since $\Lift_{2}(E_{5/2}^{(1)})$ is a modular form of weight $4$,
	checking the constant term, it must be $-10 G_4(q)$.
	Writing out this identity and using M\"obius inversion, the claim follows.
\end{proof}

\subsection{Two number-theoretic identities}
For a positive integer $\delta \geq 1$ define
\[ P_t(\delta) = \prod_{p | \delta} (1-p^{-t}) \]
where the product is over prime divisors of $\delta$.

\begin{prop} \label{prop:identity1}
	For any integers $k,t \geq 1$ and $\ell,j \geq 1$
	\begin{equation} \label{claimed equality 1}
		\sum_{\substack{a,b\geq 1\\ ab\mid \ell}}
		\sum_{c\mid j}
		a^{k+t} b^t c^{k+t} \mu(a)\,
		\sigma_k\!\left(\frac{jb}{c}\right)
		P_t(a^2bc)
		=
		\ell^t \sigma_{k+t}(j).
	\end{equation}
\end{prop}
\begin{proof}
	We first observe that both sides are multiplicative\footnote{An expression $f(\ell,j)$ is multiplicative, if for all decompositions
		$\ell=\ell_1 \ell_2$ and $j=j_1 j_2$ with $\gcd( \ell_1 j_1 , \ell_2 j_2)$ we have
		$f(\ell, j) = f(\ell_1,j_1) f(\ell_2,j_2)$.},
	so we may assume that
	\[ \ell = p^u, \quad j = p^v, \quad a=p^{\alpha}, \quad b=p^{\beta}, \quad c=p^{\gamma} \]
	for a prime prime. 
	Let $f_{u,v}(\ell,j)$,$g_{u,v}(\ell,j)$ be the left and right hand side of \eqref{claimed equality 1} respectively.
	
	If $u=0$ or equivalently $\ell=1$ then
	\begin{align*}
		f_{0,v}(\ell,j) & = \sum_{c | j} c^{k+t} \sum_{e| \frac{b}{c}} e^k P_t(c) \\
		& = \sum_{m | j} m^k \sum_{c|m} c^t P_t(c) \\
		& = \sum_{m | j} m^{k} m^t \\
		%& = \sigma_{k+t}(j) \\
		& = g_{0,v}(\ell,j),
	\end{align*}
	where $\sum_{c|m} c^t P_t(c) = m^t$ holds since it holds for $c=p^u$ by a direct heck and both sides are multiplicative.
	%in $c$ and the case $c=p^u$ is a direct check.
	For all $u \geq 1$ we have
	\[
	f_{u,v}(\ell,j)
	=
	\sum_{\substack{\alpha,\beta \geq 0 \\ \alpha + \beta \leq u}} \sum_{\gamma \in v}
	p^{(k+t) \alpha + t \beta + (k+t) \gamma} \mu(p^{\alpha}) \sigma_k(p^{v + \beta - \gamma}) P_t(p^{2 \alpha + \beta + \gamma})
	\]
	Observe that $u$ only appears here in the bound $\alpha + \beta \leq u$.
	Hence $f_{u,v}(\ell,j) - f_{u-1,v}(\ell,j)$ is the same sum except now we sum over the $\alpha,\beta$ with $\alpha + \beta = u$.
	Moreover, $\mu(p^{\alpha}) = 0$ if $\alpha \geq 2$, so we have $\alpha \in \{ 0, 1 \}$. Thus we get
	\begin{align*}
		f_{u,v}(\ell,j) - f_{u-1,v}(\ell,j)
		& = p^{tu} (1-p^{-t}) \sum_{\gamma \leq v} p^{(k+t) \gamma} ( \sigma_k(p^{v + u - \gamma}) - p^k \sigma_k(p^{v+u-1-\gamma})) \\
		%& = p^{tu} (1-p^{-t}) \sum_{\gamma \leq v} p^{(k+t) \gamma} \\
		& = p^{tu} (1-p^{-t}) \sigma_{k+t}(p^{v})  \\
		& = g_{u,v}(\ell,j) - g_{u-1,v}(\ell,j).
	\end{align*}
	Thus the claim follows for all $u \geq 0$.
\end{proof}

\begin{prop} \label{prop:identity2}
	For all $t,r,k \geq 1$ and $\ell,j\geq 1$ one has
	\[
	\sum_{\substack{a,b,\delta\geq 1\\ ab\delta | \ell}}
	\sum_{c | j}
	a^{t+r+k} b^t c^{t+k} \delta^{t+r+k}
	\mu(a) 
	\sigma_k\left(\frac{jb}{c}\right)
	P_r(\delta)\,
	P_t(a^2 b\delta c) 
	= \ell^t \sigma_{t+k}(j).
	\]
\end{prop}
\begin{proof}
	This reduces to Proposition~\ref{prop:identity1} 
	by setting $x= a \delta$ and using
	$P_t(a^2 b \delta c) = P_t(x^2 b c)$ (which follows since $P_t(n)$ only depends on the prime factors of $n$)
	and $\sum_{x=a \delta} \mu(a) P_r(\delta) = \mu(x) x^{-r}$ 
	(which follows by observing that both sides are multiplicative so checking it for $x=p^u$ gives the claim).
\end{proof}

\subsection{Proof of Theorem~\ref{thm:modular identity}}
%We first conside the case $h=3$. 
By definition of the lift we have
\begin{equation} \label{03fsdf}
	\begin{aligned}
		\Lift_{2(h+r)}( D_{\tau}^r(G_{2h-2}) E_{5/2}^{(1)}) = &  
		\delta_{r0} \left( \frac{1}{4} \frac{B_{2h} B_{2h-2}}{2h \cdot (2h-2)}
		-\frac{B_{2h-2}}{2(2h-2)} \sum_{d \geq 1} q^d \sum_{\ell|d} \ell^{2h-1} c^{(1)}\left( \frac{d^2}{4 \ell^2}, \frac{d}{2 \ell} \right) \right) \\
		& +  \sum_{d \geq 1} q^d \sum_{\ell|d} \sum_{j \geq 1} \ell^{2r+2h-1} c^{(1)}\left( \frac{d^2}{4 \ell^2} - j, \frac{d}{2 \ell} \right) \sigma_{2h-3}(j) j^r
	\end{aligned}
\end{equation}
The first term on the right is exactly the first term on the right in the definition of $F_{h,r}$ since the constants match and by Lemma~\ref{lemma:c1 square coefficients} we have
\begin{align*}
	\frac{-1}{10} \sum_{\ell|d} \ell^{2h-1} c^{(1)}\left( \frac{d^2}{4 \ell^2}, \frac{d}{2 \ell} \right)
	& =
	\sum_{\ell|d} \ell^{2h-1} \sum_{a| \frac{d}{\ell}} \mu(a) a \sigma_3(d/a \ell) \\
	& = 
	\sum_{\delta | d} \sigma_3(d/\delta) \delta^{2h-1} \sum_{a|\delta} \frac{\mu(a)}{a^{2h-2}} \\
	& = 
	\sum_{\delta | d} \sigma_3(d/\delta) \delta^{2h-1}  P_{2h-2}(\delta).
\end{align*}

We hence have to show that the second term in \eqref{03fsdf} matches the second term in the definition of $F_{h,r}$. 
For that we insert Theorem~\ref{thm:coeff of e_{5/2}delta} into the latter term and use Propositions~\ref{prop:identity1} and \ref{prop:identity2}.

Concretely, let us start with the case $h=3$. Then we obtain
\begin{align*}
	& \sum_{d \geq 1} q^d \sum_{\delta \geq 1} 
	\sum_{k | d} k^{3+2r} \sum_{i \geq 1} c^{(\delta)}\left( \frac{d^2}{4 \delta k^2} - i, \frac{d}{2 k \delta} \right) \sigma_1(i) i^r \delta^{5+r} \prod_{p|\delta} (1-p^{-4})  \\
	& = 
	\sum_{d \geq 1} q^d \sum_{\delta \geq 1} \sum_{k|d} \sum_{i \geq 1} \sum_{ \substack{a \\ a | d/k \\ a^2 | \delta}} \sum_{\substack{b \\ b | d/ak \\ b| \delta /a^2 \\ b| i}} c^{(1)}\left( \frac{d^2}{4 a^2 b^2 k^2} - \frac{\delta i}{a^2 b^2}, \frac{d}{2 a b k} \right) k^{2r+3} \sigma_1(i) i^r \delta^{3+r} P_2(\delta) \mu(a) b^2
\end{align*}
where we used that $N = -i$. Setting
\[ \ell = a b k, \quad \delta = a^2 b c, \quad i = \frac{j b}{c} \]
we see that $b|i$ is equivalent to $c|j$, so the above becomes
\begin{align*}
	=\sum_{d \geq 1} q^d \sum_{\ell|d} \ell^{3+2r} \sum_{j \geq 1}
	j^r
	c^{(1)}\left( \frac{d^2}{4 \ell^2} - j, \frac{d}{2 \ell} \right)
	\sum_{\substack{a,b,c \geq 1 \\ ab|\ell \\ c | j}} a^3 b^2 c^3 \mu(a) \sigma_1(jb/c) P_2(a^2 b c)
\end{align*}
By the case $k=1,t=2$ of Proposition~\ref{prop:identity1} 
this is exactly the second term in \eqref{03fsdf} as desired. This completes case $h=3$.

In the case $h \geq 4$ setting $\delta = \delta_2 / \delta_1$ we obtain
\begin{align*}
	& 
	\sum_{d \geq 1} q^d \sum_{\substack{\delta_1,\delta_2 \geq 1 \\ \delta_1 | \delta_2 \\ \delta_1 | d \\ }}
	\sum_{k | \frac{d}{\delta_1}} k^{3+2r} 
	\sum_{i \geq 1} c^{(\delta_2/\delta_1)}\left( 
	\frac{d^2}{4 \delta_1 \delta_2 k^2} - i, \frac{d}{2 \delta_2 k} \right) \sigma_1(i) i^r \\
	& \quad \quad \quad \quad \cdot
	\delta_1^{2h-5+r} \delta_2^{2h-1+r}  P_{2h-6}(\delta_1) P_{2h-4}(\delta_2)
	\prod_{p|\delta_2/\delta_1} \frac{1-p^{-4}}{1-p^{-2}} \\
	= & 
	\sum_{d \geq 1} q^d 
	\sum_{\substack{\delta_1,\delta_2 \geq 1 \\ \delta_1 | \delta_2 \\ \delta_1 | d \\ }}
	\sum_{k | \frac{d}{\delta_1}} \sum_{i \geq 1}
	\sum_{\substack{a > 0 \\ a| \frac{d}{\delta_1 k} \\ a^2 | \delta}}
	\sum_{\substack{b > 0 \\ b|i \\ b|\frac{d}{\delta_1 a k} \\ b | \frac{\delta}{a^2}}}
	c^{(1)}\left( \frac{d^2}{4 \delta_1^2 a^2 b^2 k^2} - \frac{\delta i}{a^2 b^2}, \frac{d}{2 a b \delta_1 k} \right) k^{2r+3} \sigma_1(i) i^r \\ & \quad \quad \quad \quad \cdot (\delta_1 \delta_2)^{2h-3+r} 
	P_{2h-6}(\delta_1) P_{2h-4}(\delta_2) \mu(a) b^2
\end{align*}
where we again used that $N = -i$. Setting
\[ \ell = a b \delta_1 k, \quad \delta_2 = a^2 b \delta_1 c, \quad i = \frac{b}{c}j \]
we again get that $b|i$ is equivalent to $c|j$, so the above becomes
\[
\sum_{d \geq 1} q^d \sum_{\ell | d} \ell^{2r+3} \sum_{j \geq 1} j^r 
c^{(1)}\left( \frac{d^2}{4 \ell^2} - j, \frac{d}{2 \ell} \right)
\sum_{\substack{a,b,\delta_1,c \\ a b \delta_1 | \ell \\ c|j}}
\mu(a) \sigma_1\left( \frac{b j}{c} \right) P_{2h-6}(\delta_1) P_{2h-4}(a^2 b \delta_1 c) a^{4h-9} b^{2h-4} c^{2h-3} \delta_1^{4h-9}.
\]
By the case $r=2h-6$, $t=2h-4$, $k=1$ of Proposition~\ref{prop:identity2} 
this is exactly the second term in \eqref{03fsdf} as desired. This completes the case $h \geq 4$. \qed

		\section{An operator identity for symmetric functions}
		\label{appendix:an operator identity}
Let $\Lambda$ be the ring of symmetric functions in $h$ variables $x_1,\ldots,x_h$.
Let $e_0,\ldots, e_h$ be the elementary symmetric polynomials defined by the identity
\[ E(t) = \sum_{i=0}^{h} e_i t^{i} = \prod_{i=1}^{h} (1 + x_i t). \]
In particular $e_0 = 1$. We also set $e_k = 0$ if $k \notin \{ 0, \ldots, h \}$.
Let $p_k$ be the power sum symmetric polynomials defined by $p_k = \sum_i x_i^k$. 
In particular, $p_0=h$.
We have the identity
\[ E(t) = \exp\left( \sum_{r \geq 1} \frac{(-1)^{r-1}}{r} z^r p_r \right). \]
We will also require below the generating series
\[
H(t) = \frac{t E'(t)}{E(t)} = t \frac{d}{d t} \log E(t) = \sum_{r \geq 1} (-1)^{r-1} t^r p_r 
K(t) = h - H(t) = \sum_{r \geq 0} (-1)^r p_r t^r.
\]

We have $\Lambda = \BC[e_1,\ldots, e_h] = \BC[p_1,\ldots,p_h]$.

Set $\partial_r = \frac{d}{d p_r}$ for $r \in \{ 1, \ldots, h \}$ and $\partial_r = 0$ for $r>h$.

Consider the following differential operator acting on $\Lambda$:
\[
\CL_{g,h}
= 
2 \sum_{r \geq 1} r (r-2g-1) p_{r-1} \partial_r - 4 \sum_{r \geq 2} \sum_{k=1}^{r-1} r p_k p_{r-1-k} \partial_r
- 2 \sum_{k, \ell \geq 1} k \ell p_{k + \ell - 1} \partial_k \partial_{\ell}
\]

The goal here is to prove that:

\begin{prop}
	\begin{equation}\label{eq:Lgh in e}
		\begin{aligned}
			\CL_{g,h}
			={}&-2\sum_{i,j=1}^h
			\left(\sum_{a=0}^{\min(i,j)-1}(i+j-1-2a)e_a e_{i+j-1-a}\right)
			\frac{d}{d e_i} \frac{d}{d e_j} \\
			&\quad -4\sum_{i=1}^h (g-i+1)(h-i+1)e_{i-1} \frac{d}{de_i} .
		\end{aligned}
	\end{equation}
\end{prop}
\begin{proof}
	Since 
	\[ \frac{d}{d p_r} E(t) = \frac{(-1)^{r-1}t^r}{r} E(t) \]
	we have %$\frac{d}{d p_r} e_m = e_{m-r} \frac{(-1)^{r-1}}{r}$. Thus
	\[
	\frac{d}{d p_r} = \frac{(-1)^{r-1}}{r}\sum_{m} e_{m-r} \frac{d}{d e_m}.
	\]
	Inserting we find that
	\begin{align*}
		\CL_{g,h} = & 2 \sum_{r \geq 1} \sum_{m} (r-2g-1) p_{r-1} (-1)^{r-1} e_{m-r} \frac{d}{d e_m} \\
		& -4 \sum_{r \geq 2} \sum_{k=1}^{r-1} \sum_m p_k p_{r-1-k} (-1)^{r-1} e_{m-r} \frac{d}{d e_m} \\
		& - 2 \sum_{k,\ell \geq 1} \sum_{m} (-1)^{k+ \ell} p_{k+\ell - 1} e_{m-k-\ell} \frac{d}{d e_m} \\
		& - 2 \sum_{k,\ell \geq 1} \sum_{m, \tilde{m}} p_{k+\ell - 1} (-1)^{k+\ell} e_{m-k} e_{\tilde{m}-\ell} \frac{d}{d e_m} \frac{d}{d e_{\tilde{m}}}
	\end{align*}
	We hence have
	\[ \CL_{g,h} = \sum_{m} A_m \frac{d}{d e_m} + \sum_{i,j} B_{ij} \frac{d}{d e_i} \frac{d}{d e_j} \]
	where the $A_m, B_{ij}$ are computed as follows by forming generating series. For the $A$-terms:
	\begin{align*}
		\sum_{m} A_m t^m = & 2 t E(t) (-2g K(t) + t K'(t)) + 4 t H(t) K(t) E(t) -2 t^2 H'(t) E(t) \\
		= & 4 t E(t) \left[ -gh + (g+h) H(t) - t H'(t) - H(t)^2 \right] \\
		= & -4 t \left[ gh E(t) - (g+h) t E'(t) + \left( t \frac{d}{d t} \right)^2 E(t) \right] \\
		= & - 4 t \sum_{a=0}^{h} (g-a)(h-a) e_a t^a.
	\end{align*}
	Thus $A_m = -4 (g-m+1)(h-m+1) e_{m-1}$.
	
	For the $B_{ij}$-terms we use:
	\begin{align*}
		\sum_{i,j \geq 0}B_{ij} t^i u^j
		= & -2 E(t) E(u) \sum_{k,\ell \geq 1} (-1)^{k+\ell} p_{k+\ell-1} t^k u^{\ell} \\
		= & -2 E(t) E(u) \frac{t u}{t - u} (H(t) - H(u)) \\
		= & - 2 \frac{tu}{t-u} (t E'(t) E(u) - E(t) u E'(u)) \\
		= & - 2 \frac{tu}{t-u} \sum_{a,b \geq 0} (a-b) e_a e_b t^a u^b \\
		= & - 2 \frac{tu}{t-u} \sum_{a<b} (b-a) e_{a} e_{b} (t^b u^a - t^a u^b) \\
		= & -2 \sum_{a<b} e_{a} e_{b} (t^b u^{a+1} + t^{b-1} u^{a+2} + \cdots + t^{a+1} u^b)
	\end{align*}
	If we extract now the $t^i u^j$ coefficient, then each summand contributes if we have
	$a+1 \geq j$ and $a+1 \leq i$, so $a \leq \min(i,j) - 1$. Since $i+j = a+b+1$ we get
	$b-a = i+j-2a-1$ in this case. So
	\[ B_{ij} = -2 \sum_{a=0}^{\min(i,j)-1} (i+j-2a-1) e_a e_{i+j-a-1}.  \qedhere \]
\end{proof}

\bibliography{references}
\bibliographystyle{amsplain}
%\printbibliography

%\bibliographystyle{siamplain}
%\bibliographystyle{siam}

\footnotesize

\end{document}